%% file: main.tex
\documentclass{article}

\input{Preambule}

\title{A Fibration Theorem and Chas-Sullivan product for Morse-Novikov Homology with differential graded coefficients}
\author{Robin Riegel \footnote{email : r.riegel@math.unistra.fr} \\
IRMA, Université de Strasbourg}
\date{}

\begin{document}
\selectlanguage{english}
\renewcommand{\contentsname}{Table of Contents}
\renewcommand{\proofname}{Proof}
\renewcommand{\refname}{References}
\renewcommand{\figurename}{Figure}
\setlength{\parindent}{0pt}

\maketitle

\begin{center}
    \textbf{Abstract}
\end{center}

Given an oriented, closed and connected manifold $X$ and a nonzero cohomology $u\in H^1(X,\R)$, we extend the constructions of Morse Homology with differential graded coefficients of \cite{BDHO23} and of the Chas-Sullivan product described on this model in \cite{Rie24} to Morse-Novikov Homology with differential graded coefficients. More precisely, we will construct a Morse-Novikov complex with differential graded coefficients and prove that, given a fibration $\fibration$ such that $\pi^*u = 0 \in H^1(E,\R)$ and the fiber is locally path-connected, there exists a Morse-Novikov model for a "Novikov completion" $H_*(E,u)$ of the singular homology of $E.$ Moreover, we will prove that if the fibration is endowed with a morphism of fibrations on its fibers, then there exists a Chas-Sullivan-like product on $H_*(E,u)$ that can be described within this model.

\tableofcontents

\section{Introduction}

\subsection{Context and goals}

The goal of this paper is to extend to the Morse-Novikov setting the construction of Morse Homology with differential graded coefficients (also referred to as enriched Morse Homology)  and extend two theorems stated in this setting. One states that enriched Morse Homology recover the homology of total spaces of fibrations \cite[Theorem 7.2.1]{BDHO23} and the other states that there exists a Morse model for string topology operations on total spaces of fibrations such as the Chas-Sullivan product \cite[Theorem 7.1]{Rie24}

\subsubsection{Morse-Novikov theory}

The extension of Morse theory for 1-forms $\alpha \in \Omega^1(X)$ such that $[\alpha] \in H^1(X,\Z)$ has been first formalized by Novikov in \cite{Nov81} by considering a $\Z$-covering $\hat{X} \overset{p}{\to} X$ such that $p^*\alpha = dS \in \Omega^1(\hat{X})$ in order to reduce the problem to the study of a Morse function $f = \exp{(iS)} : X \to S^1$.
This approach has then been generalized by Sikorav in \cite{Sik87} for a real cohomology class $[\alpha] = u \in H^1(X,\R)$ by considering the Morse-Novikov complex both as a \textbf{projective limit} of singular chains on the universal cover $\tilde{X} \overset{\tilde{\pi}}{\to} X$ relative to sublevel sets $\tilde{f} \leq c$ where $d\tilde{f} = \tilde{\pi}^*\alpha$, and as a \textbf{cell complex} with local coefficients in the \textbf{Novikov ring}
$$\Lambda_u = \left\{ \sum_{g \in \pi_1(X)} n_g g, \textup{ for every } c\in \R, \textup{ the number of } g \in \pi_1(X) \textup{ such that } n_g \neq 0  \textup{ and } u(g) >c \textup{ is finite}\right\}.$$

To the best of the author's knowledge the definition and study of the moduli space of trajectories associated with a pseudo-gradient adapted to a 1-form $\alpha \in \Omega^1(X)$ date back to Latour \cite{Lat94}, where the author defines the Morse-Novikov complex as a \textbf{Morse complex} with local coefficients in $\Lambda_u$ in order to characterize cohomology classes $u \in H^1(X,\R)$ that can be represented by a 1-form with no critical points. 

These moduli spaces of trajectories do not behave as nicely as their Morse counterpart. One of the main differences of interest for this paper is that, given $x,y \in \Crit(\alpha)$, the space of trajectories $\traji{\alpha}{x,y}$ between $x$ and $y$ may contain arbitrarily long flow lines with respect to the length induced by $\alpha$ and therefore cannot be compactified.

\begin{defi}\label{def : length by alpha}
    The \textbf{length of a continuous path} $\gamma : I \to X$ is given by $$L(\gamma) = \int_{\gamma} \alpha.$$
\end{defi}

To address this problem, Latour considers trajectories of bounded length $\traji{\alpha}{x,y,A},$ where $A >0$.\\

\subsubsection{Morse homology with differential graded coefficients}

From an oriented, closed and connected manifold $X^n$ and a preferred basepoint $\star \in X$, the authors of \cite{BDHO23} construct its \textbf{Morse complex with coefficients in a differential graded (DG) module} $(\F_*, \partial_{\F})$ over the differential graded algebra (DGA) $R_* = C_*(\Omega X)$ of the cubical complex of the space of loops in $X$ based at $\star$. This is also referred to as an \textbf{enriched Morse complex}. Given a Morse function $f : X \to \R$ and $\xi$, a pseudo-gradient adapted to $f$, this complex writes 

$$C_*(X,\F_*) = \F_* \otimes \Z \Crit(f),$$

where the differential is twisted by a family of chains called, in reference to the seminal paper \cite{BC07}, \textbf{Barraud-Cornea twisting cocycle}

$$\left\{ m_{x,y} \in C_{|x|-|y|-1}(\Omega X), x,y \in \Crit(f) \right\}.$$

The cocycle $(m_{x,y})_{x,y}$ is obtained by evaluating in $\Omega X$ representatives of the fundamental classes of the moduli spaces of Morse trajectories $\trajb{x,y}$. The twisted differential is given by

$$\partial (\alpha \otimes x) = \partial_{\F} \alpha \otimes x + (-1)^{|\alpha|} \sum_y \alpha \cdot m_{x,y} \otimes y.$$

In \cite{Rie24}, we make use of their construction and, in particular, of one of their main results, the Fibration Theorem, in order to build a Morse theoretic Chas-Sullivan product $$\CS : H_*(E)^{\otimes 2} \to H_*(E),$$ for any (Hurewicz) fibration $\fibration$ endowed with a morphism of fibrations $m : E \ftimes{\pi}{\pi} E \to E$.\\

\textbf{Fibration Theorem} (\cite[Theorem 7.2.1]{BDHO23})
    \emph{Let $\fibration$ be a fibration with model fiber $F=\pi^{-1}(\star)$ and equip $C_*(F)$ with the $C_*(\Omega X)$-module structure induced by a transitive lifting function $\Phi : E \ftimes{\pi}{\ev_0} \mathcal{P}X \to E$ associated with this fibration (see \cite[Section 7.1]{BDHO23}). Then, there exists a quasi-isomorphism denoted $$\Psi_E : C_*(X, C_*(F)) \to C_*(E).$$}

\textbf{DG Chas-Sullivan product }(\cite[Theorem 7.1]{Rie24})

\emph{Let $\fibration$ be a fibration endowed with a morphism of fibrations $m : E \ftimes{\pi}{\pi} E \to E$. Let $F = \pi^{-1}(\star)$ and endow $\F = C_*(F)$ with a DG right $C_*(\Omega X)$-module structure induced by a transitive lifting function. There exists a degree $-n$ product}
$$\CS : H_*(X, \F)^{\otimes 2} \to H_*(X, \F)$$

\emph{that is associative if $m_*$ is associative in homology, graded commutative if $m_*$ is commutative in homology, admits a neutral element if there exists a unit section $s : X \to E$ and this product satisfies the Functoriality and Spectral sequence properties.}

\emph{Moreover, this product corresponds in homology, via the Fibration Theorem, to the product $\mu_* : H_i(E) \otimes H_j(E) \to H_{i+j-n}(E)$ defined by Gruher-Salvatore in \cite{GS07}. In particular, if the fibration is the loop-loop fibration $\Omega X \hookrightarrow \mathcal{L}X \to X$, then $\CS$ corresponds to the Chas-Sullivan product.}\\

The goal of this paper is to extend the construction of Morse Homology with differential graded coefficients and the previous two results to the case where the underlying topological data does not come from trajectories associated with a Morse function, but from trajectories associated with a Morse 1-form. \textbf{We thus consider Morse-Novikov theory instead of Morse theory.} \\

\subsection{Morse-Novikov homology with differential graded coefficients} 

In our construction, we prevent the trajectories $\lambda$ between two critical points $x,y \in \Crit(\alpha)$ from being arbitrarily long by prescribing the endpoints of their lift $\tilde{\lambda} \in \traj{\tilde{x}, g\tilde{y}}$ in the universal cover $\tilde{X}$ by fixing $g \in \pi_1(X).$

Given $x,y \in \Crit(\alpha)$ and $g \in \pi_1(X)$, the space $\traji{g}{x,y}$ of such trajectories can be compactified and the evaluation into loops of a representative of its fundamental class will give rise to a chain 
$$m^g_{x,y} \in C_{|x|-|y|-1}(\Omega^gX)$$
in the space of based loops of homotopy class $g \in \pi_1(X).$ Therefore, the natural DGA $R_*$ in which to consider the twisting cocycle
$$m_{x,y} = \sum_g m^g_{x,y} \in R_{|x|-|y|-1}$$
is the "Novikov completion" of $C_*(\Omega X)$,
$$R_*=C_*(\Omega X, u) = \left\{ \sum_g n_{g} \gamma^g, \ \gamma^g \in C_*(\Omega^gX) \textup{ and } \forall c \in \R, \ \#\{g \in \pi_1(X), n_g  \neq 0 \textup{ and } u(g)>c\} < \infty \right\}.$$

We denote by $H_*(\Omega X,u)$ its homology. Notice that $H_0(\Omega X,u) = \Lambda_u.$  We define an enriched Morse-Novikov complex with coefficients in a right $C_*(\Omega X,u)$-module $(\F^u_*, \partial_{\F^u})$ by

$$C_*(X,\F^u) = \F^u_* \otimes_{\Z} \Z \Crit(\alpha),$$ endowed with the differential 

$$\partial (\sigma \otimes x) = \partial_{\F^u} \sigma \otimes x + (-1)^{|\alpha|} \sum_y \sigma \cdot m_{x,y} \otimes y.$$

We will denote the homology of this complex  $H_*(X,\F^u_*)$, or $H_*(X,\F^u_*,u)$ if we need to be explicit about the chosen class $u\in H^1(X,\R)$.

We prove invariance properties with regard to all the choices needed to build such a complex. By choosing a well-suited data set, we prove the following:

\begin{proposition}[Corollary \ref{cor : DG Morse Novikov homotopy equivalent to DG Morse}]\label{Proposition A}
    For all right $C_*(\Omega X,u)$-module $\F^u$, the Morse-Novikov complex $C_*(X,\F^u,u)$ with coefficients in $\F^u$ is chain homotopy equivalent to the Morse complex $C_*(X,\F^u)$ where $\F^u$ is here considered a $C_*(\Omega X)$-module.
\end{proposition}

This generalizes the proof of Latour \cite[Théorème 2.18]{Lat94} that Morse-Novikov homology is isomorphic to the homology of $X$ with local coefficients in $\Lambda_u$. For this reason, we call our result the  \textbf{Latour Trick} (see Section \ref{subsection : Latour Trick}).

\subsection{Morse-Novikov Fibration theorem}

As in enriched Morse Homology, the main source of such coefficients are given by (Hurewicz) fibrations.
Given a fibration $F \hookrightarrow E \overset{\pi}{\to} X$, there exists a \textbf{transitive lifting function} $\Phi : E \ftimes{\pi}{\ev_0} \mathcal{P}X \to E$ that lifts all paths in $X$ and respects concatenation in the sense that lifting a concatenation of two paths is the same as lifting them one after the other : $\Phi(\Phi(e,\gamma),\delta) = \Phi(e,\gamma \#\delta).$
The restriction $\Phi : F \times \Omega X \to F$, also referred to as the \textbf{holonomy map} associated with $\Phi$, induces a $C_*(\Omega X)$-module structure on the cubical complex $C_*(F)$.

\subsubsection{Primitive and Novikov completion}

Suppose now that $F \hookrightarrow\fibration$ is such that $\pi^*u = 0 \in H^1(E,\R)$ and $F$ is locally path-connected. We prove in Proposition \ref{prop : existence of primitive} that in this case there exists a notion of "\textbf{primitive}" of $\pi^*\alpha$ for any $\alpha \in \Omega^1(X)$ representing $u$, even if $E$ is not a manifold. 

More precisely, we will prove in Lemma \ref{lemme : pi*u=0 iff factors through integration cover} that a fibration $\fibration$ satisfies $\pi^*u = 0$ if and only if the fibration factors through the integration cover $\hat{X}_u$ of $u \in H^1(X,\R).$

$$\xymatrix{
E \ar[dr]_{\pi} \ar[r]^{\varphi_u}& \hat{X}_u \ar[d]^{\hat{\pi}} \\
& X.
}$$

Since $\hat{\pi}^*u = 0 \in H^1(\hat{X}_u,\R)$, for every $\alpha \in u$, there exists a primitive $\hat{f} : \hat{X}_u \to X$ of $\hat{\pi}^*\alpha$.

\begin{defi}[Definition \ref{def : primitive of pi alpha}]
    We will refer to the map $f = \hat{f} \circ \varphi_u : E \to \R$, as a \textbf{primitive of} $\boldsymbol{\pi^*\alpha}.$  
\end{defi}

\begin{proposition}[Proposition \ref{prop : existence of primitive}]\label{proposition B}
    Assume that $F$ is locally path-connected. Let $\star_E \in \pi^{-1}(\star)$ and $\alpha \in u$. If $\pi^*u = 0$, all primitives $f : E \to \R$ of $\pi^*\alpha$ are continuous and differ only by a constant. Moreover, any primitive satisfies 
    \begin{equation*}
        \forall (e,\gamma) \in E \ftimes{\pi}{ev_0} \mathcal{P} X, \ f(\Phi(e,\gamma)) - f(e) = \int_{\gamma} \alpha.
    \end{equation*}
\end{proposition}

In this case, the $C_*(\Omega X)$-module structure on $C_*(F)$ extends to a $C_*(\Omega X,u)$-module structure on the projective limit $C_*(F,u) = \limproj C_*(F, f{\lvert_F} \leq c)$ of the cubical complexes of $F$ relative to the sublevel sets of $f\lvert_F$.

\begin{theorem}[\textbf{Morse-Novikov Fibration Theorem}, Theorem \ref{thm: Fibration Theorem Morse Novikov}]\label{theorem C}
    Let $u \in H^1(X,\R)$ and let
    
    $F \hookrightarrow \fibration$ be a fibration such that $\pi^*u = 0$ and $F$ is locally path-connected. Let $\alpha \in u$ and let $f : E \to \R$ be a primitive of $\pi^*\alpha$. Denote $$C_*(E,u) = \limproj C_*\left(E, f \leq c \right).$$ There exists a quasi-isomorphism 

    $$\Psi_E : C_*(X,  C_*(F,u)) \to C_*(E, u).$$

    Moreover, these complexes depend, up to chain homotopy equivalence, only on $u$ i.e. neither on $\alpha$ nor on the chosen primitive $f : E \to \R$ of $\pi^*\alpha$ nor on the choice of transitive lifting function $\Phi$ for the fibration $\fibration$. 
\end{theorem}

\subsubsection{Examples}

This homology theory can be used to study total spaces of fibrations over the integration cover $\hat{X}_u$ of a cohomology class $u \in H^1(X,\R)$.

Let us give a more concrete example. Let $u \in H^1(X,\R)$.
Consider the fibration $E = \mathcal{P}_{\star \to X} X \overset{\ev_1}{\to} X$ of the Moore based paths $\mathcal{P}_{\star \to X} X = \{ \gamma : [0,a] \to X  \textup{ continuous, } \gamma(0) = \star \}$ to $X$ given by the evaluation at the endpoint. The model fiber of this fibration $F = \ev_1^{-1}(\star) = \Omega X$ is locally path-connected and consider the transitive lifting function $$\deffct{\Phi}{\mathcal{P}_{\star \to X} X \ftimes{\ev_1}{\ev_0} \mathcal{P}X}{\mathcal{P}_{\star \to X} X}{(\gamma,\tau)}{\gamma\tau}.$$
Since $\mathcal{P}_{\star \to X}X$ is contractible, in particular $H^1(\mathcal{P}_{\star \to X}X,\R) = \{0\}$ and $\ev_1^*u = 0$. Let $f : E \to \R$ be a primitive of $\ev_1^*\alpha$.
Theorem \ref{theorem C} states that $$H_*(X,C_*(\Omega X,u)) \simeq H_*(\mathcal{P}_{\star \to X}X,u),$$

the homology of the cubical complex $$C_k(\mathcal{P}_{\star \to X}X,u) = \left\{ \sum_{i \geq 0} n_i \omega_i, \ \omega_i \in C_k\left(\mathcal{P}_{\star \to X} X\right) \textup{ and } \forall c \in \R, \ \#\{i \geq 0, \ n_i \neq 0, f\circ\omega_i > c\} < \infty \right\}.$$

This homology is expected by the author to be useful to study the homology of the space $$\mathcal{M}_u = \left\{ \gamma : [0, + \infty) \to X, \ \gamma(0) = \star \textup{ and } \lim_{t \to + \infty}\int_{0}^t \gamma^* \alpha = - \infty  \right\}$$ that is defined and used by Latour in \cite{Lat94} to give a necessary and sufficient condition for a cohomology class $u \in H^1(X,\R)$ to have a representative $\alpha \in \Omega^1(X)$ with no critical points.

\subsection{Chas-Sullivan product in enriched Morse-Novikov Homology}

Let $F \hookrightarrow \fibration$ be a fibration endowed with a morphism of fibrations $m : E \ftimes{\pi}{\pi} E \to E$.
In enriched Morse Homology, the degree $-n$ product $\CS : H_*(X, C_*(F))^{\otimes 2} \to H_*(X, C_*(F))$ is defined, up to sign, as the composition of three maps

$$\underbrace{H_*(X,C_*(F))^{\otimes 2}}_{\simeq H_*(E)^{\otimes 2}} \overset{K}{\to} \underbrace{H_*(X^2,C_*(F^2))}_{\simeq H_*(E \times E)} \overset{\Delta_!}{\to} \underbrace{H_{*-n}(X,\Delta^*C_*(F^2))}_{\simeq H_{*-n}(E \ftimes{\pi}{\pi} E)} \overset{\tilde{m}}{\to} H_{*-n}(X,C_*(F)),$$

where: \begin{itemize}
    \item $K : C_*(X,\Xi,C_*(F))^{\otimes 2} \to C_*(X^2,\Xi_{X^2},C_*(F^2))$ is the DG cross-product defined in \cite[Section 6.3]{Rie24},
    \item $\Delta_! : C_*(X^2,\Xi_{X^2},C_*(F^2)) \to C_{*-n}(X,\Xi,\Delta^*C_*(F^2))$ is the shriek map of the diagonal $\Delta : X \to X^2$ defined in \cite[Sections 9 and 10]{BDHO23}.
    \item $\tilde{m} : C_{*-n}(X,\Xi,\Delta^*C_*(F^2)) \to C_{*-n}(X,\Xi,C_*(F))$ is the map induced by the morphism of fibrations $m : E \ftimes{\pi}{\pi} E \to E$ defined in \cite[Section 5]{Rie24}.\\
\end{itemize} 

Let $u = [\alpha] \in H^1(X,\R)$ and assume that $\pi^*u =0$ and $F$ is locally path-connected. Let $\Xi$ be a set of DG Morse data.
We will build Morse-Novikov extensions of each of these maps with the same properties:\\

\begin{itemize}
    \item[$\bullet$] \textbf{The DG cross-product} $K : C_*(X,\Xi,C_*(F,u))^{\otimes 2} \to C_*(X^2,\Xi_{X^2},C_*(F^2,u))$ in Section \ref{subsection : The cross product K}. We have the following extension of \cite[Theorem 6.14]{Rie24}:\\
    
\begin{theorem}[\textbf{DG Morse-Novikov Künneth formula} Theorem \ref{thm : Morse-Novikov Künneth formula}] \label{theorem D}
Let $(Y, \star_Y)$ be an oriented, closed and connected manifold and let $v \in H^1(Y,\R)$. Let $\Xi_Y$ be a set of DG Morse data. Let $G \hookrightarrow E_Y \overset{\pi_Y}{\to} Y$ be a fibration such that $\pi_Y^*v = 0$.
There exists a \textbf{Künneth twisting cocycle} on the product $X \times Y$ $$\left\{m^{K,(g,h)}_{z,w} \in C_{|z|-|w|-1}(\Omega^g X \times \Omega^h Y)\right\}$$ such that:

\begin{itemize}
    \item[1)] The twisting cocycle $m^K$ computes the same homology $H_*(X \times Y, C_*(F \times G, u + v))$ as the Barraud-Cornea cocycle.
    \item[2)] The map $$\deffct{K}{C_*(X, \Xi, C_*(F,u)) \otimes C_*(Y, \Xi_Y, C_*(G,v))}{C_*(X \times Y, m^K_{z,w}, C_*(F \times G,u+v))}{(\alpha \otimes x) \otimes (\beta \otimes y)}{(-1)^{|\beta||x|}(\alpha , \beta) \otimes (x,y)}$$ is a quasi-isomorphism of complexes.\\
\end{itemize}

\end{theorem}

\item[$\bullet$] \textbf{The shriek map} $\Delta_! : C_*(X^2,\Xi_{X^2},C_*(F^2,u)) \to C_{*-n}(X,\Xi,\Delta^*C_*(F^2,u))$ in Section \ref{Section : Functoriality in the DG Morse-Novikov theory}. 
We have the following extension of \cite[Theorem 8.1.1]{BDHO23}:\\

\begin{proposition}[\textbf{DG Morse-Novikov Functoriality}, Proposition \ref{prop : Functoriality in DG Morse-Novikov}]\label{Proposition E}
Let $X^n,Y^k,Z^{\ell}$ be oriented, closed and connected manifolds.
Let $u \in H^1(Y,\R)$ and $\G_*$ be a DG right $C_*(\Omega Y, u)$-module.
A continuous map $\varphi: X \to Y$ induces in homology a direct map
$$\varphi_*: H_*(X,\varphi^*\G, \varphi^*u) \to H_*(Y, \G,u)$$
and a shriek map
$$\varphi_!: H_*(Y,\G,u) \to H_{*+n-k}(X,\varphi^*\G, \varphi^*u)$$ with the following properties:

\begin{enumerate}
    \item \textup{(IDENTITY)} We have $\Id_* = \Id_!= \Id: H_*(Y,\G,u) \to H_*(Y,\G,u)$.

    \item \textup{(COMPOSITION)} Given continuous maps $X \overset{\varphi}{\to} Y \overset{\psi}{\to} Z$, $u_Z \in H^1(Z,\R)$ and $\F$ a DG right $C_*(\Omega Z, u_Z)$-module, 

    $$(\psi \circ \varphi)_* = \psi_* \circ \varphi_*: H_*(X,\psi^*\varphi^*\F, \psi^*\varphi^*u_Z) \to H_*(Z, \F, u_Z)$$ and
    $$(\psi \circ \varphi)_! = \varphi_! \circ \psi_!: H_*(Z,\F,u_Z) \to H_{*+n-\ell}(X,\psi^*\varphi^*\F, \psi^*\varphi^*u_Z).$$
    \item \textup{(HOMOTOPY)} Two homotopic maps induce the same direct and shriek maps.

    \item \textup{(SPECTRAL SEQUENCE)} The direct and shriek maps are limit of morphisms between the spectral sequences associated with the corresponding enriched complexes, given at the second pages by 
    $$\varphi_{p,q,*} : H_p(X,\varphi^*H_q(\G)) \to H_p(Y,H_q(\G))$$ 
    and 
    $$\varphi_{p,q,!} : H_p(Y,H_q(\G)) \to H_{p+n-k}(X,\varphi^*H_q(\G))$$ the usual direct and shriek maps in homology with local coefficients.\\
\end{enumerate}
\end{proposition}
    
    \item[$\bullet$] The morphism $\tilde{m} : C_{*-n}(X,\Xi,\Delta^*C_*(F^2, u)) \to C_{*-n}(X,\Xi,C_*(F,u))$ induced by the \textbf{morphism of fibrations} $m : E \ftimes{\pi}{\pi} E \to E$ in Section  \ref{subsection : Morphism of fibrations}. We will prove the following extension of \cite[Theorem 5.8]{Rie24}:\\
    
    \begin{theorem}[\textbf{Morphism of fibrations in DG Morse-Novikov theory}, Theorem \ref{thm : morphisme induit commute avec iso Morse-Novikov}] \label{theorem F}
        Let $F_1 \hookrightarrow E_1 \overset{\pi_1}{\to} X$ and $F_2 \hookrightarrow E_2 \overset{\pi_2}{\to} X$ be fibrations such that $\pi_2^*u =0$ and $F_2$ is locally path-connected. Let $\varphi : E_1 \to E_2$ be a morphism of fibrations.

    \begin{itemize}
        \item[i)] There exists a sequence of maps $\left\{\varphi_{n+1} : [0,1]^n \times F_1 \times \Omega X^{n-1} \times \mathcal{P}_{\star \to X} X \to E_2\right\}$ called a \textbf{coherent homotopy} for $\varphi$ that induces a morphism of complexes $$\tilde{\varphi} : C_*(X,\Xi, C_*(F_1,f_1)) \to C_*(X,\Xi, C_*(F_2,f_2)).$$

        \item[ii)] A coherent homotopy induces a chain homotopy $v : C_*(X,\Xi,C_*(F_1,f_1)) \to C_{*+1}(E_2,f_2)$ between $\Psi_2 \circ \tilde{\varphi}$ and $\varphi_* \circ \Psi_1$, where the morphisms $\Psi_1$ and $\Psi_2$ are the quasi-isomorphisms given by the Fibration Theorem. In other words, the following diagram commutes up to chain homotopy
    \end{itemize}
    \[
    \xymatrix{
    C_*(X, C_*(F_1,u)) \ar[r]^-{\Tilde{\varphi}} \ar[d]^{\Psi_1} & C_*(X, C_*(F_2,u)) \ar[d]_{\Psi_2} \\
    C_*(E_1,u) \ar[r]_{\varphi_*} & C_*(E_2,u).
    }
    \]
\end{theorem}

\end{itemize} 

\begin{rem}
    Theorem \ref{theorem D}, Proposition \ref{Proposition E} and Theorem \ref{theorem F} are extensions of \cite[Theorem 6.14]{Rie24},  \cite[Theorem 8.1.1]{BDHO23} and \cite[Theorem 5.8]{Rie24}, respectively, since they correspond to the case $u=0$.
\end{rem}
These three maps define a Chas-Sullivan-type product in enriched Morse-Novikov theory.
$$\CS = (-1)^{n(n-j)} \tilde{m} \circ \Delta_! \circ K : H_{i}(X,C_*(F,u)) \otimes H_{j}(X,C_*(F,u))\to H_{i+j-n}(X,C_*(F,u)).$$

\begin{theorem}\label{theorem G}[\textbf{DG Chas-Sullivan product in Morse-Novikov theory}, Theorem \ref{thm : Chas-Sullivan product in Morse-Novikov}]
    The product $\CS$ has the properties of Associativity, Commutativity, Functoriality and Spectral Sequence.
\end{theorem}

In general, this product does not have a neutral element. Indeed, such an element would have degree $n$ but for the universal cover $\pi_1(X) \hookrightarrow \tilde{X} \to X$, we would obtain a neutral element in $H_n(X,\Lambda_u) = H_n(X ; u)$. However, it is known that for the classical Morse-Novikov theory, $H_n(X;u)=0$ if $u \neq 0$ (see \cite[Lemme 4.1]{Lat94}). 
In \cite{Rie24}, we proved that if $\fibration$ admits a section $s : X \to E$ such that $m(s(\star),e) = m(e,s(\star)) = e$ for all $e \in F$, then $\CS : H_*(X,C_*(F))^{\otimes 2} \to H_*(X,C_*(F))$ admits a neutral element. However, if $\pi^*u = 0$ and there exists a section $ s : X \to E$ of this fibration, it would mean that $u=0$. This hypothesis is therefore not helpful in this case. \\

\subsection{Structure of this paper}
In Section \ref{section : Construction of the enriched Morse-Novikov complex}, we build the enriched Morse-Novikov complex and study in Section \ref{section : Spectral sequence and DG Morse-Novikov toolset} the associated spectral sequence as well as, in Section \ref{section : Invariance of DG Morse-Novikov homology}, prove Proposition \ref{Proposition A}. We then prove Proposition \ref{proposition B} the Morse-Novikov Fibration Theorem \ref{theorem C} in Section \ref{section : A Fibration Theorem for DG Morse-Novikov theory}. We define the direct and shriek maps in enriched Morse-Novikov theory in Section \ref{Section : Functoriality in the DG Morse-Novikov theory} and prove Proposition \ref{Proposition E}. We prove Theorems \ref{theorem D}, \ref{theorem F} and \ref{theorem G} in Section \ref{section : Chas-Sullivan type product in enriched Morse-Novikov homology}. Since we need some algebraic results about the relations between projective limits and homology, we added the Appendix \ref{Appendix : Homology and projective limit, algebraic properties}.

In this article, we generalize Morse homology with differential graded coefficients to the Morse-Novikov setting and generalize the Fibration Theorem \cite[Theorem 7.2.1]{BDHO23} and the DG Chas-Sullivan product \cite[Theorem 7.1]{Rie24} to this new setting.

\section{Construction of the enriched Morse-Novikov complex}\label{section : Construction of the enriched Morse-Novikov complex}

Let $(X,\star)$ be a pointed, oriented, closed, and connected manifold, and $u \in H^1(X,\R)$ a fixed cohomology class.
Following the algebraic setup \cite[Section 4.1]{BDHO23} for the DGA $R_* = C_*(\Omega X,u)$, we will define in this section a Morse-Novikov complex with coefficients in a right $C_*(\Omega X,u)$-module (Definition \ref{def : DG Morse-Novikov complex}).

\subsection{Morse-Novikov trajectories}

Morse theory studies the gradient trajectories associated with a Morse function $h: X \to \R$. More generally, Morse-Novikov theory explores the behavior of trajectories associated with a closed 1-form.

\begin{defi}

\begin{itemize}
    \item[$\bullet$] \textbf{A Morse 1-form} $\alpha \in \Omega^1(X)$ is a closed 1-form that, locally around its zeroes, is the differential of a Morse function $\alpha = dh$. We use the notation $x \in \Crit(\alpha)$ if $\alpha_x = 0$.
    
    \item[$\bullet$] The \textbf{index $|x|$ of a critical point} $x \in \Crit(\alpha)$ is the Morse index of the critical point $x \in \Crit(h).$

    \item[$\bullet$] A vector field $\xi: X \to TX$ is \textbf{a (negative) pseudo-gradient adapted to $\alpha$} if 
    \begin{itemize}
        \item For all $x \in X$, $\alpha_x(\xi_x) \leq 0$ with equality if and only if $x \in \Crit(\alpha).$
        \item For all $x \in \Crit(\alpha) = \Crit(h)$, there exists a Morse chart around $x$ such that $\xi$ coincides with $-\textup{grad}(h)$ in that Morse chart.
    \end{itemize}
    The same arguments as in Morse theory ensure that such a pseudo-gradient always exists (see \cite[2.1.c]{AD14}).

    \item[$\bullet$] Denote by $\phi^s$ the flow of a pseudo-gradient $\xi$ adapted to a Morse 1-form $\alpha$. The \textbf{orbit} of $a \in X$ is denoted by $\{\phi^s(a)\}_{s \in \R}.$
\end{itemize}
\end{defi}

 Fix a Morse 1-form $\alpha$ and an adapted pseudo-gradient $\xi$.
The (un)stable manifolds and spaces of trajectories can be defined in the same way as in Morse theory.

\begin{defi}
    Let $x \in \Crit(\alpha)$. The \textbf{unstable manifold of $x$} is the $|x|$-dimensional manifold $$W^u(x) = \left\{ a \in X, \lim_{s \to - \infty}\phi^s(a) = x \right\}.$$  Its \textbf{stable manifold} is the $n-|x|$-dimensional manifold $$W^s(x) = \left\{ a \in X, \lim_{s \to + \infty}\phi^s(a) = x \right\}.$$  
\end{defi}

The Kupka-Smale Theorem ensures that, for generic $\xi,$ the unstable manifold $W^u(x)$ and the stable manifold $W^s(y)$ are transverse for each pair of critical points $x,y \in \Crit(\alpha)$. We will always assume that $\xi$ satisfies this property. We will also say that the pair $(\alpha,\xi)$ is Morse-Smale.

\begin{defi}
    Let $x,y \in \Crit(\alpha)$, $x\neq y$. The \textbf{space of parametrized trajectories} $$\mathcal{M}(x,y) = W^u(x) \cap W^s(y)$$ is endowed with the $\R$-action given by $$s \cdot a = \phi^s(a).$$ \textbf{The space of (unparametrized) trajectories} between $x$ and $y$ is defined by  $$\traj{x,y}:= \faktor{\mathcal{M}(x,y)}{\R}$$ and it is a manifold of dimension $|x|-|y|-1$.
\end{defi}

In this section, we list the definitions and properties given by François Latour in \cite{Lat94} on which our construction will rely.

In Morse theory, any orbit $\gamma : I \to X$ of a Morse-Smale pair $(h,\xi)$ connects two critical points $x$ and $y$ of $h$ and has finite length $h(x) - h(y)$. In Morse-Novikov theory, there is nothing to prevent orbits to have infinite length (with respect to the length induced by $\alpha$, see Definition \ref{def : length by alpha}).

\begin{lemme}\cite[2.4]{Lat94}
    An orbit $\{\phi^s(a)\}$ has finite length if and only if it connects two critical points, \emph{i.e.} there exist $x,y \in \Crit(\alpha)$ such that $$\lim_{s \to -\infty}\phi^s(a) = x \textup{ and } \lim_{s \to -\infty}\phi^s(a) = y.$$
\end{lemme}

We therefore place a particular importance on bounded orbits. In particular, considering only orbits of bounded length enables us to compactify the (un)stable manifolds and the space of trajectories.

\begin{defi}
    Let $x,y \in \Crit(\alpha)$.
    
    $\bullet$ Define $$\trajb{x,y} = \traj{x,y} \cup \bigcup_{\substack{k \geq 1 \\ z_1, \dots, z_{k}}} \traj{x,z_1} \times \traj{z_1,z_2} \times \dots \times \traj{z_{k},y}$$ the \textbf{spaces of broken trajectories.}\\

    $\bullet$ Define $$\wb{u}{x} = W^u(x) \cup \bigcup_{y \in \Crit(\alpha)}\trajb{x,y} \times W^u(y) \textup{ and } \wb{s}{x} = W^s(x) \cup \bigcup_{y \in \Crit(\alpha)} W^s(y) \times \trajb{y,x}.$$
\end{defi}

We can extend the notion of length to $\lambda =(\lambda_1, \dots, \lambda_{k}) \in \trajb{x,y}$ by $L(\lambda) = \sum_{i=1}^k L(\lambda_i).$

\begin{prop}\cite[Proposition 2.6]{Lat94}
For each pair $x,y\in \Crit(\alpha)$ and $A>0$, the space $$\trajb{x,y,A}:= \left\{\lambda \in \trajb{x,y}, \ L(\lambda) \leq A\right\}$$ of broken trajectories of length bounded by $A$ is a compact manifold of dimension $|x|-|y|-1$ with boundaries and corners.
\end{prop}

 If $a \in W^u(x)$, define $D_x(a) = \int_{-\infty}^0 \gamma^*\alpha = L(\gamma)$ where $\gamma: (-\infty,0] \to X$, $\gamma(s) = \phi^s(a).$ This function extends to $(\lambda,a) \in \trajb{x,y} \times W^u(y) \subset \wb{u}{x}$ by $D_x(\lambda,a) = L(\lambda) + D_y(a).$

 \begin{prop}\cite[Proposition 2.8]{Lat94}
     For each critical point $x \in \Crit(\alpha)$ and $A>0$, the space $$\wb{u}{x,A} = \left\{ a \in \wb{u}{x}, \ D_x(a) \leq A\right\} $$ is a compact manifold of dimension $|x|$ with boundaries and corners.
 \end{prop}
 
 \begin{rem}
     We refer to \cite[2.5]{Lat94} or \cite[3.2.a and 4.9.b]{AD14} for the definitions of the topologies on $\trajb{x,y}$ and $\wb{u}{x}.$
 \end{rem}

\subsection{DG Morse-Novikov data}

In order to define the enriched Morse-Novikov complex, we start by defining a \textbf{twisting cocycle} adapted to the Morse-Novikov setting. For that, we adapt the set of DG Morse data (\cite[Section 5.2]{BDHO23}) to the spaces of trajectories of a closed 1-form $\alpha \in u$.
Fix $\alpha \in u$ a closed 1-form and $\xi$ a pseudo-gradient adapted to $\alpha$.

In our construction, instead of considering trajectories of bounded length, we have chosen to consider trajectories whose lifts to the universal cover $\tilde{\pi} : \tilde{X} \to X$ have prescribed endpoints. This point of view describes spaces of trajectories with similar behaviour to that of spaces of trajectories of bounded length but I deemed more natural to index the spaces of trajectories between two critical points with $\pi_1(X)$ instead of $\R$ since loops are already of special interest in the enriched Morse setting.

Let $\Y$ a tree rooted in $\star$ for which the set of external vertices is $\Crit(\alpha)$. Choose a lift $\tilde{\star} \in \tilde{X}$ of $\star$. For every $x \in \Crit(\alpha)$, denote $\tilde{x} \in \tilde{X}$ the lift of $x$ obtained by lifting $\Y$ in $\tilde{X}$ at $\tilde{\star}$ (see Figure \ref{fig: lift of tree}).

\begin{figure}[h!]
    \centering
    \includegraphics[scale=0.25]{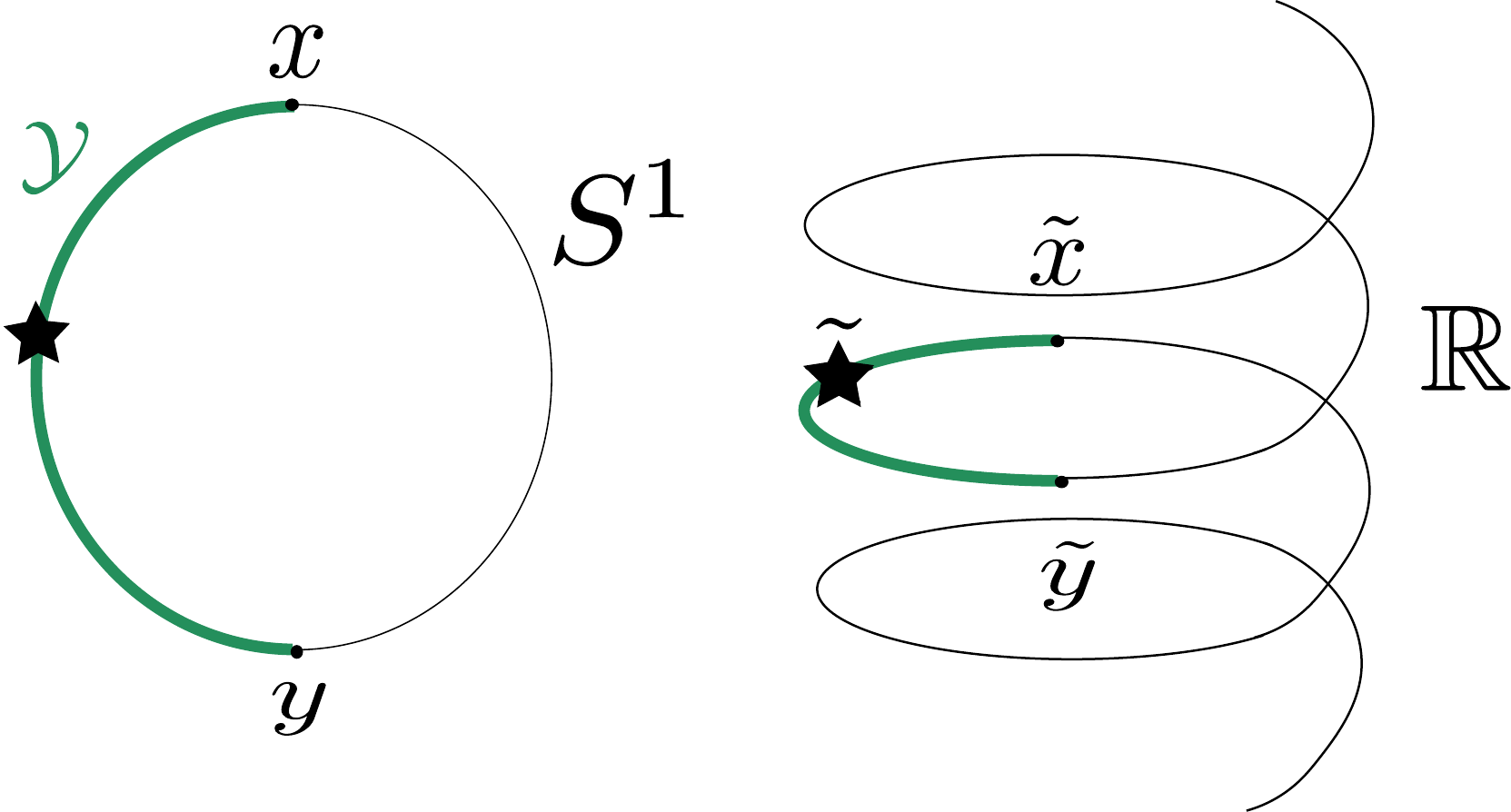}
    \caption{Choice of lifts}
    \label{fig: lift of tree}
\end{figure}

\subsubsection{Spaces of broken trajectories with prescribed endpoints in the universal cover}

\begin{defi}\label{defi: Lg(x,y)}
    Let $x,y \in \Crit(\alpha)$ be a pair of critical points. Let $\lambda \in \traj{x,y}$.
     Let $g(\lambda) \in \pi_1(X)$ such that the lift $\tilde{\lambda}$ of $\lambda$ starting at $\tilde{x}$ in $\tilde{X}$ ends at $g(\lambda)\tilde{y}$.
     For a given $g \in \pi_1(X),$ we denote 
     $$\traji{g}{x,y} = \left\{ \lambda \in \traj{x,y}, \ g(\lambda) = g \right\}$$
     and $$\trajbi{g}{x,y} = \traji{g}{x,y} \cup \bigcup_{\substack{k \geq 1, z_1, \dots,z_k \\ g_1\cdot \ \dots \  \cdot g_{k+1} = g}} \traji{g_1}{x,z_1} \times \dots \times \traji{g_{k+1}}{z_k,y}$$
     the \textbf{spaces of (broken) trajectories with prescribed endpoints in the universal cover}.

We endow it with the topology and with the orientation induced by $\trajb{x,y}$, orientation given by a choice $o$ of orientation for the Latour cells $\left(\wb{u}{x}\right)_x$.
\end{defi}

If $\alpha = df$ is exact, we can parametrize each trajectory $\lambda \in \traj{x,y}$ by the values of the Morse function $f : X \to \R$ using that $f$ is strictly decreasing along the trajectory and each level set $\{f = t\}$ intersects transversely the image of $\lambda$ exactly once for all $t \in [f(y), f(x)]$. For a generic closed 1-form $\alpha$, we use a primitive on the universal cover to parametrize the trajectories. Let $\tilde{h}: \tilde{X} \to \R$ be a primitive (in the usual sense) of $\tilde{\pi}^*\alpha \in \tilde{\pi}^*u \in H^1(\tilde{X},\R)  = \{0\}$. It satisfies the equation \begin{equation}\label{eq : primitive in universal cover}
   \forall g\in \pi_1(X), \ \forall \tilde{a} \in \tilde{X}, \  \tilde{h}(g\tilde{a}) = u(g) + \tilde{h}(a).
\end{equation}

\begin{lemme}\label{lemme: parametrization map}
 For each $x,y \in \Crit(\alpha)$, there exists a continuous \textbf{parametrization map} $$\Gamma_{x,y}: \trajb{x,y} \to \mathcal{P}_{x \to y}(X)$$ such that:
  $\Gamma_{x,y}(\lambda,\lambda') = \Gamma_{x,z}(\lambda)\#\Gamma_{z,y}(\lambda')$ for all $(\lambda, \lambda') \in \trajb{x,z} \times \trajb{z,y} \subset \trajb{x,y}.$
\end{lemme}

\begin{proof}

Let $\Tilde{\Gamma}_{x,y}: \traj{x,y} \to \mathcal{P}_{\tilde{x} \to \tilde{\pi}^{-1}(y)} \tilde{X}$ be the parametrization map that parametrizes by the values of $\Tilde{h}$  by
$$\Tilde{\Gamma}_{x,y}(\lambda)(t) = \Tilde{\lambda} \cap \Tilde{h}^{-1}(\Tilde{h}(\Tilde{x})-t) \ \text{for all} \ t \in [0, \Tilde{h}(\Tilde{x}) - \Tilde{h}(g(\lambda)\Tilde{y})],$$

where $\lambda \in \trajb{x,y}$, $\Tilde{\lambda}$ is its lift in $\Tilde{X}$ starting at  $\tilde{x}$.
The map $\Gamma = \tilde{\pi} \circ \tilde{\Gamma}$ clearly satisfies the wanted relation.

\end{proof}

\begin{rem}\label{rem : param by an exact 1-form}
    If $\alpha = dh$ is exact, then this parametrization map is the parametrization by the values of $h$  used in enriched Morse theory. Indeed, in this case, $h \circ \tilde{\pi} : \tilde{X} \to \R$ is a primitive of $\tilde{\pi}^*\alpha.$
\end{rem}

We now prove that the spaces $\trajbi{g}{x,y}$ have the appropriate structure in order to define a twisting cocycle, \emph{i.e.} they are compact manifolds with boundary and corners of the expected dimension.

\begin{prop}
    For each pair $x,y \in \Crit(\alpha)$ and $g \in \pi_1(X)$, the space $\trajbi{g}{x,y}$ is a compact manifold of dimension $|x|-|y|-1$ with boundary and corners, and $$\partial \trajbi{g}{x,y} = \bigcup_{\substack{k \geq 1, z_1, \dots,z_k \\ g_1\cdot \ \dots \  \cdot g_{k+1} = g}} \traji{g_1}{x,z_1} \times \dots \times \traji{g_{k+1}}{z_k,y}=\bigcup_{\substack{|y| \leq |z| \leq |x| \\ g'.g''=g}} \trajbi{g'}{x,z} \times \trajbi{g''}{z,y}.$$ 
\end{prop}

\begin{proof}
    We prove that $\trajbi{g}{x,y}$ is either empty or an union of connected components of the compact manifold $$\trajb{x,y,A} = \{\lambda \in \trajb{x,y}, L(\lambda)\leq A\}$$ for some $A>0$. Indeed, if $g \in \pi_1(X)$, $\lambda, \lambda' \in \trajbi{g}{x,y}$, then 

    \begin{align*}
        L(\lambda) &= \int_{\Gamma_{x,y}(\lambda)} \alpha = \int_{\tilde{\Gamma}_{x,y}(\lambda)} \tilde{\pi}^*\alpha\\
        &= \tilde{h}(\tilde{x}) - \tilde{h}(g\tilde{y}) \\
        &= \tilde{h}(\tilde{x}) - \tilde{h}(\tilde{y}) - u(g)\\
        &=   L(\lambda').
    \end{align*}

Therefore, there exists $A(g)\geq 0$ such that $\trajbi{g}{x,y} \subset \trajb{x,y,A(g)}.$ We now prove that a path in $\trajb{x,y,A(g)}$ that starts in $\trajbi{g}{x,y}$ has to end in $\trajbi{g}{x,y}$.

If $\gamma: [0,1] \to  \trajb{x,y,A(g)}$ is a path between $\lambda$ and $\lambda'$ in $\trajb{x,y,A(g)}$, then $\Gamma_{x,y} \circ \gamma$ induces a fixed ends homotopy between $\Gamma_{x,y}(\lambda)$ and $\Gamma_{x,y}(\lambda')$. The homotopy lifting property of universal covers concludes that $g(\lambda) = g(\lambda')$. 
\end{proof}

 The compactness issue being resolved by considering these spaces of broken trajectories with prescribed endpoints in the universal cover, we can now achieve the next step towards the construction of a twisting cocycle in this context: the introduction of a \textbf{representing chain system}.

\begin{prop}\label{prop: representing chain system}
    There exists a family $\left\{ s^g_{x,y} \in C_{|x|-|y|-1}(\trajbi{g}{x,y}), \ x,y \in \Crit(\alpha), \ g \in \pi_1(X) \right\}$ such that for each $ x,y \in \Crit(\alpha), \ g \in \pi_1(X)$:

    \begin{enumerate}
        \item $s^g_{x,y}$ is a cycle relative to the boundary and represents the fundamental class of $\trajbi{g}{x,y}$.\\
        \item The following relation holds\begin{equation}\label{eq: système representatif}
            \partial s^g_{x,y} = \displaystyle \sum_{\substack{z\\g'.g'' = g}}
            (-1)^{|x|-|z|} s_{x,z}^{g'} \times s_{z,y}^{g''}.
        \end{equation} 
        The multiplication is given here by the inclusion $\trajbi{g'}{x,z} \times \trajbi{g''}{z,y} \subset \trajbi{g}{x,y}$ induced by the concatenation.\\
    \end{enumerate}
\end{prop}

\begin{proof}

We follow the idea of \cite[Proposition 5.2.6]{BDHO23}. We proceed by induction on $a=|x|-|y|-1$. If $a=0$ and $g \in \pi_1(X)$, then $\trajbi{g}{x,y}$ is an orientable and compact 0-dimensional manifold. It is therefore a finite set of points each associated with a sign $\trajbi{g}{x,y}=\{\pm c_1, \dots, \pm c_k\}$ and we set  $s_{x,y}^g = \sum_i\pm c_i \in C_0(\trajbi{g}{x,y})$.

Let $a \geq 1$, $x,y \in \Crit(\alpha)$ such that $|x|-|y|-1 = a$ and $g \in \pi_1(X)$. We assume that $s^{g'}_{c,d} \in C_{|c|-|d|-1}(\trajbi{g'}{c,d})$ satisfying 1. and 2. has been constructed for all $|c| - |d| -1 \leq a-1$ and $g' \in \pi_1(X)$.

For $z \in \Crit(\alpha)$ such that $|x|>|z|>|y| $ and $g',g'' \in \pi_1(X)$ such that $g'\cdot g''=g$, from the inclusions $\trajbi{g'}{x,z} \times \trajbi{g''}{z,y} \subset \partial \trajbi{g}{x,y} \subset \trajbi{g}{x,y}$, it follows that there exist two fundamental classes $$[\partial \trajbi{g}{x,y}] \in H_{|x|-|y| - 2} \left( \partial \trajbi{g}{x,y}\right)$$ and $$[\trajbi{g'}{x,z} \times \trajbi{g''}{z,y}] \in H_{|x|-|y| - 2}\left(\trajbi{g'}{x,z} \times \trajbi{g''}{z,y}, \partial \left(\trajbi{g'}{x,z} \times \trajbi{g''}{z,y}\right)\right).$$
We gave $\trajbi{g'}{x,z} \times \trajbi{g''}{z,y}$ the orientation induced by the exterior normal of  $\trajbi{g}{x,y}$.

The chain $(-1)^{|x| - |z|} s^{g'}_{x,z} \times s^{g''}_{z,y} \in C_{|x|-|y| - 2}(\trajbi{g'}{x,z} \times \trajbi{g''}{z,y})$ is a cycle relative to the boundary because, by the induction hypothesis, $s^{g'}_{x,z}$ and $s^{g''}_{z,y}$ are cycles.
 Since we oriented $\trajbi{g'}{x,z} \times \trajbi{g''}{z,y}$ by the exterior normal, \cite[Proposition 5.2.4]{BDHO23} implies that $(-1)^{|x| - |z|} s^{g'}_{x,z} \times s^{g''}_{z,y}$ represents the fundamental class of $[\trajbi{g'}{x,z} \times \trajbi{g''}{z,y}].$

Moreover, $\displaystyle \sum_{\substack{z \\ g'.g''=g}} (-1)^{|x| - |z|} s^{g'}_{x,z} \times s^{g''}_{z,y} \in C_{|x|-|y| - 2}(\partial \trajbi{g}{x,y})$ is a cycle. Indeed, 

\begin{align*}
\partial \sum_{\substack{z \\ g'.g''=g}} (-1)^{|x| - |z|} s^{g'}_{x,z} \times s^{g''}_{z,y} &= \sum_{\substack{|x| > |z_1| >|z| >|y| \\ g'.g''=g\\ g'_1.g'_2=g'}} (-1)^{2|x| - |z| - |z_1|} s^{g'_1}_{x,z_1} \times s^{g'_2}_{z_1,z} \times s^{g''}_{z,y}\\
& + (-1)^{|x|-|z|-1} \sum_{\substack{|x| > |z| >|z_2| >|y| \\ g'.g''=g\\ g''_1.g''_2=g''}} (-1)^{|x| - |z_2|} s^{g'}_{x,z} \times s^{g''_1}_{z,z_2} \times s^{g''_2}_{z_2,y}\\
&= \sum_{\substack{|x| > |c| >|d| >|y| \\ g_1.g_2.g_3=g}} (-1)^{|c| + |d|} s^{g_1}_{x,c} \times s^{g_2}_{c,d} \times s^{g_3}_{d,y} + (-1)^{|c| + |d| +1} s^{g_1}_{x,c} \times s^{g_2}_{c,d} \times s^{g_3}_{d,y} \\
&= 0.
\end{align*}

Since, $\partial \trajbi{g}{x,y} = \displaystyle \bigsqcup_{\substack{|x|>|z|>|y|\\ g'.g''=g}} \trajbi{g'}{x,z} \times \trajbi{g''}{z,y}$ and $(-1)^{|x| - |z|} s^{g'}_{x,z} \times s^{g''}_{z,y}$ represents the fundamental class $[\trajbi{g'}{x,z} \times \trajbi{g''}{z,y}]$ for all $|x|>|z|>|y|$, then $\displaystyle \sum_{\substack{z\\g'.g''=g}} (-1)^{|x| - |z|} s^{g'}_{x,z} \times s^{g''}_{z,y}$ represents the fundamental class $[\partial \trajbi{g}{x,y}]$.

Take $s'^{g}_{x,y}$ a representative of $[\trajbi{g}{x,y}]$. Then, $\partial s'^g_{x,y}$ represents $[\partial \trajbi{g}{x,y}]$ and therefore, there exists $p^g_{x,y} \in C_{|x|-|y| - 1}(\trajbi{g}{x,y})$ such that 

\[\partial s'^g_{x,y} - \displaystyle \sum_{\substack{z\\g'.g''=g}} (-1)^{|x| - |z|} s^{g'}_{x,z} \times s^{g''}_{z,y} = \partial p^g_{x,y}.\]

We conclude the proof by defining $s^g_{x,y} = s'^g_{x,y} - p^g_{x,y}.$
\end{proof}

\begin{defi}
    We call a family $\left\{s^g_{x,y} \in C_{|x|-|y|-1}(\trajbi{g}{x,y}), \ x,y\in \Crit(\alpha), g \in \pi_1(X)\right\}$ as in Proposition \ref{prop: representing chain system}, a \textbf{representing chain system for the moduli spaces of trajectories} of $(\alpha,\xi)$.
\end{defi}

To build a twisting cocycle and define the DG Morse-Novikov complex, we now have to understand how to evaluate a representing chain system into $\Omega X.$ We define the evaluation maps in the same way as in the case of DG Morse homology \cite[Lemma 5.2.10] {BDHO23} using the parametrization maps $\Gamma_{x,y}$ defined in Lemma \ref{lemme: parametrization map}.

\begin{lemme}\label{lemme: def q}

There exists a family of continuous maps $q_{x,y}: \trajb{x,y} \to \Omega X$ such that:

\begin{enumerate}
    \item For all $\lambda \in \trajb{x,y}$, the homotopy class $g \in \pi_1(X)$ of $q_{x,y}(\lambda)$ is $g(\lambda)$.
    \item If $(\lambda, \lambda') \in \trajb{x,z} \times \trajb{z,y}$, then $q_{x,y}(\lambda, \lambda') = q_{x,z}(\lambda) \# q_{z,y}(\lambda')$.
\end{enumerate}
\end{lemme}

\begin{proof}

We consider the projection $p: (X,\star) \to (X/\Y,\star)$, that collapses all critical points of $\alpha$ to the base point. Since $\Y$ is contractible, there exists a homotopy inverse $\theta: (X/\Y,[\Y]) \to (X,\star)$ of $p$. We may and will assume that $\theta([\Y]) = \star$. From now on, we will also denote $\star = [\Y] \in X/\Y$, as this will serve as the base point in $X/\Y$.
We then define $$q_{x,y} = \theta \circ p \circ \Gamma_{x,y}.$$

We now consider $\ell = \gamma_x \# \Gamma_{x,y}(\lambda) \# \gamma_y^{-1}$, where $\gamma_x$ and $\gamma_y$ are, respectively, the branches of $\Y$ connecting $\star$ to $x$ and $\star$ to $y$. Since $p(\ell) = p(\Gamma_{x,y}(\lambda))$, we find that $q_{x,y}(\lambda) = \theta \circ p (\ell)$, and therefore $q_{x,y}(\lambda)$ is homotopic to $\ell$.\\
If we denote $\Tilde{\gamma_x}$ and $\Tilde{\gamma_y}$ as the lifts of $\gamma_x$ and $\gamma_y$, then by the definition of the choices of the lifts $\Tilde{x}$ and $\Tilde{y}$, $\Tilde{\gamma_y}^{-1} \# \Tilde{\gamma_x}$ connects $\Tilde{y}$ to $\Tilde{x}$. Thus, we reparametrize $\ell = \gamma_y^{-1} \# \gamma_x \# \Gamma_{x,y}(\lambda)$ as a loop based at $y$, that lifts to $\Tilde{X}$ as a path connecting $\Tilde{y}$ to $g(\lambda)\Tilde{y}$, given by $\Tilde{\ell} = \Tilde{\gamma_y}^{-1} \# \Tilde{\gamma_x} \# \Tilde{\Gamma}_{x,y}(\lambda)$. Hence, $g(\lambda) = [\ell] = [q_{x,y}(\lambda)]$, and condition 1. is satisfied.
Condition 2. is satisfied by construction.

\end{proof}

\begin{defi}\label{defi: twisting cocycle}
Define the family $$\left\{m^g_{x,y} = q_{x,y,*}(s^g_{x,y}) \in C_{|x|-|y|-1}(\Omega^g X), \ x, y \in \Crit(\alpha), \ g \in \pi_1(X)\right\},$$ where we denoted $\Omega^g X = \{\gamma \in \Omega X, \ [\gamma]=g\}$. This family satisfies

\begin{equation}\partial m^g_{x,y} = \sum_{\substack{z \\ g'.g''=g}} (-1)^{|x|-|z|} m^{g'}_{x,z} \cdot m^{g''}_{z,y}, \end{equation}

where the multiplication is defined by the Pontryagin product. 
\end{defi}

\begin{rem}\label{rem: projection cocycle tordant g}

Point 1. of the previous lemma and the construction of $s_{x,y}^g$ in the case where $|x| = |y| +1$ show that, in this case, $\traji{g}{x,y} = \{\lambda_1, \dots, \lambda_k\}$ is a finite set and the projection of $m_{x,y}^g \in C_0(\Omega X)$ onto $H_0(\Omega X) = \mathbb{Z}[\pi_1(X)]$ is 

\begin{equation}\label{eq: twisting cocycle in H_0}
    \hat{m}_{x,y}^g =  \sum_i\epsilon(\lambda_i) g,
\end{equation} 

where, for $i \in \{1, \dots, k\}$, $\epsilon(\lambda_i) \in \{\pm 1\}$ is the sign associated with the orientation of $\lambda_i$.

\end{rem}

\begin{defi}
    Denote $\Xi = (\alpha,\xi,s^g_{x,y},o,\Y, \theta)$ the data needed in order to build the family $$\left\{m^g_{x,y} \in C_{|x|-|y|-1}(\Omega^g X), \ x, y \in \Crit(\alpha), \ g \in \pi_1(X)\right\}.$$ We will call $\Xi$ a \textbf{set of DG Morse-Novikov data}.
\end{defi}

\subsection{Enriched Morse-Novikov Homology}

We begin by describing the natural DGA $R_*$ which is the recipient of the \textbf{twisting cocycle}. In enriched Morse Homology, $R_* = C_*(\Omega X)$ endowed with the Pontryagin product. In the present setting, we will consider a "Novikov completion" $R_*=C_*(\Omega X,u)$ of $C_*(\Omega X)$.

\subsubsection{Novikov completion of the DGA of chains on \texorpdfstring{$\Omega X$}{Omega X}}

We see the cohomology class $u \in H^1(X,\R)$ as a morphism $\deffct{u}{\pi_1(X)}{\R}{\gamma}{\int_{\gamma}\alpha.}$

The Morse-Novikov complex $C_*(\alpha,\xi,u) = \Lambda_u \otimes \Z\Crit(\alpha)$ has coefficients in $$\Lambda_u = \left\{ \sum_{i\geq 0} n_ig_i, \ \forall c \in \R, \ \#\{g_i, \ u(g_i)\geq c\} < \infty \right\}.$$  This is a completion of $\Z[\pi_1(X)]$ for the ultra metric norm $|\ell|_u = \textup{e}^{-v(\ell)}$ where $v(n_1g_1 + \dots + n_kg_k) = \min u(g_i)$.

We will use a similar notion of completion for $C_*(\Omega X)$ to define the enriched Morse-Novikov complex. First remark that $u: \pi_1(X) \to \R$ can evaluate cubic chains of loops. Indeed, for any $g\in \pi_1(X)$, $\Omega^g X$ is a connected component of $\Omega X$.  We define $$\deffct{u}{C^0([0,1]^k,\Omega X)}{\R}{\omega}{u([\omega(0)])}$$ and extend it to $C_k(\Omega X)$ by $u(n_1\omega_1 + \dots + n_l\omega_l) = \displaystyle \max_{n_i \neq 0} u(\omega_i)$.

\begin{defi}\label{def : (Omega X,u)}
    For $k \in \N^*$, define $$C_*(\Omega X, u) = \left\{ \sum_g n_{g} \gamma^g, \ \gamma^g \in C_*(\Omega^gX) \textup{ and } \forall c \in \R, \ \#\{g \in \pi_1(X), n_g  \neq 0 \textup{ and } u(g)>c\} < \infty \right\}.$$ This forms a complex with the differential 
    $$ \hat{\partial}: C_k(\Omega X,u) \to C_{k-1}(\Omega X,u), \   \hat{\partial}\left(\sum n_i \omega_i \right) = \sum n_i \partial \omega_i.$$
\end{defi}

\begin{rem}
    If $u = 0$, then $C_*(\Omega X,u) = C_*(\Omega X).$
\end{rem}

\begin{lemme}
    For $c \in \R$, we denote $\Omega_cX = \{ \gamma \in \Omega X, u(\gamma) \leq c\}$. Then,
   $$ C_k(\Omega X,u) = \limproj C_k(\Omega X,\Omega_cX).$$
\end{lemme}
\begin{proof}
    This is a direct consequence of Lemma \ref{Lemme: limite projective d'un quotient de complexe} where $C^n = C_*(\Omega_n X)$ and of Corollary \ref{cor : limproj with parameter in R}.
\end{proof}

We endow $C_*(\Omega X,u)$ with the DGA structure induced by the Pontryagin product on $C_*(\Omega X)$:

\begin{itemize}
    \item[$\bullet$] The sum 
    $$\sum_i n_i \omega_i + \sum_i m_i \tau_i := \sum_i (n_i \omega_i + m_i \tau_i) \in C(\Omega X,u).$$
    \item[$\bullet$] The multiplication 
    $$\left(\sum_i n_i \omega_i \right) \cdot \left(\sum_j m_i \tau_i \right):= \sum_k \left(\sum_{i+j=k} n_im_j \omega_i \cdot \tau_j\right) \in C(\Omega X,u).$$
\end{itemize}

The multiplication is well-defined since, for any $c\in \R$, $$u(\omega_i \cdot \tau_j) \geq c \Leftrightarrow u(\omega_i) \geq c-u(\tau_j)$$ and therefore $$\#\{ (i,j), \ u(\omega_i \cdot \tau_j) \geq c\} \leq \# \{i, \ u(\omega_i) \geq c-\max_j u(\tau_j)\} < \infty. $$

This is the natural DGA to consider in order to define an analogue of the Barraud-Cornea twisting cocycle in the Morse-Novikov setting.

\subsubsection{Twisting cocycle}

\begin{prop}\label{prop : twisting cocycle MN bien dans la completion}
    Let $\Xi$ be a set of Morse-Novikov data. Given $x,y \in \Crit(\alpha)$, the chain $$m_{x,y} = \sum_{g \in \pi_1(X)} m^g_{x,y} \textup{ belongs to } C_{|x|-|y|-1}(\Omega X,u)$$

    and the family $\{m_{x,y}\}_{x,y \in \Crit(\alpha)}$ satisfies the Maurer-Cartan equation \begin{equation}\label{eq: Maurer-cartan m_x,y }
        \hat{\partial} m_{x,y} = \sum_z (-1)^{|x|-|z|} m_{x,z} \cdot m_{z,y}.
    \end{equation}
\end{prop}

\begin{proof}
    We first remark that, since the evaluation maps $q_{x,y}$ map $ \trajbi{g}{x,y}$ into $ \Omega^g X$ (1. of Lemma \ref{lemme: def q}), it follows that $u(m^g_{x,y}) = u(g)$ for any $x,y \in \Crit(\alpha)$ and $g \in \pi_1(X).$
     Let $c \in \R$ and $x,y \in \Crit(\alpha)$.

    The space
    $$\bigcup_{u(g) \geq c} \trajbi{g}{x,y} = \trajb{x,y, \tilde{h}(\tilde{x}) - \tilde{h}(\tilde{y}) - c},$$ is a compact manifold and therefore admits a finite number of connected components. It follows that \begin{align*}
       & \forall c \in \R,\ \#\{g \in \pi_1(X), \ \trajbi{g}{x,y} \neq \emptyset \textup{ and } u(g) \geq c\} < \infty \\
        &\Leftrightarrow \forall c \in \R, \ \#\{g \in \pi_1(X), \ m^{g}_{x,y} \neq 0 \textup{ and } u(g) \geq c\} < \infty \\
        &\Leftrightarrow m_{x,y} \in C_{|x|-|y|-1}(\Omega X,u).
    \end{align*}

   We now prove that the family $\{m_{x,y}\}$ satisfies the Maurer-Cartan equation.

   \begin{align*}
     \hat{\partial} m_{x,y} & = \sum_{ g \in \pi_1(X)} \partial m^g_{x,y}\\
&=\sum_{\substack{z \\ g \in \pi_1(X) \\ g'.g'' = g}} (-1)^{|x|-|z|} m^{g'}_{x,z}\cdot m^{g''}_{z,y}\\
&= \sum_{z} (-1)^{|x|-|z|} \sum_{g',g'' \in \pi_1(X)}  m^{g'}_{x,z} \cdot m^{g''}_{z,y} \\
&=  \sum_z (-1)^{|x|-|z|} m_{x,z}\cdot m_{z,y}.
 \end{align*}

\end{proof}

\begin{defi}
    We will call such a family $(m_{x,y})$ a \textbf{Barraud-Cornea twisting cocycle} associated with the set of DG Morse-Novikov data $\Xi$. More generally, if $\alpha$ is a Morse 1-form representing $u$, we will call \textbf{twisting cocycle} any family $\left\{ m_{x,y} \in C_{|x|-|y|-1}(\Omega X,u), \ x,y \in \Crit(\alpha)\right\}$ satisfying the Maurer-Cartan equation \eqref{eq: Maurer-cartan m_x,y }.
\end{defi}

\begin{rem}
    Since, in practice, we will always work with the family $$\left\{m^g_{x,y}  \in C_{|x|-|y|-1}(\Omega^g X), \ x, y \in \Crit(\alpha), \ g \in \pi_1(X)\right\},$$ we will, with a slight abuse of language, also call this family a twisting cocycle.
\end{rem}

Now that we defined the notion of twisting cocycle, we will use it to define a twisted complex with differential graded coefficients.

\subsubsection{Morse-Novikov complex with coefficients in a DG module over \texorpdfstring{$\boldsymbol{C_*(\Omega X,u)}$}{hat(C)(Omega X)}}

Given a DG right $C_*(\Omega X,u)$-module $(\F^u_*,\partial)$, we define an enriched Morse-Novikov complex with coefficients in $\F^u_*$.

\begin{defi}\label{def : DG Morse-Novikov complex}
    Let $\alpha$ be a Morse 1-form representing $u$ and $\left\{m_{x,y} \in C_{|x|-|y|-1}(\Omega X,u), \ x,y \in \Crit(\alpha) \right\}$ be a twisting cocycle. We endow $$C_*(X, m_{x,y}, \F^u):= \F^u_* \otimes_{\Z} \Z\Crit(\alpha)$$ with the differential 
    $$\partial(\sigma \otimes x) = \partial \sigma \otimes x + \sum_{y} \sigma \cdot m_{x,y} \otimes y.$$
\end{defi}

\begin{prop}
    The linear map $\partial: C_*(X, m_{x,y}, \F^u) \to C_{*-1}(X, m_{x,y}, \F^u)$ is a differential.
\end{prop}

\begin{proof}
    We check that $\partial^2=0$. We compute

\begin{align*}
    \partial^2 (\sigma \otimes x) &= \partial \left ( \partial \sigma \otimes x + (-1)^{|\sigma|}  \sigma \cdot \sum_{y} m_{x,y}  \otimes y \right)\\
    &= (-1)^{|\sigma|-1}  \sum_{y}  \partial \sigma \cdot m_{x,y} \otimes y + (-1)^{|\sigma|} \sum_{y} \partial (\sigma \cdot m_{x,y}) \otimes y \\
    &\quad - \sum_{y} (-1)^{|x| -|y|} \sum_{ z} (\sigma \cdot m_{x,y}) \cdot m_{y,z} \otimes z\\
    &= - (-1)^{|\sigma|} \sum_{y} \partial \sigma \cdot m_{x,y} \otimes y + (-1)^{|\sigma|} \sum_{ y} \partial \sigma \cdot m_{x,y} \otimes y\\
    & \quad + \sum_{y,z} (-1)^{|x|-|z|}\sigma \cdot m_{x,z}\cdot m_{z,y} \otimes y - \sum_{z,y} (-1)^{|x| -|y|} \sigma \cdot m_{x,y} \cdot m_{y,z} \otimes z\\
    &=0.
\end{align*}

\end{proof}

\begin{defi}
    Let $\F^u_*$ be a DG right $C_*(\Omega X,u)$-module, $\Xi$ be a set of DG Morse-Novikov data and $(m_{x,y})$ the associated Barraud-Cornea twisting cocycle. We denote $$C_*(X,\Xi ,\F^u) = C_*(X,m_{x,y},\F^u).$$
\end{defi}

We denote by $H_*(X,\F^u)$ or $H_*(X,\F^u,u)$ the homology of this complex, without further mention of $\Xi$. This is justified by Theorem \ref{thm : invariance}, where we prove that this homology does not depend on the set of DG Morse-Novikov data $\Xi$, but only on the cohomology class $u \in H^1(X,\R).$

\section{Spectral sequence and enriched Morse-Novikov toolset}\label{section : Spectral sequence and DG Morse-Novikov toolset}

In this section we will prove that, in this setting, we have a spectral sequence and a DG toolset similar to the one in enriched Morse theory \cite[Section 2.3]{BDHO23}.

\subsection{Spectral Sequence}\label{subsection: suite spectrale M-N}

Consider $\F^u$ a DG right-module over $C_*(\Omega X,u)$, a Morse 1-form $\alpha$ representing $u$ and a twisting cocycle $(m_{x,y})$.
The enriched Morse-Novikov complex is naturally filtered by

$$F_p(C_k(X,m_{x,y},\F^u)) = \bigoplus_{\substack{ i+j = k \\ i\leq p}} \F^u_j \otimes_{\Z} \Z\Crit_i(\alpha).$$

The 0-th page of the associated spectral sequence is thus 

$$E^0_{p,q} = \faktor{F_p(C_{p+q}(X,m_{x,y},\F^u))}{F_{p-1}(C_{p+q}(X,m_{x,y},\F^u))} = \F^u_q \otimes_{\Z} \Z\Crit_p(\alpha).$$

The associated differential $d^0: E^0_{p,q} \to E^0_{p,q-1}$ is given by 

\[d^0(\sigma \otimes x) = \partial \sigma \otimes x.\]

Therefore, its first page is given by 

\[E^1_{p,q} = H_q(\F^u_*) \otimes_{\Z} \Z\Crit_p(\alpha) \]

and its differential $d^1: E^1_{p,q} \to E^1_{p-1,q}$ by 

\[d^1(\hat{\sigma} \otimes x) = (-1)^q \sum_{|y| = |x|-1} \hat{\sigma} \cdot \hat{m}_{x,y} \otimes y,\]

where $\hat{m}_{x,y} \in H_0(C_*(\Omega X,u)) =\Lambda_u$ is the projection of the cycle $m_{x,y} \in C_0(C_*(\Omega X,u))$, $\hat{\sigma} \in H_q(\mathcal{F}^u)$ and $H_q(\mathcal{F}^u)$ is endowed with its canonical module structure over $ \Lambda_u$, \emph{i.e.}

$$\begin{matrix}

H_0(C_*(\Omega X,u)) & \to & \End(H_q(\F^u_*))\\
\hat{\gamma} & \mapsto & ( \hat{\sigma} \mapsto \widehat{\sigma. \gamma} ).

\end{matrix}$$

\begin{lemme}
    If $(m_{x,y})$ is a Barraud-Cornea twisting cocycle, then for each $x,y \in \Crit(\alpha)$ with $|x|-|y|=1$, $$\hat{m}_{x,y} = [x,y] \in \Lambda_u,$$ where we denoted $[x,y] = \displaystyle \sum_{\lambda \in \trajb{x,y}} \epsilon(\lambda) g(\lambda)$ the coefficient $\partial_{x,y}$ of the Morse-Novikov complex associated with the Morse-Smale pair $(\alpha,\xi).$ Here $\epsilon(\lambda)$ is the sign given by the orientation of $\trajbi{g(\lambda)}{x,y}$.
\end{lemme}

\begin{proof}
The first point of Lemma \ref{lemme: def q} shows that, if $x,y \in \Crit(\alpha)$ and $|x|-|y|=1$, then
$$\hat{m}_{x,y}^g =  \sum_{\lambda \in \trajbi{g}{x,y}}\epsilon(\lambda) g.$$ 
Therefore, 

$$\hat{m}_{x,y} = \sum_g \hat{m}_{x,y}^g = \sum_{\substack{g, \\ \lambda \in \trajbi{g}{x,y}} } \epsilon(\lambda) g = \sum_{\lambda \in \trajb{x,y}} \epsilon(\lambda) g(\lambda).$$
\end{proof}
 
It follows that, in this case, the differential $d^1$ corresponds, up to a sign, to the differential of the complex $H_q(\mathcal{F}^u_*) \otimes_{\Lambda_u} C(\alpha,\xi)$, and consequently

$$E^2_{p,q} \simeq H_{p}(H_q(\mathcal{F}^u_*) \otimes_{\Lambda_u} C(\alpha,\xi)).$$

\subsection{Enriched Morse-Novikov toolset}

In this section, we adapt the DG Morse toolset \cite[Section 2.3]{BDHO23}. 
The following proposition is a reformulation \cite[Proposition 2.3.3]{BDHO23} where the DGA $R_* = C_*(\Omega X,u).$

\begin{prop}\label{prop: continuation map}

 Let $\alpha_0$ and $\alpha_1$ be Morse 1-forms representing $u$ and $\xi_0$, $\xi_1$ be adapted pseudo-gradients.  Let $\F_*^u$ be a DG right $C_{*}(\Omega X,u)$-module.
Let $\{m^0_{x,z} \in C_{|x|-|z|-1}(\Omega X,u), \ x,z \in \Crit(\alpha_0)\}$ and $\{m^1_{y,w} \in C_{|y|-|w|-1}(\Omega X,u), \ y,w \in \Crit(\alpha_1)\}$ be twisting cocycles.
Assume that, for all $x \in \Crit(\alpha_0)$ and $y \in \Crit(\alpha_1)$ such that $|x| \geq |y|$, there exists $\nu_{x,y} \in C_{|x|-|y|}(\Omega X,u)$ satisfying

\begin{equation}\label{eq: Continuation map}
    \hat{\partial} \nu_{x,y} = \sum_{z \in \Crit(\alpha_0)} m_{x,z}^{0} \cdot \nu_{z,y} - \sum_{w \in \Crit(\alpha_1)} (-1)^{|x| - |w|} \nu_{x,w} \cdot m^{1}_{w,y}.
\end{equation}

Then, the map $$\deffct{\Psi}{C_*(X,m^0_{x,y},\F^u)}{C_*(X,m^1_{x,y},\F^u)}{\sigma \otimes x}{\displaystyle \sum_{y \in \Crit(\alpha_1)} \sigma \cdot \nu_{x,y} \otimes y,}$$

is a morphism of complexes.
\end{prop}

\begin{flushright}
    $\blacksquare$
\end{flushright}

\begin{rem}\label{rem: continuation map g}
    Let $x,z \in \Crit(\alpha_0)$ and $y,w \in \Crit(\alpha_1)$. Write $$m^0_{x,z} = \sum_{g \in \pi_1(X)} m^{0,g}_{x,z},$$ where $m^{0,g}_{x,z} \in C_{|x|-|z|-1}(\Omega^g X)$ and $$m^1_{y,w} = \sum_{g\in \pi_1(X)} m^{1,g}_{y,w},$$ where $m^{1,g}_{y,w} \in C_{|y|-|w|-1}(\Omega^g X)$.
    If $\left\{ \nu^g_{x,y} \in C_{|x|-|y|}(\Omega^gX)\ x \in \Crit(\alpha_0), \ y \in \Crit(\alpha_1) \textup{ and }  g \in \pi_1(X)\right\}$ is a family of chains such that
    
    \begin{equation}\label{eq: continuation map g}
   \forall x \in \Crit(\alpha_0), \forall y \in \Crit(\alpha_1), \forall g \in \pi_1(X), \  \partial \nu^{g}_{x,y} = \sum_{\substack{z \in \Crit(\alpha_0)\\ g'\cdot g''=g}} m_{x,z}^{0,g'} \cdot \nu^{g''}_{z,y} - \sum_{\substack{w \in \Crit(\alpha_1)\\g' \cdot g''=g}} (-1)^{|x| - |w|} \nu^{g'}_{x,w} \cdot m^{1,g''}_{w,y}
\end{equation}

and $\nu_{x,y} := \sum_g \nu^g_{x,y} \in C_*(\Omega X,u)$ for all $x \in \Crit(\alpha_0), \ y \in \Crit(\alpha_1)$. Then, $(\nu_{x,y})_{x \in \Crit(\alpha_0), \ y \in \Crit(\alpha_1)}$ satisfies \eqref{eq: Continuation map} and the map $$\deffct{\psi}{C_*(X,m^0_{x,y},\F^u)}{C_*(X,m^1_{x,y},\F^u)}{\sigma \otimes x}{\displaystyle \sum_{y \in \Crit(\alpha_1)} \sigma \cdot \nu_{x,y} \otimes y,}$$

is a morphism of complexes.

\end{rem}

We now state \cite[Proposition 2.3.4]{BDHO23} for $R_* = C_*(\Omega X,u)$ that is used to build homotopies between morphism of complexes constructed as above.

\begin{prop} \label{Prop: homotopy Criterion Ai}
    Let $\alpha_0$ and $\alpha_1$ be Morse 1-forms and $\F^u$ be a DG right $C_*(\Omega X,u)$-module.
    Let $\{m^0_{x,z} \in C_{|x|-|z|-1}(\Omega X,u), \ x,z \in \Crit(\alpha_0)\}$ and $\{m^1_{y,w} \in C_{|y|-|w|-1}(\Omega X,u), \ y,w \in \Crit(f_1)\}$ be twisting cocycles.
    Let $$\{\nu_{x,y}\in C_{|x|-|y|}(\Omega X,u), \ x \in \Crit(\alpha_0), \ y \in \Crit(\alpha_1)\}$$ and $$\{\nu'_{x,y}\in C_{|x|-|y|}(\Omega X,u), \ x \in \Crit(\alpha_0), \ y \in \Crit(\alpha_1)\}$$ be cocycles that satisfy \eqref{eq: Continuation map}.
    Let $\Psi$ and $\Psi'$ be the morphisms associated with $\{\nu_{x,y}\}$ and $\{\nu'_{x,y}\}$ respectively (see Proposition \ref{prop: continuation map}).

     Suppose that there exists a cocycle $\{h_{x,y} \in C_{|x|-|y|+1}(\Omega X,u)\}$ such that 

\begin{equation}\label{eq: homotopy map}
    \hat{\partial} h_{x,y} = \nu_{x,y} - \nu'_{x,y} + \sum_{z \in \Crit(\alpha_0)} (-1)^{|x|-|z|} m^{0}_{x,z} \cdot h_{z,y} + \sum_{w \in \Crit(\alpha_1)} (-1)^{|x|-|w|} h_{x,w} \cdot m^{1}_{w,y}.
\end{equation}

Then the map $$\deffct{H}{C_*(X, m^0_{x,y}, \F^u)}{C_{*+1}(X, m^1_{x,y}, \F^u)}{\sigma \otimes x}{\displaystyle \sum_{y'} \sigma \cdot h_{x,y'} \otimes y'}$$ 
is a chain homotopy between $\Psi$ and $\Psi'$.
\end{prop}

\begin{rem}
    Let $x,z \in \Crit(\alpha_0)$ and $y,w \in \Crit(\alpha_1)$. Write $$m^0_{x,z} = \sum_{g \in \pi_1(X)} m^{0,g}_{x,z}, \quad m^1_{y,w} = \sum_{g\in \pi_1(X)} m^{1,g}_{y,w}$$ and
    $$\nu_{x,y} = \sum_{g \in \pi_1(X)} \nu^g_{x,y}, \quad \nu'_{x,y} = \sum_{g \in \pi_1(X)} \nu'^g_{x,y}.$$ 
    If $\left\{ h^g_{x,y} \in C_{|x|-|y|+1}(\Omega^gX)\ x \in \Crit(\alpha_0), \ y \in \Crit(\alpha_1) \textup{ and }  g \in \pi_1(X)\right\}$ is a family of chains such that 
    
    \begin{equation}\label{eq: homotopy map g}
    \partial h_{x,y}^g = \nu_{x,y}^g - \nu'^{g}_{x,y} + \sum_{\substack{g'\cdot g''=g \\z \in \Crit(\alpha_0)}} (-1)^{|x|-|z|} m^{0,g'}_{x,z} \cdot h_{z,y}^{g''} + \sum_{\substack{g' \cdot g''=g \\ w \in \Crit(\alpha_1)}} (-1)^{|x|-|w|} h_{x,w}^{g'} \cdot m^{1,g''}_{w,y}
\end{equation}

and $$h_{x,y} = \displaystyle \sum_{g \in \pi_1(X)} h_{x,y}^{g} \textup{ belongs to } C_{|x|-|y|+1}(\Omega X,u).$$

Then, the family $(h_{x,y})$ satisfies \eqref{eq: homotopy map} and $$\deffct{H}{C_*(X,m^0_{x,y},\F^u)}{C_{*+1}(X,m^1_{x,y},\F^u)}{\sigma \otimes x}{\displaystyle \sum_{y \in \Crit(\alpha_1)} \sigma \cdot h_{x,y} \otimes y}$$

is a chain homotopy between $\Psi$ and $\Psi'$.
\end{rem}

\section{Invariance of enriched Morse-Novikov Homology}\label{section : Invariance of DG Morse-Novikov homology}

We will prove in this section that, for any right $C_*(\Omega X,u)$-module $\F^u$, the enriched Morse-Novikov complex $C_*(X,\Xi,\F^u)$ depends, up to chain homotopy equivalence, only on the fixed De Rham cohomology class $u \in H^1(X,\R)$. We will then prove that this complex computes the same homology as the enriched Morse complex $C_*(X,\Xi, \F^u)$ where $\F^u$ is here considered as a right $C_*(\Omega X)$-module.

\subsection{Continuation map}

Given two sets of DG Morse-Novikov data $\Xi_0 = (\alpha,\xi_0,s^{0,g}_{x,y},o_0,\Y_0, \theta_0)$ and $\Xi_1 = (\alpha + dh,\xi_1,s^{1,g}_{x,y},o_1,\Y_1, \theta_1)$ as well as a DG right $C_*(\Omega X,u)$-module $\F^u_*$, we will build a homotopy equivalence called \textbf{continuation map} $$\Psi_{01}: C_*(X,\Xi_0,\F^u) \to C_*(X,\Xi_1,\F^u)$$ by following \cite[Theorem 6.3.1]{BDHO23}.

 \begin{thm}\label{thm : invariance}
    1) Given two sets $\Xi_0$ and $\Xi_1$ of DG Morse-Novikov data on $X$, there exists a \textbf{continuation map} $\Psi_{01}: C_*(X,\Xi_0,\F^u) \to C_*(X,\Xi_1,\F^u)$ that is a homotopy equivalence and its chain homotopy type does only depend on $\Xi_0$ and $\Xi_1$. The map $\Psi_{01}$ is in particular a quasi-isomorphism. 
    
    2) Given another set of data $\Xi_2$ on  $X$ and denoting $\Psi_{ij}$ the continuation map between the data $\Xi_i$ and $\Xi_j$, then $\Psi_{00}$ is homotopic to the identity and $\Psi_{02}$ is homotopic to $\Psi_{12}\circ \Psi_{01}$. In particular, in homology

    $$\Psi_{00} = \textup{Id} \ \textup{and} \ \Psi_{12}\circ \Psi_{01} = \Psi_{02}.$$
\end{thm}

\begin{proof}

This proof consists of three steps:
\begin{enumerate}
    \item Construction of $\Psi_{01} : C_*(X,\Xi_0,\F^u) \to C_*(X,\Xi_1,\F^u)$
    \item If $\Xi_0 = \Xi_1$, then $\Psi_{01}$ is homotopic to the identity.
    \item If $\Xi_2$ is another set of data, then $\Psi_{12} \circ \Psi_{01} : C_*(X,\Xi_0,\F^u) \to C_*(X,\Xi_2,\F^u)$ is homotopic to $\Psi_{02} : C_*(X,\Xi_0,\F^u) \to C_*(X,\Xi_2,\F^u)$.
\end{enumerate}

We will only prove step 1 and refer to the proof of \cite[Theorem 6.3.1]{BDHO23} for the steps 2 and 3. The same arguments apply using the spaces of broken trajectories with fixed endpoints in the universal cover.\\

\underline{\textbf{Step 1:} Construction of $\Psi_{01}$.}

Let $\Xi_0 = (\alpha, \xi_0, s^{0,g}_{x,y}, o_0, \Y_0, \theta_0)$ and $\Xi_1 = (\alpha + dh, \xi_1, s^{1,g}_{x,y}, o_1, \Y_1, \theta_1)$ be two sets of data. Let $\F$ be a DG right $C_*(\Omega X,u)$-module. We will define a chain homotopy equivalence

\[\Psi_{01}: C_*(X,\Xi_0,\F) \to C_*(X,\Xi_1,\F)\]

using Proposition \ref{prop: continuation map}.
For that, we define a \textbf{continuation set of DG Morse-Novikov data} $\Xi$ on $[0,1] \times X$. We will refer to $\Psi_{01}$ as a \textbf{continuation map}. We proceed as in \cite[Section 6.2]{BDHO23}. \\

$\bullet$ Let $\epsilon>0$ and $h : X \to \R$ be a Morse function. Let $H : [-\epsilon,1+ \epsilon] \times X \to \R $ be a smooth homotopy such that $$\left\{ \begin{array}{ll}
    H(t,\cdot) = 0 & \textup{if } t \in [-\epsilon, \epsilon]  \\
     H(1, \cdot) = h & \textup{if } t \in [1-\epsilon, 1+\epsilon] . 
\end{array} \right.$$

Consider $g: [-\epsilon, 1+ \epsilon] \to \R$ a Morse function that has a maximum at $0$, a minimum at $1$  and no other critical points. Assume that $g$ is decreasing enough in $(0,1)$ for the following inequality to be true  $$ \forall x \in X, \ \forall s \in (0, 1), \quad \frac{\partial H(x,s)}{\partial s} + g'(s) < 0.$$

Let $\beta \in \Omega^1([-\epsilon, 1 + \epsilon] \times X)$ be the Morse 1-form defined by $$\forall x \in X, \ \forall s \in (0, 1), \quad \beta_{s,x} = \alpha_x + d_{s,x} H + g'(s)ds.$$

Hence, $$\Crit(\beta) = \{0\} \times \Crit(\alpha) \cup \{1\} \times \Crit(\alpha + dh),$$

and if locally $\alpha = df$ in $U_x$ around each critical point $x \in \Crit(\alpha)$, then
$$\beta_{s,z} = d_{s,z}(f + H + g)$$
in $[-\epsilon,1 + \epsilon] \times U_x$. Since $d^2_{0,x}(f + H + g) = d^2_x f + g''(0)$ and $g''(0)<0$, if $x \in \Crit_k(\alpha)$, then $(0,x) \in \Crit_{k+1}(\beta)$. Similarly, if $y \in \Crit_k(\alpha + dh)$, then $(1,y) \in \Crit_k(\beta)$ because $g''(1)>0$.\\

This reasoning shows, in particular, that $\beta$ is Morse, and we choose a pseudo-gradient $\xi$ for $\beta$ that coincides with $\xi_0 - \textup{grad} \ g$ on $[-\epsilon, \epsilon] \times X$ and with $\xi_1 - \textup{grad} \ g$ on $[1-\epsilon, 1+\epsilon] \times X$.\\
In the following, we denote by $x_0$ the critical point $(0,x)$ of $\beta$ if $x \in \Crit(\alpha)$, and by $y_1$ the critical point $(1,y)$ of $F$ if $y \in \Crit(\alpha + dh)$.\\

$\bullet$ Let $(\Y)_{t \in [0,1]}$ be a homotopy with fixed base point between the trees $\Y_0$ and $\Y_1$ and define 

\[\Y = \bigcup_{t \in [0,1]} ( \{t\} \times \Y_t) \subset [0,1] \times X.\]

$\bullet$ Let $p: [0,1] \times X \to ([0,1] \times X )/\Y$ such that for all $t\in [0,1]$, the restriction $$ p \lvert_{\{t\} \times X}: \left(\{t\} \times X,\{t\} \times \star \right ) \to \left(\faktor{\{t\} \times X}{\Y_t}, \{t\} \times \star \right)$$ is the canonical projection (we recall the notation $[\Y] = \star \in X/\Y$).

Choose $$\Theta: \faktor{([0,1] \times X )}{\Y} \to [0,1] \times X,$$ to be a homotopy inverse $p$ such that for $t\in [0,1]$, it maps $\left(\faktor{\{t\} \times X}{\Y_t}, \{t\} \times \star \right)$ to $\left(\{t\} \times X,\{t\} \times \star\right )$.\\

$\bullet$ For $x_0 \in \Crit(\alpha)$, we choose $$o(x_0):= \Or \ \wb{u}{x_0} = \left( \frac{\partial}{\partial t}, \Or \ \wb{u}{x}\right) = \left( \frac{\partial}{\partial t}, o_0(x)\right).$$

For $y_1 \in \Crit(\alpha + dh)$, we choose $$o(y_0):= \Or \ \wb{u}{y_1} = \Or \ \wb{u}{y} = o_1(y).$$

It remains to build a representing chain system for the moduli spaces of trajectories $\trajbi{\beta}{x_i,y_j}$.

\begin{lemme}\label{sys chaîne induit}

There exists a representing chain system $\left\{s^{\beta,g}_{w,z} \in C_{|w|-|z|-1}\left(\trajbi{\beta,g}{w,z}\right), \ w,z \in \Crit(\beta), g \in \pi_1(X)\right\}$ such that

\begin{enumerate}
    \item If $w=x_0$ and $z=y_0$, $$s_{x_0,y_0}^{\beta,g} = (-1)^{|x| - |y|}s_{x,y}^{0,g}.$$
    \item If $w=x_1$ and $z=y_1$, $$s_{x_1,y_1}^{\beta,g} = s_{x,y}^{1,g}.$$
\end{enumerate}

\end{lemme}

\begin{proof}

Let $x,y \in \Crit(\alpha)$ and $g \in \pi_1(X)$.
We compare the orientations of the two spaces $\trajbi{\beta,g}{x_0,y_0}$ and $\trajbi{\alpha,g}{x,y}$. The orientation rules give:

$$\begin{array}{llcl}
 & \left( \Or \ \trajbi{\beta,g}{x_0,y_0}, - \xi, \Or \ \wbi{u}{\beta}{y_0}\right) & = & \Or \ \wbi{u}{\beta}{x_0}\\
\Longleftrightarrow & \left( \Or \ \trajbi{\beta,g}{x_0,y_0}, - \xi, \frac{\partial}{\partial t}, \Or \ \wbi{u}{\alpha}{y}\right) & = & \left(\frac{\partial}{\partial t}, \Or \ \wbi{u}{\alpha}{x} \right)\\
\Longleftrightarrow & (-1)^{|x| - |y|}\left( \frac{\partial}{\partial t},  \Or \ \trajbi{\beta,g}{x_0,y_0}, - \xi, \Or \ \wbi{u}{\alpha}{y}\right) & = &  \left(\frac{\partial}{\partial t}, \Or \ \wbi{u}{\alpha}{x}\right)\\
\Longleftrightarrow & (-1)^{|x| - |y|}\left( \Or \ \trajbi{\beta,g}{x_0,y_0}, - \xi, \Or \ \wbi{u}{\alpha}{y}\right) & = &  \Or \ \wbi{u}{\alpha}{x}.\\
\Longleftrightarrow & \Or \ \trajbi{\alpha,g}{x,y} & = & (-1)^{|x| - |y|} \Or \ \trajbi{\beta,g}{x_0,y_0}.
\end{array}$$

Therefore, $s^{\beta,g}_{x_0,y_0}:= (-1)^{|x|-|y|} s^{0,g}_{x,y}$ is a cycle that represents the fundamental class of $\trajbi{\beta,g}{x_0,y_0}$ and it satisfies \eqref{eq: système representatif}.

The proof that $s^{\beta,g}_{x_1,y_1}:= s^{1,g}_{x,y}$ is a cycle that represents the fundamental class of $\trajbi{\beta,g}{x_1,y_1}$ and satisfies \eqref{eq: système representatif} is similar.

 From there, we apply the procedure described in Proposition \ref{prop: representing chain system} to inductively build $$s^{\beta,g}_{x_0,y_1} \in C_{|x|-|y|}(\trajbi{\beta,g}{x_0,y_1})$$ for $x \in \Crit(\alpha)$ and $y \in \Crit(\alpha +dh)$ such that $$\left\{s^{\beta,g}_{w,z} \in C_{|w|-|z|-1}\left(\trajbi{\beta,g}{w,z}\right), \ w,z \in \Crit(\beta)\right\}$$ is a representing chain system for the moduli spaces of trajectories of $(\beta,\xi)$.
\end{proof}

Let $\left\{s^{\beta,g}_{w,z} \in C_{|w|-|z|-1}\left(\trajbi{\beta,g}{w,z}\right), \ w,z \in \Crit(\beta)\right\}$ be such a representing chain system.

Let $x \in \Crit(\alpha)$, $y\in \Crit(\alpha +dh)$, denote $\sigma^g_{x,y}:= s^{\beta,g}_{x_0,y_1}  \in C_{|x|-|y|}(\trajbi{\beta,g}{x_0,y_1})$. Equation \eqref{eq: système representatif} for $\sigma_{x,y}^g$ becomes 

\begin{equation}\label{sigma 2}
    \partial \sigma_{x,y}^g = \sum_{\substack{z \in Crit(\alpha)\\ g'\cdot g''=g}} s_{x,z}^{0,g'} \times \sigma_{z,y}^{g''}
- \sum_{\substack{w \in Crit(\alpha + dh)\\ g' \cdot g''=g}} (-1)^{|x| -|w|} \sigma_{x,w}^{g'} \times s_{w,y}^{1 , g''}.
\end{equation}

For each $w,z \in \Crit(\beta)$, define an evaluation map $q^{\beta}_{w,z}: \trajbi{\beta}{w,z} \to \Omega X$ by $$q^{\beta}_{w,z} = \pi_X \circ \Theta \circ p \circ \Gamma^{\beta}_{w,z},$$ where $\pi_X: [0,1] \times X \to X$ is the canonical projection and $\Gamma^{\beta}_{x_0,y_1}: \trajbi{\beta}{w,z} \to \mathcal{P}_{w,z}([0,1] \times X)$ is the parametrization map defined in Lemma \ref{lemme: parametrization map} using $\beta$. Since $$q^{\beta}_{w,z} = \pi_X \circ q_{w,z},$$ where $q_{w,z}: \trajbi{\beta}{w,z} \to \Omega([0,1] \times X)$ is the evaluation map defined in Lemma \ref{lemme: def q} for the pair $(\beta,\xi)$, we have:

\begin{lemme}\label{lemme : q beta}

The maps $q_{w,z}^{\beta}: \trajbi{\beta}{w,z} \to \Omega X$ satisfy the following properties:\\

\begin{enumerate}
    \item If $x_0 \in \Crit_{k+1}(\beta)$ and $y_1 \in \Crit_k(\beta)$, then for all $\lambda \in \mathcal{L}_{\beta} (x,y)$, $g(\lambda) = [q^{\beta}_{x_0,y_1}(\lambda)]$.\\
    \item If $(\lambda, \lambda') \in \trajbi{\beta}{w,d} \times \trajbi{\beta}{d,z}$, then $q_{w,z}^{\beta}(\lambda, \lambda') = q_{w,d}^{\beta}(\lambda) \# q_{d,z}^{\beta}(\lambda')$.\\
    \item If $x_i,y_i \in Crit(\beta)$, then $q_{x_i,y_i}^{\beta} = q_{x,y}$ is the function defined in Lemma \ref{lemme: def q} for the pair $(\alpha + idh, \xi_i)$ for $i \in \{0,1\}$.
\end{enumerate}

\end{lemme}

We then define, for all $g \in \pi_1(X)$:
\begin{itemize}
\item $\nu_{x,y}^g = - q_{x,y,*}^{\beta} (\sigma_{x,y}^g)$
\item $m_{x,y}^{0,g} = q_{x,y,*}^{\beta} ( s_{x,y}^{0,g})$
\item $m_{x,y}^{1,g} = q_{x,y,*}^{\beta} ( s_{x,y}^{1,g})$.
\end{itemize}

From \eqref{sigma 2}, we deduce 

\[\partial \nu_{x,y}^g = \sum_{\substack{z \in \Crit(\alpha)\\ g'.g''=g}} m_{x,z}^{0,g'} \cdot  \nu_{z,y}^{g''} - \sum_{\substack{w \in \Crit(\alpha + dh) \\ g'.g''=g}} (-1)^{|x| - |w|} \nu_{x,w}^{g'} \cdot m^{1, g''}_{w,y}.\]

We use the same arguments as in Proposition \ref{prop : twisting cocycle MN bien dans la completion} to prove that $\nu_{x,y} = \displaystyle \sum_{g \in \pi_1(X)} \nu^g_{x,y} \in C_{|x|-|y|}(\Omega X,u)$. Proposition \ref{prop: continuation map} shows that 

$$\begin{array}{cccc}
    \Psi_{01}:  & C_*(X,\Xi_0,\F) &\to& C_*(X,\Xi_1,\F)  \\
      & (\sigma \otimes x) & \mapsto &  \displaystyle \sum_{y \in \Crit(\alpha + dh)}\sigma \cdot \nu_{x,y} \otimes y.
\end{array}$$ 

is a morphism of complexes.

\end{proof}

\subsection{Latour Trick}\label{subsection : Latour Trick}

Let $\F^u_*$ be a right $C_*(\Omega X,u)$-module and $r >0$.
Since, up to homotopy equivalence, the complex $C_*(X,\Xi, \F^u_*)$ does not depend on $r$ nor on the representative $\alpha \in r u$, we can choose a particular representative $\alpha \in r u$ and an adapted pseudo-gradient $\xi$ such that  $C_*(X,\Xi, \F^u_*)$ is chain homotopy equivalent to the enriched Morse complex with coefficients in $\F^u_*$ seen as a right $C_*(\Omega X)$-module. 

\begin{prop}[Latour Trick] \label{prop : Latour Trick}
    Given a Morse function $h: X \to \R$, there exists a set of DG Morse-Novikov data $\Xi^L=(\alpha,\xi,s^g_{x,y},o,\mathcal{Y},\theta)$ and set of DG Morse data $\Xi^M=(h,\xi,\sum_g s^g_{x,y},\mathcal{Y}, \theta)$  such that  $C_*(X,\Xi^L,\F^u_*)$ is a DG Morse complex and there exists a chain homotopy equivalence between the DG Morse complexes $$C_*(X,\Xi^L,\F^u_*) \cong C_*(X,\Xi^M,\F^u_*),$$ where $\F^u_*$ is seen as a $C_*(\Omega X)$-module.
\end{prop}

This proposition is called \textbf{Latour Trick} since the argument used in the proof of this proposition is similar to the argument Latour used to prove in \cite{Lat94} that the Morse-Novikov homology is the Morse homology with local coefficients in $\Lambda_u$.

Using the property of invariance with respect to the set of DG Morse-Novikov data and to the set of DG Morse data, we obtain the following corollary.

\begin{cor}\label{cor : DG Morse Novikov homotopy equivalent to DG Morse}
    For all set of DG Morse-Novikov data $\Xi$ and all set of DG Morse data $\Xi_M$, 

    $$C_*(X,\Xi,\F^u_*) \cong C_*(X,\Xi_M,\F^u_*)$$ and in particular $H_*(X,\F^u_*,u) \cong H_*(X,\F^u_*).$
\end{cor}

\begin{proof}[Proof of Proposition \ref{prop : Latour Trick}]
    
Let $\alpha_0 \in u$ be a Morse 1-form on $X$, let $h: X \to \R$ be a Morse function and $U$ an open neighborhood of $\Crit(h)$ small enough such that $\alpha_0\lvert_U = dg\lvert_U$, where $g: X \to \R$ is a smooth function. Consider $\alpha_1 = \alpha_0-dg \in u$ that is zero on $U.$ For $\epsilon >0$, consider $$\alpha_{\epsilon} = \epsilon \alpha_1 +  dh \in \epsilon u$$ and $\xi$ a pseudo-gradient adapted to the Morse function $h$. We infer that, if $\epsilon$ is small enough, then $\xi$ is a pseudo-gradient adapted to
$$\alpha:= \alpha_{\epsilon}$$
and $\Crit(\alpha) = \Crit(h).$ For this particular pair $(\alpha,\xi)$, the orbits are those of the Morse-Smale pair $(h,\xi)$ and therefore $$\bigcup_{g \in \pi_1(X)} \trajbi{\alpha,g}{x,y} = \trajbi{\alpha}{x,y} = \trajbi{h}{x,y}.$$ In particular, if $\left\{s^g_{x,y}, \ x,y \in \Crit(\alpha), g \in \pi_1(X)\right\}$ is a representing chain system for the moduli spaces of trajectories $\trajbi{\alpha,g}{x,y}$, then $$s_{x,y} = \sum_{g \in \pi_1(X)} s^g_{x,y} \in C_{|x|-|y|-1}\left(\trajbi{h}{x,y}\right)$$ is representing chain system for the moduli spaces of trajectories $\trajbi{h}{x,y}$ and
$$m_{x,y} = \sum_g m^g_{x,y} \in C_{|x|-|y|-1}(\Omega X)$$
is the Barraud-Cornea twisting cocycle associated with $\Xi^L_M = (h,\xi, s_{x,y},o,\mathcal{Y}, \theta)$ as defined in \cite{BDHO23} up to a correction of parametrization. Indeed, the Barraud-Cornea twisting cocycle is obtained by parametrizing the trajectories by the values of $h \circ \tilde{\pi}$ in the universal cover (see \cite[Proposition 6.6]{Rie24} and Remark \ref{rem : param by an exact 1-form}), while the family $\{m^g_{x,y}\}$ is obtained by parametrizing the trajectories by $h \circ \tilde{\pi} + \epsilon \tilde{h}$ where $d\tilde{h} = \pi^*\alpha_1$. We know from \cite[Proposition 6.6]{Rie24} that DG Morse complexes do not depend, up to chain homotopy equivalence, on the function we use to parametrize the trajectories. It follows that the complexes $C_*(X,\Xi,\F^u_*)$ and $ C_*(X,\Xi_M,\F^u_*)$ are homotopy equivalent.
\end{proof}

\subsection{Computation of \texorpdfstring{$H_0$}{H0} and \texorpdfstring{$H_1$}{H1}}\label{subsection : Computation of H0 and H1}

Let $u \in H^1(X,\R)$ and let $\F^u$ be a right $C_*(\Omega X,u)$-module.
Assume that $u \neq 0$. Latour has shown \cite[Lemme 4.1]{Lat94} that there exists a Morse 1-form $\alpha \in u$ such that $\Crit_0(\alpha) = \emptyset.$ Let $\Xi = (\alpha, \xi, \dots)$ be a set of DG Morse-Novikov data.

\paragraph{Computation of $H_0$.}

With this particular set of data,
$$C_0(X,\Xi,\F^u) = \F^u_0 \otimes \Z\Crit_0(\alpha) = 0.$$

Since $H_0(X,\F)$ does not depend on the chosen set of data, it follows that $H_0(X,\F^u) = 0.$

\paragraph{Computation of $H_1$.}

We prove that

$$\begin{array}{lll}
    H_1(X,\F^u) & \cong H_1(H_0(\F^u_*)  \otimes_{\Lambda_u} C_*(\alpha,\xi)) & \cong H_0(\F^u_*) \otimes_{\Lambda_u} \Tor^{\Z[\pi_1(X)]}_1(\Z, \Lambda_u)  \\
     & \cong H_1(\tilde{C}_*(h,\xi) ; H_0(\F^u_*)) & \cong \Tor^{\Z[\pi_1(X)]}_1(\Z , H_0(\F^u)), 
\end{array}$$

where $\tilde{C}_*(h,\xi)$ is the lifted Morse complex associated with a Morse-Smale pair $(h,\xi)$ that computes the homology of the universal cover $\tilde{X}.$
 In particular, $H_1(X,\F^u)$ depends only on $\pi_1(X)$ and on $H_0(\F^u_*).$ We start with the first line using a Morse 1-form $\alpha$ such that $\Crit_0(\alpha) = \emptyset$

Let $(E^r_{p,q})$ be the spectral sequence associated with the filtration $$F_p(C_k(X,\Xi,\F^u)) = \bigoplus_{\substack{i+j=k \\ i \leq p}} \F_j^u \otimes \Z \Crit_p(\alpha).$$

We proved in Section \ref{subsection: suite spectrale M-N} that $E^1_{p,q} = H_q(\F^u_*) \otimes \Z \Crit_p(\alpha)$ and $E^2_{p,q} = H_p(H_q(\mathcal{F}^u_*) \otimes_{\Lambda_u} C(\alpha,\xi)).$
For degree reasons, $$E^{\infty}_{1,0} = E^{2}_{1,0} = H_1(H_0(\mathcal{F}^u_*) \otimes_{\Lambda_u} C(\alpha,\xi)).$$ Since $\Crit_0(\alpha) = \emptyset$, $$E^{\infty}_{0,1} = E^0_{0,1} = \F^u_1 \otimes \Z \Crit_0(\alpha) = 0.$$

McCleary proved in \cite[Theorem 2.6]{McC01}, that if we denote $$F_pH_k(X,\F^u) = \textup{Im}\left(H_k(F_p(C_*(X,\Xi,\F^u))) \overset{\textup{inclusion}_*}{\to} H_k(X,\F^u) \right),$$ then
$$0= E^{\infty}_{0,1} = F_0 H_1(X,\F^u)$$ and
$$E^{\infty}_{1,0} = \faktor{F_1H_1(X,\F^u)}{F_0H_1(X,\F^u)} = F_1H_1(X,\F^u).$$

Since $F_2(C_2(X,\Xi,\F^u)) = C_2(X,\Xi,\F^u)$ and $F_2(C_1(X,\Xi,\F^u)) = C_1(X,\Xi,\F^u)$, it follows that $H_1(F_2(C_*(X,\Xi,\F^u))) = H_1(X,\F^u)$ and 
$$0 = E^{\infty}_{2,-1} = \faktor{H_1(X,\F^u)}{E^{\infty}_{1,0}}.$$

We conclude that \begin{equation}
    H_1(X,\F^u) = E^{\infty}_{1,0} = E^2_{1,0} = H_1(H_0(\mathcal{F}^u_*) \otimes_{\Lambda_u} C(\alpha,\xi)).
\end{equation} 

We proceed in the same manner using the spectral sequence of change of coefficients $(F^r_{p,q})$ (see \cite[Théorème 5.5.1]{God73}) that converges towards  $H_*(C_*(\alpha,\xi), H_0(\F^u_*))$. Its second page is given by $$\Tor_p^{\Lambda_u}(H_q(X; u), H_0(\F^u_*)).$$

Since $$F^{\infty}_{1,0} = F^2_{1,0} = \Tor^{\Lambda_u}_1(H_0(\F^u_*), H_0(X; u)) = 0$$ and $$F^{\infty}_{0,1} = F^2_{0,1} = H_0(\F^u_*) \otimes_{\Lambda_u} H_1(X;u) \cong H_0(\F^u_*) \otimes_{\Lambda_u} \Tor^{\Z[\pi_1(X)]}_1(\Z, \Lambda_u),$$ it follows that

\begin{equation}
    H_1(X, \F^u_*) \cong H_0(\F^u_*) \otimes_{\Lambda_u} \Tor^{\Z[\pi_1(X)]}_1(\Z, \Lambda_u).
\end{equation}

Indeed, using the Universal Coefficient theorem for principal domains, we have a short exact sequence $$0 \to \underbrace{H_1(\tilde{X})}_{=0} \otimes_{\Z[\pi_1(X)]} \Lambda_u \to  H_1(X;u) \to \Tor^{\Z[\pi_1(X)]}_1(H_0(\tilde{X}), \Lambda_u)\to 0$$ that gives an isomorphism  $H_1(X;u) \cong  \Tor^{\Z[\pi_1(X)]}_1(\Z, \Lambda_u)$.

We now use the Latour Trick to prove that $$H_1(X,\F^u) \simeq H_1(\tilde{C}_*(h,\xi) ; H_0(\F^u)) \simeq \Tor^{\Z[\pi_1(X)]}_1(\Z, H_0(\F^u_*)).$$

Let $h : X \to \R$ be a Morse function and let $\xi$ be a pseudo-gradient adapted to $h$. We proved in Proposition \ref{prop : Latour Trick} that the enriched Morse-Novikov homology $H_*(X,\F^u_*)$ coincides with the enriched Morse homology still denoted $H_*(X,\F^u_*)$ where $\F^u_*$ is here considered as a right $C_*(\Omega X)$-module. There is a spectral sequence $(E^r_{p,q})$ that converges towards $H_*(X,\F^u_*)$ and its second page is 

$$E^2_{p,q} = H_{p}( \tilde{C}_*(h,\xi) ; H_q(\F^u_*)).$$

In particular, $E^2_{1,0} = H_1(\tilde{C}_*(h,\xi) ; H_0(\F^u_*))$ and $$E^2_{0,1} = H_0(\tilde{C}_*(h,\xi) ; H_1(\F^u_*)) \cong H_1(\F^u_*) \otimes_{\Z[\pi_1(X)]} H_0(\tilde{X}) \cong H_1(\F^u_*) \otimes_{\Z[\pi_1(X)]}  \Z  = 0.$$

Indeed, let $\sigma \otimes 1 \in H_1(\F^u_*) \otimes_{\Z[\pi_1(X)]}  \Z$ and $g  \in \pi_1(X)$ such that $u(g)<0.$  Since $H_1(\F^u_*)$ is a $\Lambda_u$-module and $1-g$ is invertible in $\Lambda_u$ of inverse $(1-g)^{-1} = \sum_{i\geq 0} g^i,$ it follows that $$\sigma \otimes 1  =  (1-g)(1-g)^{-1} \cdot \sigma \otimes 1 =  (1-g)^{-1} \cdot \sigma \otimes (1-g)\cdot 1 =  (1-g)^{-1} \cdot \sigma \otimes (1-1) = 0.$$

Using the Universal Coefficient theorem, it follows that \begin{equation}
    H_1(X,\F^u) \cong H_1(\tilde{C}(h,\xi); H_0(\F^u)) \cong \Tor^{\Z[\pi_1(X)]}_1(\Z, H_0(\F^u)).
\end{equation}

\section{A Fibration Theorem for enriched Morse-Novikov theory. Proof of Theorem \ref{theorem C}}\label{section : A Fibration Theorem for DG Morse-Novikov theory}

In this section, we present a Morse-Novikov generalization of the Fibration Theorem

\begin{thm} \cite[Theorem 7.2.1]{BDHO23}
    Let $\Xi$ be a set of DG Morse-Novikov data on $X$ and let $E \to X$ be a Hurewicz fibration over $X$ of fiber $F=\pi^{-1}(\star)$. Let $\Phi: E \ftimes{\pi}{\ev_0} \mathcal{P}X \to E$ be a transitive lifting function for this fibration and endow $C_*(F)$ with the associated $C_*(\Omega X)$-module structure. Then, there exists a quasi-isomorphism

    $$\Psi_E: C_*(X,\Xi,C_*(F)) \to C_*(E).$$
\end{thm}

The holonomy of the fiber induced by the transitive lifting function $\Phi : F \times \Omega X \to F$ makes $C_*(F)$ the natural DG right $C_*(\Omega X)$-module to consider in the context of a fibration. We will define a Novikov completion $C_*(F,u)$ of $C_*(F)$ that carries a $C_*(\Omega X,u)$-module structure induced by $\Phi : F \times \Omega X \to F$ and will prove that the complex $C_*(X,C_*(F,u))$ is quasi-isomorphic to a Novikov completion $C_*(E,u)$ of $C_*(E)$.

\subsection{Primitive and Novikov completion for the cubical complex of the total space of a fibration}

Let $F \hookrightarrow \fibration$ be a (Hurewicz) fibration and fix $\star_E \in F$. We can and will assume that $E$ is connected. 
Let $u \in H^1(X)$ be a cohomology class such that the morphism $\pi^*u := u\circ \pi : \pi_1(E) \to \R$ is zero. 

Let $$\hat{X}_u := \faktor{\mathcal{P}X}{\sim_u} \overset{\hat{\pi}}{\to} X$$ be the integration cover of $u \in H^1(X,\R)$, where $[\gamma] \sim_u [\tau]$ if $\ev_0(\gamma) = \ev_0(\tau)$, $\ev_1(\gamma) = \ev_1(\tau)$ and $[\gamma\#\tau^{-1}] \in \Ker(u)$. The covering $\hat{\pi} : \hat{X}_u \to X$ is given by $\hat{\pi}([\gamma]) = \ev_1(\gamma).$

\begin{lemme}\label{lemme : pi*u=0 iff factors through integration cover}

Assume $\pi^*u=0$. The map $$\deffct{\varphi_u}{E}{\hat{X}_u}{e}{[\pi \circ \gamma_e]},$$ where $\gamma_e \in \mathcal{P}_{\star_E \to e}$ is any path joining $\star_E$ to $e$ is well-defined and the fibration $\fibration$ factors through $$\xymatrix{
E \ar[dr]_{\pi} \ar[r]^{\varphi_u}& \hat{X}_u \ar[d]^{\hat{\pi}} \\
& X.
}$$ 
If we also assume that $F$ is path-connected, then $\varphi_u$ is continuous and is a fibration.
\end{lemme}

\begin{proof}
    Let $e \in E$  and $\gamma_e, \tau_e \in \mathcal{P}_{\star_E \to e} E$. Then $\pi^*u(\gamma_e \# \tau_e^{-1}) = 0$ and therefore $[\pi \circ \gamma_e] = [\pi\circ\tau_e] \in \hat{X}_u.$ It follows that the map $$\deffct{\varphi_u}{E}{\hat{X}_u}{e}{[\pi \circ \gamma_e]}$$ is well-defined since $[\pi \circ \gamma_e] \in \hat{X}_u$ is independent on the choice of $\gamma_e \in \mathcal{P}_{\star_E \to e}E.$ Moreover it is clear from the definition that $\hat{\pi} \circ \varphi_u(e) = \ev_1(\pi \circ \gamma_e) = \pi(e)$ for any $e \in E$. 

    We now prove that $\varphi_u$ is continuous. Let $e \in E$ and $(e_n) \in E^{\N}$ be a sequence in $E$ that converges towards $e$. The goal is to define a sequence of paths $(\delta_n)$ such that for every $n \geq 0$, $\delta_n \in \mathcal{P}_{e_n \to e} E$ and 
    $\pi\circ\delta_n$ converges to a loop $\ell$ in $\Ker(u).$
    Indeed, if we find such a sequence of paths, then for any path $\gamma_e \in \mathcal{P}_{\star_E \to e} E$, we can compute
     $$
        \varphi_u(e_n) = [\pi \circ (\gamma_e \# \delta_n^{-1})] \to [(\pi \circ \gamma_e) \# \ell^{-1}] = [\pi \circ \gamma_e] = \varphi_u(e).$$

    Let $(b_n) = (\pi(e_n))$ a sequence in $X$ that converges towards $b = \pi(e).$ Since $X$ is locally path-connected, there exists a family of paths $\gamma_n \in \mathcal{P}_{b_n \to b} X$, for $n$ sufficiently large, that converges towards the constant path at $b$.  Let $\gamma \in \mathcal{P}_{\star \to b}$ and $a\geq 0$ such that $\gamma:[0,a] \to X$. For $n \in \N$, denote
    $$g_n =  \Phi(e_n,\gamma_n\# \gamma^{-1}) \in \pi^{-1}(\star)$$
    and
    $$\eta_n = s \mapsto \Phi(e_n,(\gamma_n \# \gamma^{-1})_{\lvert_{[0,s]}}) \in \mathcal{P}_{e_n \to g_n} E.$$
    Since $\Phi$ is continuous, the sequence $(g_n)$ converges towards $ v=\Phi(e,\gamma^{-1}) \in F$. Since $F$ is supposed locally path-connected, there exists a family of paths $$\tau_n \in \mathcal{P}_{g_n \to v} F$$ for $n$ large enough that converges to the constant path at $v$. Consider the sequence of paths defined by
    
    $$\nu_n = \eta_n \# \tau_n \# (s \mapsto \Phi(v,\gamma_{\lvert_{[0,s]}})) \in \mathcal{P}_{e_n \to e'} E,$$
    
    where $e'= \Phi(v,\gamma) = \Phi(e,\gamma^{-1} \gamma)$. Let 
    $$\kappa :t \mapsto \Phi(e, \gamma^{-1}_{\lvert_{[0,t]}} \# \gamma_{\lvert_{[a-t,a]}}) \in \mathcal{P}_{e' \to e} \pi^{-1}(b).$$

    \begin{figure}[h!]
        \centering
        \includegraphics[width=0.5\linewidth, height = 9cm]{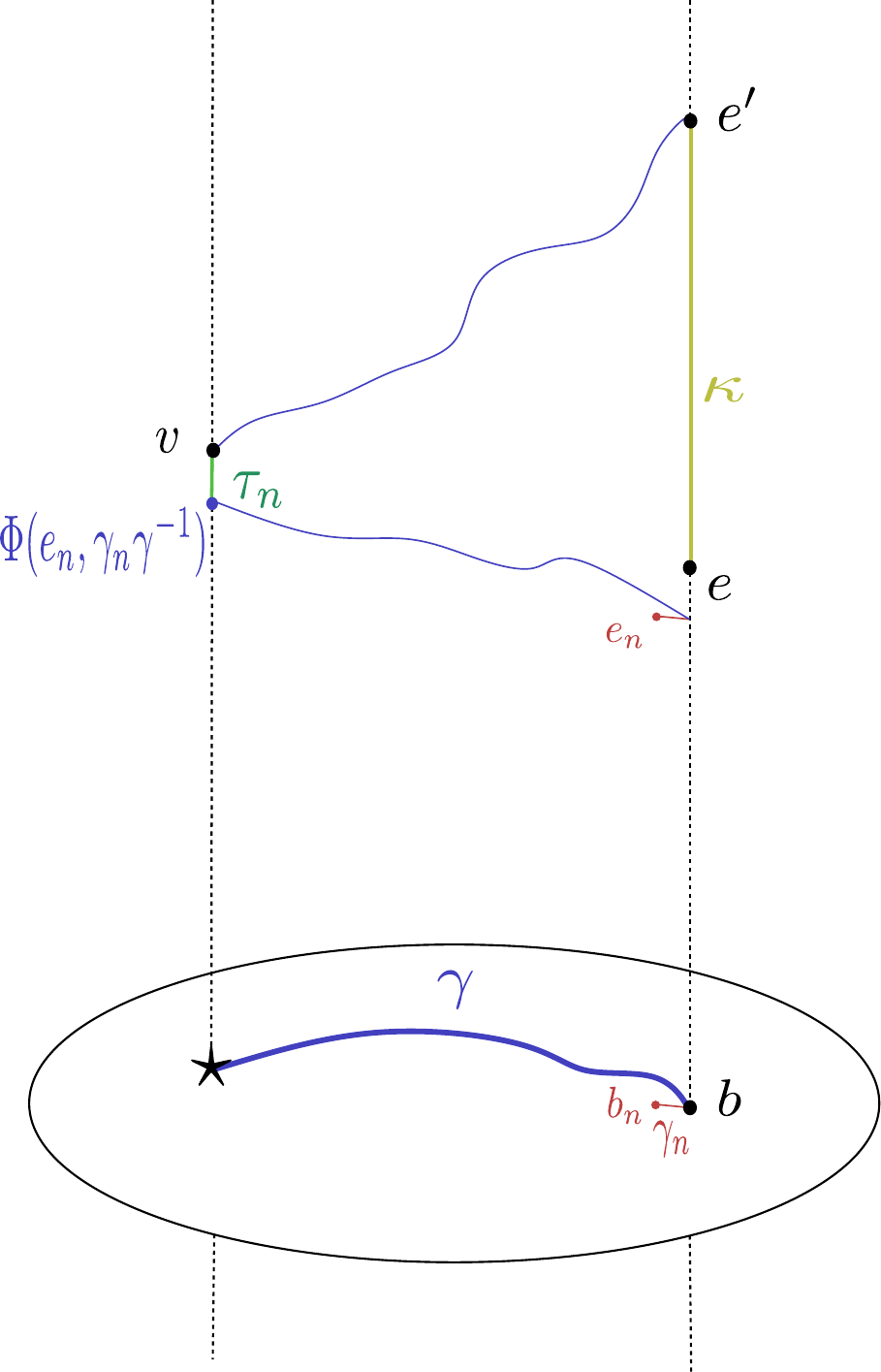}
        \caption{The path $\delta_n$.}
        \label{fig : delta n}
    \end{figure}

    Define the family of paths $\delta_n = \kappa \# \nu_n \in \mathcal{P}_{e_n \to e} E$ (see Figure \ref{fig : delta n}). Since $\pi \circ \delta_n = \gamma_n \# \gamma^{-1} \#  \gamma$, it follows that $\pi \circ \delta_n$ converges towards $\ell = \gamma^{-1}\# \gamma \in \Ker(u)$.

    It remains to prove that $\varphi_u$ is a fibration. Let $W$ be a topological space, $a: W \to E$ be a continuous map and let $b : W \times [0,1] \to \hat{X}_u$ be a continuous map such that $b(\cdot,0) = \varphi_u \circ a.$

$$\xymatrix{
W \times \{0\} \ar@{^{(}->}[d] \ar[r]^{a} & E \ar[d]^{\varphi_u} \\
W \times [0,1] \ar@{-->}[ur]^r \ar[dr]^{\hat{\pi} \circ b} \ar[r]^b & \hat{X}_u \ar[d]^{\hat{\pi}} \\
& X
}$$
    
Since $\pi = \hat{\pi} \circ \varphi$ is a fibration, there exists $r : W \times [0,1] \to E$ such that $\hat{\pi} \circ \varphi_u \circ r = \hat{\pi} \circ b =: g.$ It suffices to prove that $\varphi_u \circ r = b$ to conclude that $\varphi_u$ is a fibration.
The maps $\varphi_u \circ r: W \times [0,1] \to \hat{X}_u$ and $b : W \times [0,1] \to \hat{X}_u$ induce two sections of the covering $g^*\hat{X}_u \to  W \times [0,1]$ that agree on $W \times \{0\}.$ There are therefore equal.
\end{proof}

\begin{rem}
    Conversely, if there exists a map $\varphi_u : E \to \hat{X}_u$ such that $\pi = \hat{\pi} \circ \varphi_u$, then $$\pi^*u = \varphi_u^*\hat{\pi}^*u = \varphi_u^*0 = 0$$ since $\pi_1(\hat{X}_u) = \ker(u).$
\end{rem}

Let $\alpha \in u$.
Since $\hat{\pi}^*u = 0 \in H^1(\hat{X}_u,\R)$, the map $\hat{f} : \hat{X}_u \to \R$, $\hat{f}([\gamma]) = \int_{\gamma} \alpha$ is a primitive of $\hat{\pi}^*\alpha \in \hat{\pi}^*u.$ 

\begin{defi}\label{def : primitive of pi alpha}
   We will refer to the map $f = \hat{f} \circ \varphi_u : E \to \R$, $f(e) = \int_{\varphi_u(e)} \alpha$ as a \textbf{primitive of} $\boldsymbol{\pi^*\alpha}.$  
\end{defi}

Remark that the definition of $\varphi_u : E \to \hat{X}_u$ depends on a choice of $\star_E \in F$.
We now prove that a primitive of $\pi^*\alpha$ depends only up to a constant on this choice of $\star_E$, is compatible with the transitive lifting function and is continuous if the fiber of the fibration is locally path-connected.

\begin{prop}\label{prop : existence of primitive}
    Assume that $F$ is locally path-connected. Let $\star_E \in \pi^{-1}(\star)$ and $\alpha \in u$. If $\pi^*u = 0$, the primitive of $\pi^*\alpha$ $$\deffct{f= \hat{f}\circ \varphi_u}{E}{\R}{e}{\int_{\pi\circ \gamma_e} \alpha,}$$ where $\gamma_e \in \mathcal{P}_{\star_E \to e} E$ is any path joining $\star_E$ to $e$, is continuous and depends only up to a constant on the chosen lift $\star_E$ of $\star.$ Moreover, it satisfies 
    \begin{equation}\label{eq : primitive f}
        \forall (e,\gamma) \in E \ftimes{\pi}{ev_0} \mathcal{P} X, \ f(\Phi(e,\gamma)) - f(e) = \int_{\gamma} \alpha.
    \end{equation}
\end{prop}

\begin{proof}
    If $\star'_E \in \pi^{-1}(\star)$, denote $f' : E \to \R$ the map defined using $\star'_E$ instead of $\star_E$. We will prove that $f'-f : E \to \R$ is constant. Choose a path $\gamma_{\star} \in \mathcal{P}_{\star'_E \to \star_E} E$ (since $E$ is assumed to be path-connected) and let $e \in E$. If $\gamma_e \in \mathcal{P}_{\star_E \to e} E$, then $\gamma_* \# \gamma_e \in \mathcal{P}_{\star'_E \to e} E$.
    Then
    \begin{align*}
        f'(e) - f(e) &= \int_{\pi \circ( \gamma_*\# \gamma_e)} \alpha - \int_{\pi\circ\gamma_e} \alpha \\
        &= \int_{\pi\circ \gamma_*} \alpha.
    \end{align*}

    Let $(e,\gamma) \in E \ftimes{\pi}{\ev_0} \mathcal{P}X$ and denote $\eta \in \mathcal{P}_{e \to \Phi(e,\gamma)} E$ defined by $$\eta(s) = \Phi(e,\gamma_{\lvert_{[0,s]}}).$$
    If $\gamma_e \in \mathcal{P}_{\star_E \to e} E$, then $\gamma_e\# \eta \in \mathcal{P}_{\star_E \to \Phi(e,\gamma)}. $ It follows that 

    \begin{align*}
        f(\Phi(e,\gamma)) - f(e) = \int_{\pi\circ \eta} \alpha = \int_{\gamma} \alpha.
    \end{align*}

    The map $f=\hat{f}\circ \varphi_u$ is a composition of continuous maps and is therefore continuous.
\end{proof}

\begin{defi}
    Let $k \in \N$. Define the complexes
    
    $$C_k(E,f):= \left\{ \sum_{i \geq 0} n_i \sigma_i , \ n_i \in \Z, \ \sigma_i \in C^0([0,1]^k,E), \ \forall c\in \R,  \#\left\{i, \ n_i \neq 0 \textup{ and } \max_{t \in [0,1]^k} f \circ \sigma_i(t) > c\right\} < +\infty\right\},$$ 

    endowed with the differential induced by the differential on $C_*(E)$    and

    $$C_k(F,f):= \left\{ \sum_{i \geq 0} n_i \sigma_i , \ n_i \in \Z, \ \sigma_i \in C^0([0,1]^k,F), \ \forall c\in \R,  \#\left\{i, \ n_i \neq 0 \textup{ and } \max_{t \in [0,1]^k} f \circ \sigma_i(t) > c\right\} < +\infty\right\},$$

    endowed with the differential induced by the differential on $C_*(F)$. 
We will denote the associated homology groups $H_*(E,u)$ and $H_*(F,u)$.
\end{defi}

\begin{lemme}
    For each $k \in \N$ and $c \in \R$, let

    $$E_{c,f} = f^{-1}((-\infty,c]) \textup{ and } F_{c,f} = E_{c,f} \cap F.$$

    We have the following isomorphisms: $$C_k(E,f) \simeq \limproj C_k(E,E_{c,f}),$$ and $$C_k(F,f) \simeq \limproj C_k(F,F_{c,f}).$$
\end{lemme}
\begin{proof}
    This is a direct consequence of Lemma \ref{Lemme: limite projective d'un quotient de complexe}.
\end{proof}

\begin{notation}
    When $f$ is implicit, we will denote $E_c = E_{c,f}$, $F_c = F_{c,f}$ and $C_*(E,u) = C_*(E,f)$, $C_*(F,u) = C_*(F,f)$. 
\end{notation}

The notation reflects the fact that the complexes $C_*(E,u)$ and $C_*(F,u)$ are independent, up to chain homotopy equivalence, of the choice $\alpha \in u$ and of the primitive $f: E \to \R$ of $\pi^*\alpha.$

\begin{lemme}
    The complexes $C_*(E,u)$ and $C_*(F,u)$ depend only on $u \in H^1(X)$ up to chain homotopy equivalence.
\end{lemme}
\begin{proof}
    First, fix a Morse 1-form $\alpha$ representing $u$. Since any primitive of $\pi^*\alpha$ differ by a constant that depend on the choice of the lift $\star_E$, the complex $C_*(E,u)$ does not depend on the chosen primitive $f: E \to \R.$

    We now prove that this complex does not depend on the 1-form $\alpha$ representing $u$. Let $\alpha +dh$ be another representative of $u$ and $f: E \to \R$ be a primitive of $\pi^* \alpha$. Then, $g = f + h \circ \pi - h(\star): E \to \R$ is a primitive of $\pi^*(\alpha + dh).$

    The map $h \circ \pi - h(\star): E \to \R$ factors through $h - h(\star): X \to \R$ which is bounded since $X$ is compact. Let $K \geq 0$ such that $|h|\leq K.$
     Therefore $\Id: E \to E$ induces maps $E_{c,f} \to E_{c+K,g}$ and $ E_{c,g} \to E_{c+K,f}$ for every $c \in \R$. Lemma \ref{lemme: limite proj homotopy equivalence} implies that $C_*(E,f) \cong C_*(E,g)$ and this homotopy equivalence induces $C_*(F,f) \cong C_*(F,g).$
\end{proof}

Since $f : E \to \R$ is compatible with the action of the transitive lifting function \eqref{eq : primitive f}, we prove that the $C_*(\Omega X)$-module structure over $C_*(F)$ extends to a $C_*(\Omega X,u)$-module structure over $C_*(F,u).$
\begin{lemme}
    A transitive lifting function $\Phi: E \ftimes{\pi}{\ev_0} \mathcal{P} X \to E$ induces a $C_*(\Omega X,u)$-module structure on $C_*(F,u)$.
\end{lemme}

\begin{proof}
    Recall that, if $\sigma \in C_i(F)$ and $\omega \in C_j(\Omega X)$, $\sigma \cdot \omega:=\Phi_*(\sigma, \omega) \in C_{i+j}(F)$.
    From \eqref{eq : primitive f}, it follows that, $$\forall \sigma \in C_i(F),\omega \in C_j(\Omega X) \ f\circ (\sigma \cdot \omega) = f\circ \sigma + u(\omega).$$ 

    Define the $C_*(\Omega X,u)$-right module structure over $C_*(F,f)$ by 

    $$\sum_i n_i \sigma_i \cdot \sum_j m_j \omega_j = \sum_{i,j} n_im_j \sigma_i \cdot \omega_j,$$
    where $\sum n_i \sigma_i \in C_*(F,f)$ and $\sum m_j \omega_j \in C_*(\Omega X,u)$. For any $c \in \R$, 
    $$\#\{(i,j), \ f\circ(\sigma_i \cdot \omega_j) \geq c\} \leq \# \left\{i,\ f\circ \sigma_i \geq c - \max_j u(\omega_j)\right\} < \infty.$$ 

    Indeed, since for every $c \in \R$, $\#\{j, \ u(\omega_j) \geq c\} < + \infty$, then $\displaystyle \max_j u(\omega_j) $ is well-defined.
    Therefore $\sum_i n_i \sigma_i \cdot \sum_j m_j \omega_j \in C_*(F,f).$
\end{proof}

We can now state the main result of this section:

\begin{thm}[Morse-Novikov Fibration Theorem]\label{thm: Fibration Theorem Morse Novikov}
    Let $\Xi$ be a set of DG Morse-Novikov data on $X$. Let $F \hookrightarrow E \to X$ be a (Hurewicz) fibration over $X$ such that $\pi^*u = 0$ or equivalently, let $F \hookrightarrow E \to \hat{X}_u$ be a fibration. Let $\Phi: E \ftimes{\pi}{\ev_0} \mathcal{P}X \to E$ be a transitive lifting function for this fibration and endow $C_*(F,u)$ with the associated $C_*(\Omega X,u)$-module structure. Then, there exists a quasi-isomorphism

    $$\Psi_E: C_*(X,\Xi,C_*(F,u)) \to C_*(E,u).$$
\end{thm}

\subsection{Proof of the theorem}

Fix $\Xi = (\alpha,\xi,s^g_{x,y},o,\Y, \theta)$ a set of DG Morse-Novikov data. Since, up to chain homotopy equivalence, the complex does not depend on $\Xi$, it suffices to prove the theorem for a well-chosen $\Xi$.

We will use a similar approach as \cite[Theorem 7.2.1]{BDHO23}.
For a well-chosen Morse 1-form $\alpha$ (see Section \ref{subsection : Latour Trick}), we will define a morphism $$\Psi_E: C_*(X,\Xi,C_*(F,u)) \to C_*(E,u)$$ by lifting the Latour cells associated with $\alpha$, and we will prove that it induces a morphism of spectral sequence and an isomorphism between the second page $E^2_{p,q}$ of the spectral sequence defined in Section \ref{subsection: suite spectrale M-N} that converges towards $H_*(X,C_*(F,u))$ and the second page of a Leray-Serre-like spectral sequence $\mathcal{E}^2_{p,q}$ that converges towards $H_*(E,u)$.

\subsubsection{Definition of \texorpdfstring{$\Psi_E$}{psiE}}

\paragraph{Pulling back by $\theta$.}

Consider $F' \hookrightarrow E' \overset{\pi}{\to} \Xq$ the pull-back fibration of $F \hookrightarrow \fibration$ by a homotopy inverse $\theta: \Xq \to X$ of the canonical projection $p : X \to \Xq$. By definition, $$E' = \Xq \ftimes{\theta}{\pi} E.$$ This fibration is naturally endowed with the transitive lifting function $$\deffct{\Phi'}{E' \ftimes{\pi'}{\ev_0} \mathcal{P}(\Xq)}{E'}{((y,e),\gamma)}{(\ev_1(\gamma), \Phi(e, \theta \circ \gamma))}.$$

Recall that we denoted $\Gamma_{x,y}: \trajb{x,y} \to \mathcal{P}_{x \to y} X$ the parametrization map defined in Lemma \ref{lemme: parametrization map}. 
For $x,y \in \Crit(\alpha)$ and $g \in \pi_1(X)$, define 

$$q'_{x,y}= p \circ \Gamma_{x,y}:\trajb{x,y} \to \Omega(\Xq)$$
and 
$$m'^g_{x,y} = q'_{x,y,*}(s^g_{x,y}) \in C_{|x|-|y|-1}(\Omega (\Xq))$$ such that the Barraud-Cornea twisting cocycle associated with $\Xi$ satisfies $m^g_{x,y} = \theta_*(m'^g_{x,y}) $. Let $$m'_{x,y} = \sum_{g \in \pi_1(X)} m'^g_{x,y} \in C_*(\Omega (\Xq), u')$$ where we denoted $u' = \theta^*u \in H^1(\Xq)$. Let $f' : E' \to \R$, $f'(y,e) = f(e)$ and define
\begin{align*}
    C_k(E',u') & := \left\{ \sum_{i \geq 0} n_i \sigma_i , \ n_i \in \Z, \ \sigma_i \in C^0([0,1]^k,E'), \ \forall c\in \R,  \#\left\{i, \ \max_{t \in [0,1]^k} f' \circ \sigma_i(t) > c\right\} < +\infty\right\} \\
    & \cong \limproj C_k(E',E'_{c})
\end{align*}

and

\begin{align*}
    C_k(F',u') &:= \left\{ \sum_{i \geq 0} n_i \sigma_i , \ n_i \in \Z, \ \sigma_i \in C^0([0,1]^k,F'), \ \forall c\in \R,  \#\left\{i, \ \max_{t \in [0,1]^k} f' \circ \sigma_i(t) > c\right\} < +\infty\right\} \\
    &\cong \limproj C_k(F',F'_{c}),
\end{align*}
    
where, for $c\in \R$, we omitted the mention of the function $f : E \to \R$ and we denoted $E'_c = \theta^{-1}(E_c)$ and $F'_c = \theta^{-1}(F_c)$.

 \begin{defi}
   Define the complex $$C_*(X,\Xi,C_*(F',u')) = C_*(F',u') \otimes_{\Z} \Z\Crit(\alpha)$$ with the differential 
$$\partial(\sigma \otimes x) = \partial \sigma \otimes x + (-1)^{|\sigma|} \sum_{y} \sigma \cdot m'_{x,y} \otimes y.$$  
 \end{defi}

\begin{lemme}
\begin{enumerate}
    \item The map $\theta: \Xq \to X$ induces a  chain homotopy equivalence $C_*(E,u) \to C_*(E',u').$
    \item The complexes $C_*(X,\Xi,C_*(F,u))$ and $C_*(X,\Xi, C_*(F',u'))$ are chain homotopy equivalent.
\end{enumerate}
\end{lemme}

\begin{proof}
  \textbf{1.}  We will use Lemma \ref{lemme: limite proj homotopy equivalence}, that we state here, to prove the first point. 

  \begin{lemme}[Lemma \ref{lemme: limite proj homotopy equivalence}]
    Let $Y,Z$ be topological spaces each filtered respectively by $(Y_c)_{c\in \R}$ and $(Z_c)_{c\in \R}$ i.e. if $c<c'$, then $Y_c \subset Y_{c'}$ and $Z_c \subset Z_{c'}$. Consider the projective systems $$\left(C_*(Y, Y_{c}), \pi^Y_{c,c'}\right) \textup{ and }\left(C_*(Z, Z_{c}), \pi^Z_{c,c'}\right),$$ where $\pi^Y_{c,c'}: C_*(Y,Y_c) \to C_*(Y,Y_{c'})$ and  $\pi^Z_{c,c'}: C_*(Z,Z_c) \to C_*(Z,Z_{c'})$   are defined by the canonical projections for $c<c'$.
Let $\phi: Y \to Z$ and $\psi: Z \to Y$ be homotopy equivalences which are inverses of each other, and denote $$F_Z : [0,1] \times Z, \ F_Z(0,\cdot) = \phi \circ \psi, \ F_Z(1, \cdot) = \Id_Z \textup{ and } F_Y : [0,1] \times Y, \ F_Y(0,\cdot) = \psi \circ \phi, \ F_Y(1, \cdot) = \Id_Y$$  associated homotopies. Assume that there exists a constant $K>0$ such that

$$\forall c \in \R, \ \phi(Y_c) \subset Z_{c+K},  \psi(Z_c) \subset Y_{c+K}, \ F_Y( \cdot ,Y_c) \subset Y_{c+2K} \textup{ and } F_Z(\cdot,Z_c) \subset Z_{c+2K}.$$

    Then $\phi_*$ and $\psi_*$ induce chain homotopy equivalences $$\xymatrix{
    \limproj C_*(Y,Y_c) \ar@<2pt>[r]^{\phi_*} &  \limproj C_*(Z,Z_c) \ar@<2pt>[l]^{\psi_*}}.$$ 

\end{lemme}

  Denote $H: X \times [0,1] \to X$ a homotopy between $\Id = H(\cdot,0)$ and $\theta \circ p = H(\cdot,1)$ and denote $H': \Xq \times [0,1] \to \Xq$ a homotopy between $\Id = H'(\cdot, 0)$ and $p\circ \theta = H'(\cdot,1)$. The map $\theta: \Xq \to X$ induces an homotopy equivalence $$\deffct{\tilde{\theta}}{E'}{E}{(y,e)}{(\theta(y),e)}.$$
    
     The map $$\deffct{W}{E \times [0,1]}{E}{(e,t)}{\Phi(e, H(\pi(e),\cdot) \lvert_{[0,t]})}$$ that lifts $ H(\pi,\cdot): E \times [0,1] \to X$ is a  homotopy between $\Id_E = W(\cdot,0)$ and $\tilde{\theta} \circ \tilde{p} = W(\cdot,1),$ where $$\deffct{\tilde{p}}{E}{E'}{e}{\left(p(\pi(e)), \Phi\left(e,H(\pi(e),\cdot)\right)\right)}$$ is a right homotopy inverse of $\tilde{\theta}$. It is actually known that $\tilde{p}$ is a homotopy inverse of $\tilde{\theta}$ (see \cite[Proposition 6.59]{Ja84}). However, the homotopy $W' : E' \times [0,1] \to E'$  between $\Id_{E'} = W'(\cdot,0)$ and $\tilde{p} \circ \tilde{\theta} = W'(\cdot,1)$ is not constructed explicitly. In our case, we can give an explicit definition of $W'.$\\
    
    Let $U \subset X$ be a contractible neighborhood of the tree $\mathcal{Y}$ and denote $U' = p(U) \subset \Xq$. We can and will assume that $\theta(U') \subset U$, that the homotopy $H : X \times [0,1] \to X$ is constant equal to the identity outside of $U$ and the homotopy $H' : \Xq \times [0,1] \to \Xq$ is constant equal to the identity outside of $U'$. In other words, $$\forall x \notin U, \forall t \in [0,1] \ H(x,t) = x$$ and $$\forall y \notin U', \forall t \in [0,1] \ H'(y,t) = y.$$

    Let $y \in \Xq$. There exists a homotopy with fixed endpoints $Q_y : [0,1] \to \mathcal{P}_{\theta(y) \to \theta \circ p \circ \theta(y)}$ between the paths $$\theta(H'(y,\cdot)) \in \mathcal{P}_{\theta(y) \to \theta \circ p \circ \theta(y)} X$$ and $$H(\theta(y),\cdot)  \in \mathcal{P}_{\theta(y) \to \theta \circ p \circ \theta(y)} X.$$ 

    Let $\tau_y : [0,1] \to \mathcal{P}X$ defined by $$\tau_y(t) = Q_y(\cdot)(t).$$ Remark that for any $t \in [0,1]$, $\tau_y(t) \in \mathcal{P}_{H(\theta(t),s) \to \theta(H'(y,t))} X$. Moreover, $\tau_y(0)$ is the constant path at $\theta(y)$ and $\tau_y(1)$ is the constant path at $\theta \circ p \circ \theta(y)$.

    \begin{figure}[h!]
        \centering
        \includegraphics[width=0.7\linewidth, scale = 1.2]{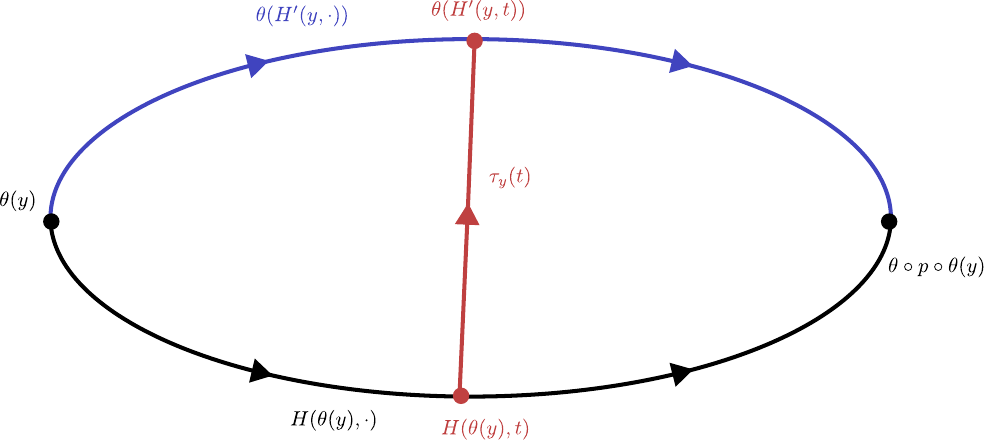}
        \caption{The path $\tau_y(t)$}
        \label{fig:enter-label}
    \end{figure}

    Hence
    $$\deffct{W'}{E' \times [0,1]}{E'}{((y,e),t)}{\left(H'(y,t), \Phi(e, H(\theta(y), \cdot)\lvert_{[0,t]} \# \tau_y(t))\right)}$$ is a homotopy between $\Id_{E'} = W'(\cdot,0)$ and $\tilde{p} \circ \tilde{\theta} = W'(\cdot,1).$

    Let $c \in \R$. By definition $\tilde{\theta}(E'_c) \subset E_c$.
    We now prove that there exists $K>0$ such that for all $t\in [0,1]$, $\tilde{F}(E_c,t) \subset E_{c+K}$.
   
    This is equivalent to finding a bound for the map $$h := f\circ pr_1 - f\circ W: E \times [0,1] \to \R.$$ We will prove that $h(e,t)$ only depends on $(\pi(e),t) \in X \times [0,1]$.  The compactness of $X \times [0,1]$ will allow to conclude.
    
    Let $(e,e') \in E \ftimes{\pi}{\pi} E$, $x= \pi(e)$ and $t \in [0,1]$. Let $\gamma \in \mathcal{P}_{e \to e'}E$ and denote $g =[\pi\circ \gamma] \in \pi_1(X).$ Since $W(\gamma,t)$ is a path between $W(e,t)$ and $W(e',t)$ that belong to the same fiber $\pi^{-1}(H(x,t))$,  \begin{align*}
        h(e,t) - h(e',t) &= f(e) - f(e') - (f\circ W(e,t) - f\circ W(e',t))\\
        & = u(g) - u([\pi\circ W( \gamma,t)])\\
        &= u(g) - u([H(\pi \circ \gamma,t)]) \\
        &=  u(g) - u([H(\pi \circ \gamma,0)]) \\
        &= 0.
    \end{align*}
    
For $t=1$, we obtain $\tilde{\theta} \circ \tilde{p}(E_c) \subset E_{c+K} \Leftrightarrow \tilde{p}(E_c) \subset E'_{c+K}$. We prove in the same manner that there exists $K'>0$ such that $W'(E'_c, \cdot) \subset E'_{c+K'}$.\\

\textbf{2.}  The homotopy equivalence $C_*(E,u) \simeq C_*(E',u')$ restricts to a homotopy equivalence $C_*(F,u) \simeq C_*(F',u')$.
Moreover, it preserves the module structures. Indeed, for all $\sigma \in C_*(F')$ and $\omega \in C_*(\Omega(\Xq))$,
\begin{align*}
    \tilde{\theta}(\sigma \cdot \omega) & = \tilde{\theta}_*(\Phi'_*(\sigma \otimes \omega)) \\
    & = \Phi_*(\tilde{\theta}(\sigma) \otimes \theta_* ( \omega)) \\
    &= \tilde{\theta}(\sigma) \cdot \theta_*(\omega)
\end{align*}

since $\Phi(\tilde{\theta},\theta) = \tilde{\theta} \circ \Phi': E' \ftimes{\pi'}{\ev_0} \mathcal{P}(\Xq) \to E.$
This equality extends by linearity to $\sigma \in C_*(F',u')$ and $\omega \in C_*(\Omega(\Xq),u').$

The map $\deffct{\Theta}{C_*(X,\Xi,C_*(F',u'))}{C_*(X,\Xi,C_*(F,u))}{\sigma \otimes x}{\tilde{\theta}(\sigma) \otimes x}$ is therefore a chain homotopy equivalence.

\end{proof}

We will define $\Psi_E: C_*(X,\Xi,C_*(F,u)) \to C_*(E,u)$ as a composition $$C_*(X,\Xi,C_*(F,u)) \simeq C_*(X,\Xi,C_*(F',u')) \overset{\Psi}{\to} C_*(E',u') \simeq C_*(E,u),$$

where the first and the last maps are given by the previous lemma.

\paragraph{Latour cells.}\label{para : Latour cells}

Let $h : X \to \R$ be a Morse function. Since for a generic 1-form $\alpha$, the Latour cells $\left(\wb{u}{x}\right)_{x \in \Crit(\alpha)}$ are not compact and do not cover $X$, we will take $\Xi$ a set of DG Morse-Novikov data given by the Latour Trick \ref{prop : Latour Trick}.
Since the orbits of the pair $(\alpha,\xi) \in \Xi$ are those of the Morse function $h$, $\left(\wb{u}{x}\right)_{x \in \Crit(\alpha)}$ form a cellular decomposition of $X$ (see \cite[Section 4.9]{AD14}. We then proceed as in \cite[Section 7.3]{BDHO23} to define $$m_x \in C_{|x|}(\mathcal{P}_{\star \to \Xq} \Xq)$$ for each $x \in \Crit(h) = \Crit(\alpha)$ such that $$\partial m_x = \sum_{y} m'_{x,y} \cdot m_y$$ and subsequently, the morphism $\Psi : C_*(X,\Xi, C_*(F',u')) \to C_*(E',u').$

\begin{lemme}\cite[Lemma 7.3.2]{BDHO23}
   There exists a family  $\left\{ s_x \in C_{|x|}(\wb{u}{x}), \ x \in \Crit(\alpha) \right\}$ such that:

\begin{enumerate}
     \item each $s_x$ is a cycle relative to the boundary and represents the fundamental class $[\wb{u}{x}]$.
     \item each $s_x$ satisfies $$\partial s_x = \sum_y s_{x,y} \cdot s_y.$$
\end{enumerate}
\end{lemme}
\begin{flushright}
    $\blacksquare$
\end{flushright}

\begin{defi}
    Such a family is called a \textbf{representing chain system for the Latour cells} adapted to the representing chain system $\left\{s_{x,y} \in C_{|x|-|y|-1}(\trajb{x,y}), \ x,y \in \Crit(\alpha) \right\}$.
\end{defi}

For convenience, we can and will assume that $\Crit_0(\alpha)$ is the singleton $\{\star\}$ and choose $\Y$ to be a tree rooted in $\star$ whose branches are gradient lines. For $x \in \Crit(\alpha)$, define $l_x \in \mathcal{P}_{x \to \star} X$ to be the branch of $\Y$ linking $x$ to $\star.$ 

The attachment maps $i_x: \wb{u}{x} \to X$ induce a cellular decomposition $$j_x: \wb{u}{x}/l_x \to \Xq$$ for which the canonical projection $p: X \to \Xq$ is cellular, \emph{i.e.} $p\circ i_x = j_x \circ p_x,$ where $p_x : \wb{u}{x} \to \wb{u}{x}/l_x$.

Recall that $i_x(a) = a$ if $a \in W^u(x)$ and, if $a=(\lambda,b) \in \partial \wb{u}{x}$, then $i_x(a) = b.$

\begin{lemme}\label{lemme : defi of parametrization maps for Latour cells}
    There exists a family of \textbf{parametrization maps} $$\left\{\Gamma_x: \wb{u}{x} \to \mathcal{P}_{x \to X} X, \ x \in \Crit(\alpha) \right\}$$ such that, for each $x \in \Crit(\alpha)$,  $$\ev_1\circ \Gamma_x = i_x \textup{ and } \Gamma_x(\lambda,a) = \Gamma_{x,y}(\lambda) \# \Gamma_y(a)$$ if $(\lambda,a) \in \trajb{x,y} \times W^u(y) \subset \partial \wb{u}{x}.$
\end{lemme}

\begin{proof}
    Recall the notation $\tilde{h}: \tilde{X} \to \R$ for a primitive of $\tilde{\pi}^*\alpha \in \tilde{\pi}^*u \in H^1(\tilde{X},\R) = \{0\},$ where $\tilde{\pi}: \tilde{X} \to X$ is the universal cover of $X$.    
    Let $a \in \wb{u}{x}$. If $a \in W^u(x)$, denote $\gamma: (-\infty,0] \to X$ the half-infinite flow line going from $x$ to $a$. In this case, denote  $$\Gamma_x(a):= \tilde{\pi} \circ \tilde{\Gamma}_x(a) \in \mathcal{P}_{x \to a} W^u(x),$$
    where $\tilde{\Gamma}_x(a)$ is the parametrization of $\tilde{\gamma} \in \mathcal{P}_{\tilde{x} \to \tilde{X}} \tilde{X}$, the lift of $\gamma$ in $\tilde{X}$ by the values of $\tilde{h}: \tilde{X} \to \R$.
    
    Since the parametrization map $\Gamma_{x,y} : \trajb{x,y} \to \mathcal{P}_{x \to y} X$ defined in Lemma \ref{lemme: parametrization map} is already defined using the values of $\tilde{h}$, we can define  $$\Gamma_x(a) = \Gamma_{x,y}(\lambda) \# \Gamma_y(b) \in \mathcal{P}_{x \to b} \wb{u}{x}$$ if $a= (\lambda,b) \in \trajb{x,y} \times W^u(y)$.

\end{proof}

\begin{rem}
    If $a \in \wb{s}{x}$, we define $\Gamma_x(a) \in \mathcal{P}_{a \to x}X$ in the same way considering the half-infinite flow line going from $a$ to $x$. We will also denote by $\Gamma_x: \wb{s}{x} \to \mathcal{P}_{X \to x} X$ this parametrization map.
\end{rem}

\begin{lemme}\label{lemme: def q Latour cell}\cite[Lemma 7.3.3]{BDHO23}
     The family of continuous maps $$q_x = p \circ \Gamma_x: \wb{u}{x} \to \mathcal{P}_{\star \to \Xq} \Xq$$ satisfies:
    \begin{enumerate}
        \item For any $(\lambda,a) \in \trajb{x,y} \times \wb{u}{y} \subset \partial \wb{u}{x}$, we have $$q_x(\lambda,a) = q'_{x,y}(\lambda) \# q_y(a).$$
        \item  $\ev_1 \circ q_x = j_x \circ p_x = p \circ i_x$.
    \end{enumerate}
\end{lemme}
\begin{flushright}
    $\blacksquare$
\end{flushright}

For each $x \in \Crit(\alpha)$, denote $$m_x = q_{x,*}(s_x) \in C_{|x|}(\mathcal{P}_{\star \to \Xq} \Xq).$$

This family satisfies \begin{equation}
    \partial m_x = \sum_{y} m'_{x,y} \cdot m_y.
\end{equation}

\begin{lemme}
    The map $$\deffct{\Psi}{C_*(X,\Xi,C_*(F',u'))}{C_*(E',u')}{\sum_i n_i\sigma_i \otimes x}{\sum_i n_i \Phi'_*(\sigma_i \otimes m_x)}$$ is well-defined and is a morphism of complexes.
\end{lemme}
\begin{proof}
    First, we prove that $\Psi$ is well-defined. Let  $x \in \Crit(\alpha)$ and $\sigma = \sum_i n_i \sigma_i \in C_*(F',u')$, \emph{i.e.} for all $c \in \R$, $$\#\left\{ i, \ n_i \neq 0 \textup{ and } f'\circ \sigma_i > c\right\} < \infty .$$ To prove that  for all $c \in \R$, $$\#\left\{ i, \ n_i \neq 0 \textup{ and } f'\circ \Phi'_*(\sigma_i \otimes m_x) > c\right\} < \infty,$$ we prove that there exists $A>0$ such that if $a \in F_c$ and $\gamma$ is a parametrized piece of a (broken) trajectory in a cell $\wb{u}{x}/l_x$, then $\Phi'(a,\gamma) \in E_{c+A}$. This is equivalent to the existence of a bound for the map 

$$\begin{array}{ccc}
    F' \times q_x\left(\wb{u}{x}\right) &\to& \R  \\
    (v,\gamma) & \mapsto & f'(\Phi'(v,\gamma)) - f'(a) = \int_{\theta\circ\gamma}\alpha. 
\end{array}$$

Since this map only depends on $\gamma \in q_x \left(\wb{u}{x}\right)$ and $q_x \left(\wb{u}{x}\right)$ is compact, $\Psi$ is well-defined.

We now prove that $\Psi$ is a morphism of complexes.
Let $\sigma \otimes x \in C_*(X,\Xi,C_*(F',u'))$. Then:
\begin{align*}
    \partial \Psi(\sigma \otimes x) & = \Phi'_{*}(\partial(\sigma \otimes m_x))\\
    & = \Psi( \partial \sigma \otimes x) + (-1)^{|\sigma|} \Phi'_{*}\left(\sigma \otimes \sum_y m'_{x,y}\cdot m_y\right)\\
    & \overset{(1)}{=}  \Psi( \partial \sigma \otimes x) + (-1)^{|\sigma|} \sum_y \Phi'_*( \Phi'_{*}(\sigma \otimes m'_{x,y}) \otimes m_y)\\
    & \overset{(2)}{=}\Psi( \partial \sigma \otimes x) + (-1)^{|\sigma|} \sum_y \Phi'_{*}(\sigma \cdot m'_{x,y} \otimes m_y)\\
    & = \Psi\left(\partial \sigma \otimes x + (-1)^{|\sigma|}  \sum_y \sigma \cdot m'_{x,y} \otimes y \right)\\
    & = \Psi(\partial(\sigma \otimes x)).
\end{align*}

Equality (1) comes from the fact that $\Phi'$ is a transitive lifting function and Equality (2) is a consequence of the definition of the $C_*(\Omega(\Xq))$-module structure on $C_*(F',u').$

\end{proof}

The proof of Theorem \ref{thm: Fibration Theorem Morse Novikov} is now reduced to showing that $\Psi: C_*(X,\Xi,C_*(F',u')) \to C_*(E',u')$ is a quasi-isomorphism. 
To do that, we will prove that $\Psi$ induces a morphism of spectral sequences that is an isomorphism on the second pages.

\subsubsection{Spectral sequences}

For each $p \in \{0, \dots,n\}$, denote by $ Sk_p =  \displaystyle \bigcup_{|x|\leq p} j_x\left(\wb{u}{x}/l_x\right)$, the $p$-skeleton of $\Xq$.
Consider the following filtrations:

\[
F_p(C_*(X,\Xi,C_*(F',u')))= \bigoplus_{\substack{i+j =k \\ j \leq p}} C_i(F',u') \otimes \Z\Crit_j(\alpha)
\]
and
\[
F_p(C_*(E',u')) = C_*(\pi^{-1}(Sk_p),u'):= \limproj C_*(\pi^{-1}(Sk_p), \pi^{-1}(Sk_p) \cap E_c).
\]

The two spectral sequences associated with these filtrations, $(E^r_{p,q})$ and $(\mathcal{E}^r_{p,q})$, converge towards $ H_*(X,C_*(F',u')) $ and $ H_*(E',u')$ respectively and their first pages are given by

\[
E^1_{p,q} = H_q(F',u') \otimes \Z\Crit_p(\alpha) = H_{q}(\limproj C_*(F',F'_c)) \otimes \Z\Crit_p(\alpha)
\]
and
\begin{align*}
    \mathcal{E}^1_{p,q} &= H_{p+q}\left(\pi^{-1}(Sk_p), \pi^{-1}(Sk_{p-1}),u'\right) \\
    &:= H_{p+q} \left( \faktor{C_*(\pi^{-1}(Sk_p),u')}{C_*(\pi^{-1}(Sk_{p-1}),u')} \right)\\
    &= H_{p+q}\left( \frac{\limproj C_*\left(\pi^{-1}(Sk_p), \pi^{-1}(Sk_p) \cap E'_c \right)}{\limproj C_*\left( \pi^{-1}(Sk_{p-1}), \pi^{-1}(Sk_{p-1}) \cap E'_c \right)}  \right)\\
    &\overset{(*)}{=} H_{p+q}\left( \limproj \left[ \frac{ C_*\left(\pi^{-1}(Sk_p), \pi^{-1}(Sk_p) \cap E'_c \right)}{ C_*\left( \pi^{-1}(Sk_{p-1}), \pi^{-1}(Sk_{p-1}) \cap E'_c \right)} \right]  \right)\\
    &= H_{p+q}\left( \limproj \left[ \frac{C_*\left(\pi^{-1}(Sk_p), \pi^{-1}(Sk_{p-1}) \right)}{C_*\left(\pi^{-1}(Sk_p) \cap E'_c, \pi^{-1}(Sk_{p-1}) \cap E'_c \right)} \right]\right).
\end{align*}

The equality $(*)$ is a consequence of Lemma \ref{lemme: limproj d'un quotient de sys proj} that we state here:

\begin{lemme}[Lemma \ref{lemme: limproj d'un quotient de sys proj}]
    Let $((B_i)_{i \in\Z_{\leq 0}}, \pi^B_{i,j})$ be a subsystem of a projective system $((A_i)_{i \in\Z_{\leq 0}}, \pi^A_{i,j})$ i.e. $$\forall i\leq j , \ B_i \subset A_i \textup{ and } \pi^B_{i,j}(B_i) = \pi^A_{i,j}(B_i) \subset B_j.$$ Assume that $\pi^B_{i,j} = \pi^{A}_{i,j} \lvert_{B_i}$ is surjective for all $i\leq j$.  Then $\pi^A$ induces a projection on $\left(\faktor{A_i}{B_i}\right)_{i \in \Z_{\leq 0}}$ such that 
    $$\limproj \faktor{A_i}{B_i} \cong \faktor{\limproj A_i}{\limproj B_i}.$$
\end{lemme}

We now prove that $\Psi : C_*(X,\Xi,C_*(F',u')) \to C_*(E',u')$ induces a quasi-isomorphism $\Psi^{(1)} : E^1_{p,q} \to \mathcal{E}^1_{p,q}$.

\begin{prop}\label{Prop: qiso cellules into Sk Novikov}
Consider $(\gamma_{x, \hat{a}})_{\hat{a} \in \wb{u}{x}/l_x}$, a continuous family of paths in $\Xq$ connecting $\star$ to $j_x(\hat{a}) \in \Xq$.
The maps $$
\deffct{\chi_x}{(F' \times \wb{u}{x}/l_x , F'  \times \partial(\wb{u}{x}/l_x)) }{(\pi^{-1}(Sk_{|x|}), \pi^{-1}(Sk_{|x|-1}))}{(v,\hat{a})}{\Phi'(v, \gamma_{x, \hat{a}})}
$$
induce an isomorphism $ \bigoplus_{|x|=p} H_q(F',u') \otimes H_p\left( \wb{u}{x}/l_x, \partial \left ( \wb{u}{x} / l_x \right ) \right) \simeq  \mathcal{E}^{1}_{p,q}.$

\end{prop}

\begin{proof}
We first show that these maps induce an isomorphism  $$\displaystyle \bigoplus_{|x|=p} H_{q}(F') \otimes  H_{p}\left( \wb{u}{x}/l_x, \partial \left( \wb{u}{x} / l_x \right)\right) \simeq H_{p+q} (\pi^{-1}(Sk_p), \pi^{-1}(Sk_{p-1}))$$ and prove in a second step that this isomorphism behaves well with respect to the filtrations.\\

\textbf{\underline{Step 1:}} \ Let $x \in \Crit(\alpha)$.

$\bullet$ We first consider smaller cells $D_x \subset \wb{u}{x}/l_x$ so that $j_x\lvert_{D_x}$ is a homeomorphism onto its image.

Consider $D\subset C$, two $p$-disks centered at $o$, $U = C \setminus \overset{\circ}{D}$, a neighborhood of $\partial C$ in $C$ that retracts to $\partial C$, and the family of straight lines, denoted $(\delta_a)_{a \in C}$, that connects $o$ to $a \in C$.  \\

\begin{figure}[h!]
    \centering
    \includegraphics[scale = 0.3]{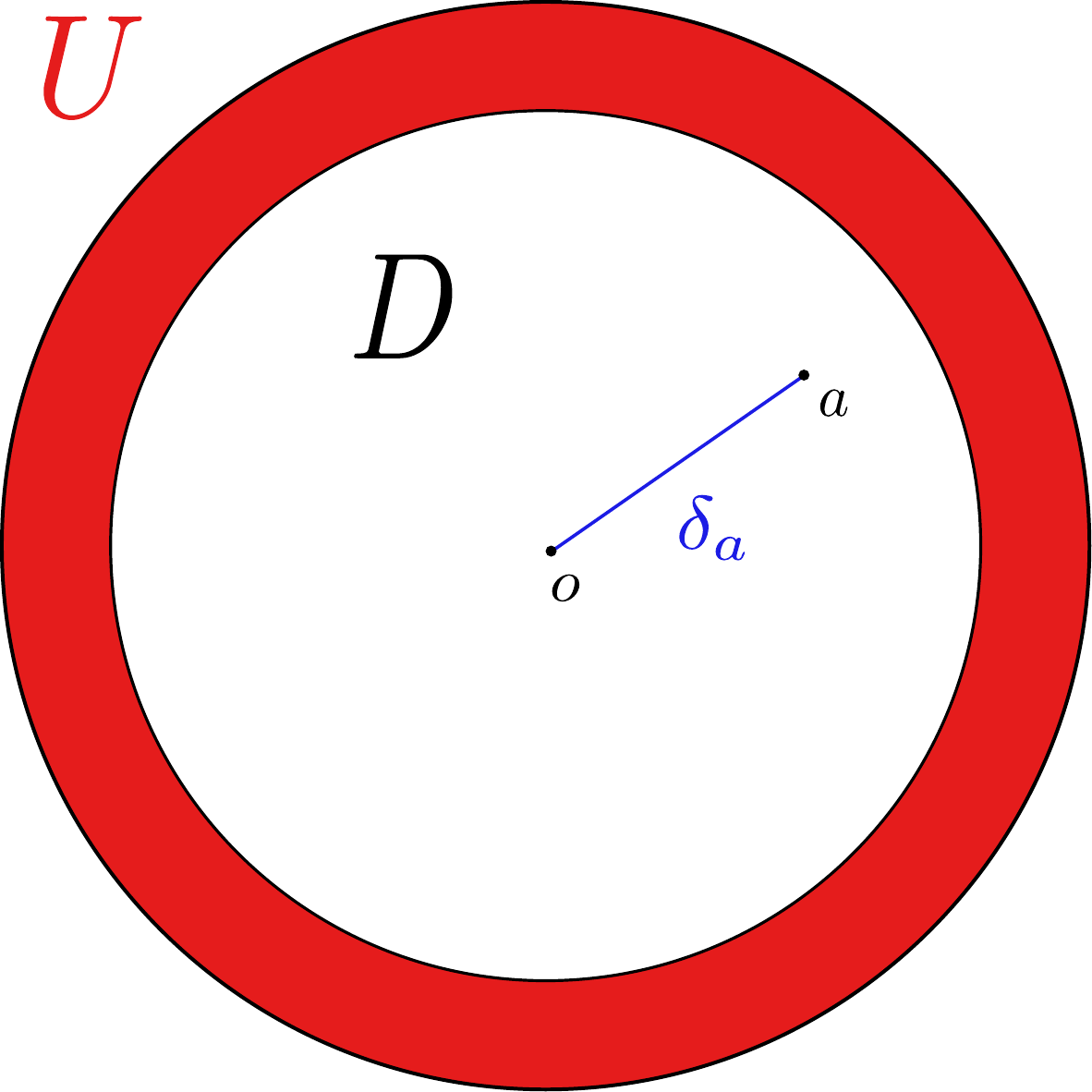}
    \caption{The cell $D$, neighborhood $U$ and the path $\delta_a$.}
    \label{fig: cell}
\end{figure}

Since $U$ can be retracted onto $\partial C$ and onto $\partial D$, $$H_*(C, \partial C) = H_*(C,U) = H_*(D, \partial D).$$

We use the following notation:
\begin{itemize}
    \item $\tau_x$ is a homeomorphism transforming $C$ into the cell $\wb{u}{x}/l_x$.
    \item $C_x = \tau_x(C) = \wb{u}{x}/l_x$.
    \item $D_x = \tau_x(D)$.
    \item $F'_x = \pi^{-1}(c_x)$ the fiber above $c_x:= j_x(\tau_x(o))$.
    \item $U_x = \tau_x(U).$
    \item $\hat{\delta}_{x,\hat{a}} = j_x \circ \tau_x \circ \delta_{\tau_x^{-1}(\hat{a})}$, a continuous family of paths connecting $c_x$ to $j_x(\hat{a}) \in j_x(C_x)$ within $j_x(C_x)$.
\end{itemize}

We consider the maps

$$\deffct{\chi_x^D}{(F'_x \times D_x, F'_x  \times \partial D_x)}{ (\pi^{-1}(j_x(D_x)), \pi^{-1}( j_x( \partial D_x))) }{(v,\hat{a})}{\Phi'(v,\hat{\delta}_{x,\hat{a}})} $$

and 

$$\deffct{\psi_x^D}{ (\pi^{-1}(j_x(D_x)), \pi^{-1}( j_x( \partial D_x))) }{(F'_x \times D_x, F'_x \times \partial D_x)}{y}{\left(\Phi'\left(y, \hat{\delta}^{-1}_{x, j_x^{-1}\circ\pi(y)}\right), j_x^{-1}(\pi(y))\right)}.$$

Since $j_x$ is a homeomorphism on $D_x$, its inverse $j_x^{-1}$ is well-defined in this case.
These two maps are homotopy inverses of each other: 
$$\chi_x^D \circ \psi_x^D: y \mapsto \Phi'\left(\Phi'\left(y, \hat{\delta}^{-1}_{x, j_x^{-1}(\pi(y))}\right), \hat{\delta}_{x, j_x^{-1}(\pi(y))}\right) = \Phi'\left(y, \hat{\delta}^{-1}_{x, j_x^{-1}(\pi(y))}\# \hat{\delta}_{x, j_x^{-1}(\pi(y))}\right)  \simeq \Id$$
and $$\psi_x^D \circ \chi_x^D: (v,\hat{a}) \mapsto \left(\Phi'\left(\Phi'\left(v , \hat{\delta}_{x,\hat{a}}\right), \hat{\delta}^{-1}_{x, \hat{a}} \right), j_x^{-1} \circ \pi'\left( \Phi'\left(v , \hat{\delta}_{x,\hat{a}}\right)\right)\right) = \left(\Phi'\left(v, \hat{\delta}_{x,\hat{a}} \#  \hat{\delta}_{x,\hat{a}}^{-1} \right), \hat{a}\right) \simeq \Id. $$

These maps therefore induce chain homotopy equivalences.\\

• We now will consider $\chi_x$ on the whole cell $C_x = \wb{u}{x}/l_x$. Define $\widetilde{Sk}_{p-1} = Sk_{p-1} \cup \displaystyle \bigcup_{|x| = p} j_x(U_x)$ such that $\displaystyle \bigcup_{|x|=p} (j_x(C_x),j_x(U_x)) = (Sk_p, \widetilde{Sk}_{p-1})$. \\

It is clear that $\widetilde{Sk}_{p-1}$ retracts to $Sk_{p-1}$, therefore $H_p(Sk_p, \widetilde{Sk}_{p-1}) = H_p(Sk_p, Sk_{p-1})$.\\

We define the map 

\[
\deffct{\chi^{C}_x}{(F'_x \times C_x, F'_x \times U_x)}{ (\pi^{-1}(Sk_p), \pi^{-1}(\widetilde{Sk}_{p-1})) }{(v,\hat{a})}{\Phi'(v, \hat{\delta}_{x,\hat{a}})}.
\]

Since $j_x: C_x \to \Xq$ is not a homeomorphism, we do not have a homotopy inverse as for $\chi^D_x$. Nonetheless, we will prove that it induces a quasi-isomorphism.

Let $\gamma \in \mathcal{P}_{\star \to c_x} \Xq$. The map

\[
\deffct{\eta_x}{F'}{F'_x}{(\star,v)}{(c_x,\Phi'(v,\gamma))}
\]

is a homeomorphism with inverse 

\[
\deffct{\eta_x^{-1}}{F'_x}{F'}{(c_x,v_x)}{\left(\star,\Phi'\left((c_x,v_x),\gamma^{-1}\right)\right)}.
\]

We consider the continuous family of paths $ (\gamma_{x,\hat{a}})_{\hat{a} \in C_x}$ connecting $\star$ to $\hat{a}$ in $C_x$, given by $\gamma_{x,\hat{a}} = \gamma \# \hat{\delta}_{x, \hat{a}}$. We define $\chi_x: F' \times C_x$ as the composition 

\[ \chi_x:
\begin{array}{ccccc}
    F' \times C_x & \overset{\eta_x}{\to} & F'_x \times C_x & \overset{\chi^C_x}{\to} &  \pi^{-1}(Sk_p) \\
    (v, \hat{a}) & \mapsto & \left(\Phi'(v,\gamma), \hat{a}\right) & \mapsto & \Phi'(\Phi'(v,\gamma), \hat{\delta}_{x, \hat{a}}) = \Phi'(v,\gamma_{x,\hat{a}}).
\end{array}
\]

We then have the following commutative diagram: 
{\footnotesize
$$
\xymatrix{ 
&
\displaystyle \bigoplus_{|x|=p} H_{p+q}(F'_x \times D_x, F'_x \times \partial D_x) \ar@<2ex>[r]_-{\sim}^-{\chi_{x,*}^{D}} \ar[d]^-{\sim}
&
\displaystyle \bigoplus_{|x|=p} H_{p+q}(\pi^{-1}(j_x(D_x)), \pi^{-1}(j_x(\partial D_x))) \ar@<2ex>[l]^-{\psi_{x,*}^{D}}_-{\sim} \ar[d]^-{\sim} \\
&
\displaystyle \bigoplus_{|x|=p} H_{p+q}(F'_x \times C_x, F'_x \times U_x) \ar[r]^-{(\chi^{C}_{x,*})_{|x|=p}} 
&
H_{p+q} (\pi^{-1} (Sk_p), \pi^{-1}(\widetilde{Sk}_{p-1}))\\
\displaystyle \bigoplus_{|x|=p} H_{q}(F') \otimes  H_{p}( C_x,\partial C_x ) \ar[r]^-{\sim} 
&
\displaystyle \bigoplus_{|x|=p} H_{p+q}(F'_x \times C_x, F'_x \times \partial C_x ) \ar[r]^-{(\chi^{C}_{x,*})_{|x|=p}} \ar[u]_-{\sim}
&
H_{p+q}(\pi^{-1}(Sk_p) , \pi^{-1}(Sk_{p-1})), \ar[u]_-{\sim}
}$$
}

where all the vertical arrows are excisions and the horizontal arrow $$ \bigoplus_{|x|=p} H_{q}(F') \otimes  H_{p}( C_x,\partial C_x ) \to  \bigoplus_{|x|=p} H_{p+q}(F'_x \times C_x, F'_x \times \partial C_x )$$ is the identification given by the maps $\eta_x$ and the Universal Coefficient Theorem, noting that, if $|x|=p$, $H_k(C_x, \partial C_x) = \mathbb{Z}$ if $k=p$, and $0$ otherwise. The two squares on the right are commutative by definition of $\chi^D_x$ and $\chi^C_x$ at the topological level.
Therefore, $$(\chi_{x,*})_{|x|=p}: \bigoplus_{|x|=p} H_q(F') \otimes H_p(C_x, \partial C_x) \to H_{p+q}(\pi^{-1}(Sk_p),\pi^{-1}(Sk_{p-1}))$$ is an isomorphism. It now suffices to remark that, since $C_x$ is contractible, any other choice for the family of paths $\gamma_{x, \hat{a}}$ would yield a chain homotopic map at the complex level and therefore the same map in homology.\\

\textbf{\underline{Step 2:}} We now check that the maps used in Step 1 respect the filtrations. Let $x \in \Crit(\alpha)$ and $(v,\hat{a}) \in F'_{x} \times C_x$. For $c \in \R$, denote $F'_{x,c} = \eta_x(F'_c).$ The real number

$$f'(\chi^C_x(v,\hat{a})) - f'(v) = \int_{\delta_{x,\hat{a}}} \theta^*\alpha$$ does not depend on $v$ but only on $\hat{a} \in C_x$. Since $C_x$ is compact, there exists $M_1 > 0$ such that for all $c\in \R$, $$\chi_x^C(F'_{x,c} \times C_x) \subset \pi^{-1}(Sk_{|x|}) \cap E'_{c+M_1}.$$

In particular, $\chi^D_x(F'_{x,c} \times D_x) \subset \pi^{-1}(j_x(D_x)) \cap E'_{c+M_1}$ and similar arguments show that there exists $M_2>0$ such that for all $c \in \R$, $$\psi^D_x\left( \pi^{-1}(j_x(D_x)) \cap E'_{c}\right) \subset F'_{x,c+M_2} \times D_x.$$ Therefore $\chi_{x,*}^D$ and $\psi^D_{x,*}$ induce homotopy equivalence inverse of each other $$H_*\left(\limproj \left[ \frac{C_*(F'_x \times D_x, F'_x \times \partial D_x)}{C_*(F'_{x,c} \times D_x, F'_{x,c} \times \partial D_x)}\right] \right) \simeq H_*\left( \limproj \left[\frac{C_*(\pi^{-1}(j_x(D_x)),\pi^{-1}(j_x(\partial D_x)))}{C_*(\pi^{-1}(j_x(D_x)) \cap E'_c,\pi^{-1}(j_x(\partial D_x)) \cap E'_c)}\right] \right).$$

It follows that,  $$(\chi_{x,*})_{|x|=p}: \bigoplus_{|x|=p} H_q(F',u') \otimes H_p(C_x, \partial C_x) \to H_{p+q}\left( \limproj \left[ \frac{C_*\left(\pi^{-1}(Sk_p), \pi^{-1}(Sk_{p-1}) \right)}{C_*\left(\pi^{-1}(Sk_p) \cap E'_c, \pi^{-1}(Sk_{p-1}) \cap E'_c \right)} \right]\right) = \mathcal{E}^1_{p,q} $$ is well-defined and the same arguments as in the first step show that it is an isomorphism.

\end{proof}

\paragraph{End of proof of Theorem \ref{thm: Fibration Theorem Morse Novikov}.}

Let $x \in \Crit(\alpha)$ and $a \in \wb{u}{x}$. Consider the path $\gamma_{x,a} = \Gamma_x(a) \in \mathcal{P}_{x \to i_x(a)} X$. For $\hat{a} = p_x(a) \in \wb{u}{x}/l_x$, we define $\gamma_{x,\hat{a}} = p(\gamma_{x,a}) = q_x(a) \in \mathcal{P}_{\star \to j_x(\hat{a})}\Xq$ (see Lemma \ref{lemme: def q Latour cell}).

At the chain level:

$$\chi_{x,*}(\sigma \otimes p_{x,*}(s_x)) = \Phi'_{*}(\sigma \otimes q_{x,*}(s_x)) = \Phi'_{*}(\sigma \otimes m_x) = \Psi(\sigma \otimes x).$$

Furthermore, since $\wb{u}{x}/l_x$ is homeomorphic to a closed disk, there is a bijection 

$$ H_p(\wb{u}{x}/l_x, \partial \wb{u}{x}/l_x) \simeq \Z\langle x\rangle $$
 
when $|x|=p$. Thus, $\Psi: C_*(X,\Xi, C_*(F',u')) \to C_*(E',u')$ induces a quasi-isomorphism on the first pages $\Psi^{(1)}: E^1_{p,q} \to \mathcal{E}^1_{p,q}$ and this concludes the proof of Theorem \ref{thm: Fibration Theorem Morse Novikov}.

\section{Functoriality in enriched Morse-Novikov theory}\label{Section : Functoriality in the DG Morse-Novikov theory}
In this section, we will use the Latour Trick to define direct and shriek maps in the DG Morse-Novikov setting, analogous to the definition in the DG Morse setting given in \cite[Section 9 and 10]{BDHO23}. We will also give a chain-level construction that is equivalent at the homology level to the definition using the Latour Trick. The functorial properties of these maps are the same as in DG Morse Homology.

\subsection{Direct and shriek maps in homology}

Using the Latour Trick (see \ref{subsection : Latour Trick}), we can generalize the definition of direct and shriek maps in DG Morse Homology to DG Morse-Novikov Homology.

Let $(Y^k,\star_Y)$ be a pointed, oriented, closed, and connected manifold of dimension $k$ and $\varphi : X \to Y$ be a continuous map. We will assume that $\varphi(\star)= \star_Y.$ Let $u_Y \in H^1(Y,\R)$, $u_X = \varphi^*u_Y \in H^1(X,\R)$ and $\G$ be a DG right $C_*(\Omega Y,u_Y)$-module. 

Let $\Xi_X$ and $\Xi_Y$ be sets of DG Morse-Novikov data adapted to $u_X$ and $u_Y.$

Let $\Xi^L_{X}$ and $\Xi^L_{Y}$ be sets of DG Morse-Novikov data adapted to $u_X$ and $u_Y$ given by the Latour Trick. We proved in Proposition \ref{prop : Latour Trick} that $$C_*(X, \Xi^L_X,\varphi^*\G) \textup{ and } C_*(Y,\Xi^L_Y, \G)$$ are DG Morse complexes.

\begin{defi}\label{defi : functoriality DG Morse-Novikov with Latour trick}
    Let $$\varphi_* : H_*(X, \varphi^*\G, u_X) \to H_*(Y,\G,u_Y) \textup{ and } \varphi_! : H_*(Y,\G,u_Y) \to H_{*+n-k}(X,\varphi^*\G, u_X)$$
    be the \textbf{direct and shriek} maps respectively induced in homology by the compositions
    $$ C_*(X, \Xi_X, \varphi^*\G) \overset{\Psi^{01}_X}{\to} C_*(X, \Xi^L_X, \varphi^*\G) \overset{\varphi_*}{\to} C_*(Y,\Xi^L_Y,\G) \overset{\Psi^{10}_Y}{\to} C_*(Y,\Xi_Y,\G)$$
and 
    $$C_*(Y,\Xi_Y,\G) \overset{\Psi^{01}_Y}{\to} C_*(Y,\Xi^L_Y,\G) \overset{\varphi_!}{\to} C_{*+n-k}(X, \Xi^L_X, \varphi^*\G) \overset{\Psi^{10}_X}{\to} C_{*+n-k}(X, \Xi_X, \varphi^*\G),$$
    
    where $\varphi_* : C_*(X, \Xi^L_X, \varphi^*\G) \to C_*(Y,\Xi^L_Y,\G)$ and $\varphi_! : C_*(Y,\Xi^L_Y,\G) \to C_{*+n-k}(X,\Xi^L_X,\varphi^*\G)$ are, respectively, the direct and shriek maps for DG Morse homology defined in \cite[Section 9 and 10]{BDHO23}. The maps $ \Psi^{01}_X : C_*(X, \Xi_X, \varphi^*\G) \to C_*(X, \Xi^L_X, \varphi^*\G)$,  $\Psi^{10}_X : C_{*}(X, \Xi^L_X, \varphi^*\G) \to C_{*}(X, \Xi_X, \varphi^*\G)$, 
    
    $\Psi_Y^{01} : C_*(Y,\Xi_Y,\G)\to C_*(Y,\Xi^L_Y,\G)$ and $\Psi_Y^{10} : C_*(Y,\Xi^L_Y,\G) \to C_*(Y,\Xi_Y,\G) $ are continuation maps (see Section \ref{section : Invariance of DG Morse-Novikov homology}).
\end{defi}

\begin{rem}
    The composition $\Psi^{10}_Y \circ \varphi_* \circ \Psi^{01}_X$ depends on the choice of  DG Morse-Novikov data $\Xi^L_X$ and $\Xi^L_Y$ up to chain homotopy. Therefore this construction does not yield a canonical map $\varphi_* : C_*(X, \Xi_X, \varphi^*\G) \to C_*(Y,\Xi_Y,\G) $ at the chain-level.
\end{rem}

The following theorem is an extension of \cite[Theorem 8.1.1]{BDHO23} for smooth maps in the DG Morse-Novikov setting.

\begin{prop}[DG Morse-Novikov Functoriality]\label{prop : Functoriality in DG Morse-Novikov}
Let $X^n,Y^k,Z^{\ell}$ be oriented, closed, and connected manifolds.
Let $u \in H^1(Y,\R)$ and $\G_*$ be a DG right $C_*(\Omega Y, u)$-module.
A continuous map $\varphi: X \to Y$ induces in homology a \textbf{direct map}
$$\varphi_*: H_*(X,\varphi^*\G, \varphi^*u) \to H_*(Y, \G,u)$$
and a \textbf{shriek map}
$$\varphi_!: H_*(Y,\G,u) \to H_{*+n-k}(X,\varphi^*\G, \varphi^*u)$$ with the following properties:

\begin{enumerate}
    \item \textup{(IDENTITY)} We have $\Id_* = \Id_!= \Id: H_*(Y,\G,u) \to H_*(Y,\G,u)$.

    \item \textup{(COMPOSITION)} Given continuous maps $X \overset{\varphi}{\to} Y \overset{\psi}{\to} Z$, $u_Z \in H^1(Z,\R)$ and $\F$ a DG right $C_*(\Omega Z, u_Z)$-module. 

    $$(\psi \circ \varphi)_* = \psi_* \circ \varphi_*: H_*(X,\psi^*\varphi^*\F, \psi^*\varphi^*u_Z) \to H_*(Z, \F, u_Z)$$ and
    $$(\psi \circ \varphi)_! = \varphi_! \circ \psi_!: H_*(Z,\F,u_Z) \to H_{*+n-\ell}(X,\psi^*\varphi^*\F, \psi^*\varphi^*u_Z).$$
    \item \textup{(HOMOTOPY)} Two homotopic maps induce the same direct and shriek maps.

    \item \textup{(SPECTRAL SEQUENCE)} The direct and shriek maps are limits of morphisms between the spectral sequences associated with the corresponding enriched complexes, given at the second pages by 
    $$\varphi_{p,q,*} : H_p(X,\varphi^*H_q(\G)) \to H_p(Y,H_q(\G))$$ 
    and 
    $$\varphi_{p,q,!} : H_p(Y,H_q(\G)) \to H_{p+n-k}(X,\varphi^*H_q(\G)),$$ the usual direct and shriek maps in homology with local coefficients.
\end{enumerate}
\end{prop}

\subsection{Sketch of a construction at the chain level}\label{subsection : Sketch of a construction at the chain level}

Let $(X^n,\star)$ and $(Y^k,\star_Y)$ be pointed, oriented, closed, and connected manifolds and $\varphi: X \to Y$ be a continuous map. Let $u \in H^1(Y,\R)$ and $u_X = \varphi^*u \in H^1(X,\R)$.
We adapt the second definition of \cite[Section 10]{BDHO23}. Let $\Xi_X$ be a set of DG Morse-Novikov data on $X$ and $\Xi_Y$ be a set of DG Morse-Novikov data on $Y$. Choose $\alpha \in \Xi_X$ and $\beta \in \Xi_Y$ Morse 1-forms such that $\alpha \in u_X$ and $\beta \in u$. Let $\G$ be a right DG $C_*(\Omega Y,u)$-module. 
We will always work under the assumption that $ \varphi(\star) = \star_Y$.

We will assume that $\varphi : X \to Y$ is smooth and refer to \cite[Section 9.7]{BDHO23} to extend the definition for smooth maps to continuous maps since the construction of the \textbf{identification morphism} is the same in the DG Morse-Novikov context.

\subsubsection{Direct maps}

We will assume the generic transversality condition $$\varphi\lvert_{W^u_{\alpha}(x)} \pitchfork W^s_{\beta}(y)$$ for each $x \in \Crit(\alpha)$ and $y \in \Crit(\beta).$

\begin{defi}
    For each $x \in \Crit(\alpha)$ and $y \in \Crit(\beta)$, define the $|x|-|y|$-dimensional manifolds $$\Mtraj^{\varphi}(x,y):= W^u_{\alpha}(x) \cap \varphi^{-1}(W^s_{\beta}(y)) \cong W^u_{\alpha}(x) \ftimes{\varphi}{} W^s_{\beta}(y) = \left\{ (a,\varphi(a)) \in W^u_{\alpha}(x) \times W^s_{\beta}(y)\right\}$$
    and $$\overline{\mathcal{M}}^{\varphi}(x,y):= \wbi{u}{\alpha}{x} \pitchfork \varphi^{-1} \left( \wbi{s}{\beta}{y} \right) \cong  \wb{u}{x} \ftimes{\varphi}{} \wb{s}{y},$$
    
    that we orient by $$\left( \Or \ \overline{\mathcal{M}}^{\varphi}(x,y), \Or \ \wbi{u}{\beta}{y}\right) = \Or \ \wbi{s}{\alpha}{x}.$$

\end{defi}

In this section, we will denote $ \functrajb{}{x,y} = \overline{\mathcal{M}}^{\varphi}(x,y)$ when no confusion can arise.
The manifold $\functrajb{}{x,y}$ is a manifold with boundary and corners 

$$\partial \functrajb{}{x,y} = \bigcup_{z \in \Crit(\alpha)} \trajbi{\alpha}{x,z} \times \functrajb{}{z,y} \cup \bigcup_{ w \in \Crit(\beta)} \functrajb{}{x,w} \times \trajbi{\beta}{w,y}.$$

It has been proven in \cite[Proposition 10.1.1]{BDHO23} that the orientation difference between the product orientation and the boundary orientation is:

    $$\Or \ \partial \functrajb{}{x,y} = (-1)^{|x|-|z|} \left( \Or \ \trajbi{\alpha}{x,z}, \Or \ \functrajb{}{z,y} \right)$$ and $$\Or \ \partial \functrajb{}{x,y} = (-1)^{|x|-|y|-1} \left( \Or \ \functrajb{}{x,w}, \Or \ \trajbi{\beta}{w,y}\right).$$

We now define \textbf{evaluation maps} for these spaces of trajectory as in \cite[Section 10.1]{BDHO23}.

\begin{lemme}
There exists a family of evaluation maps $$q^{\varphi}_{x,y}: \functrajb{}{x,y} \to \Omega Y$$ for any two critical points $x \in \Crit(\alpha)$, $y \in \Crit(\beta)$, such that on the boundary
$$q^{\varphi}_{x,y}(\lambda,a,\varphi(a)) = \varphi(q^X_{x,z}(\lambda)) \# q^{\varphi}_{z,y}(a,\varphi(a))$$ if $(\lambda,a,\varphi(a)) \in \trajb{x,z} \times \functrajb{}{z,y},$ and $$q^{\varphi}_{x,y}(a,\varphi(a),\lambda') = q^{\varphi}_{x,w}(a,\varphi(a)) \# q^Y_{w,y}(\lambda')$$ if $(a,\varphi(a), \lambda') \in \functrajb{}{x,w} \times \trajb{w,y}.$

\end{lemme}

\begin{proof}

    Let $x \in \Crit(\alpha)$, $y \in \Crit(\beta)$ and $(a,\varphi(a)) \in \functrajb{}{x,y}.$ Let $H^X: [0,1] \times X \to X$ be a homotopy between $H^X(0,\cdot) = \Id$ and $H^X(1,\cdot) = \theta_X \circ p_X,$ and $H^Y: [0,1] \times Y \to Y$ be a homotopy between $H^Y(0,\cdot) = \Id$ and $H^Y(1,\cdot) = \theta_Y \circ p_Y.$ 

    Let $\Gamma_x: \wb{u}{x} \to \mathcal{P}_{x \to X} X$ and $\Gamma_y: \wb{s}{y} \to \mathcal{P}_{Y \to y} Y$ be the parametrization maps defined in Lemma \ref{lemme : defi of parametrization maps for Latour cells}. Let $$q^X_x = \theta_X \circ p_x \circ \Gamma_x: \wb{u}{x} \to \mathcal{P}_{\star \to X} X$$ and $$q^Y_y = \theta_Y \circ p_Y \circ \Gamma_y: \wb{s}{x} \to \mathcal{P}_{Y\to \star_Y} Y.$$ 

    Define $q^{\varphi}_{x,y}: \functrajb{}{x,y} \to \Omega Y$ by $$q^{\varphi}_{x,y}(a, \varphi(a)) = \varphi(q^X_x(a)) \# \varphi(H^X(1-\cdot,a)) \# H^Y(\cdot,\varphi(a)) \#q^Y_y(\varphi(a)).$$

 Since the evaluation maps $q^X$ and $q^Y$ satisfy such boundary conditions, so does $q^{\varphi}.$
\end{proof}

As for the moduli spaces of trajectories, to achieve compactness we need to either restrict the length of such trajectories or prescribe the homotopy class of the evaluation.

\begin{defi}
     Let $g \in \pi_1(Y)$, $x \in \Crit(\alpha)$ and $y \in \Crit(\beta)$. Define $$\functrajb{g}{x,y}= \left\{ (a, \varphi(a)) \in \functrajb{}{x,y}, \ \left[q^{\varphi}_{x,y}(a,\varphi(a))\right] = g\right\}.$$
    
\end{defi}

\begin{prop}\label{prop : compactness of moduli spaces for direct maps}
    For any $x \in \Crit(\alpha)$, 
    $y \in \Crit(\beta)$ and $g \in \pi_1(Y),$  $\functrajb{g}{x,y}$ is a compact manifold with boundary and corners of dimension $|x|-|y|$ such that $$\partial \functrajb{g}{x,y} = \bigcup_{\substack{h' \in \pi_1(X), g'' \in \pi_1(Y), \\ \varphi^*(h') \cdot g'' = g,\\z \in \Crit(\alpha)}} \trajbi{h'}{x,z} \times \functrajb{g''}{z,y} \cup \bigcup_{\substack{g',g'' \in \pi_1(Y), \\ g' \cdot g'' = g, \\ w \in \Crit(\beta)}} \functrajb{g'}{x,w} \times \trajbi{g''}{w,y}$$.
\end{prop}

\begin{proof}
    We will prove that $\functrajb{g}{x,y}$ is a union of connected components of the compact manifold $$\functrajb{}{x,y,A}:= \wb{u}{x,A} \pitchfork \varphi^{-1}\left(\wb{s}{y,A}\right),$$ for $A>0$ large enough.

    We first prove that there exists $A>0$ such that $$\functrajb{g}{x,y} \subset \functrajb{}{x,y,A}.$$ Let $(a,\varphi(a)) \in \functrajb{g}{x,y}$. By definition, \begin{align*}
        u(g) &= L(q^{\varphi}_{x,y}(a,\varphi(a)))\\
        &= L(\varphi(q_x^X(a))) + L(\varphi(H^X(1-\cdot,a))) + L(H^Y(\cdot, \varphi(a))) + L(q^Y_y(\varphi(a))).
    \end{align*}

Since the function $$\begin{array}{ccc}
   X  & \to & \R  \\
    a & \mapsto & \int_{H^X(1-\cdot,a)} \varphi^*\beta + \int_{H^Y(\cdot, \varphi(a))} \beta
\end{array}$$

is continuous on the compact space $X$, there exists $K_0 \in \R$ such that $$\int_{q_x^X(a)} \varphi^*\beta + \int_{q_y^Y(\varphi(a))} \beta \leq u(g) - K_0.$$

Stokes' Theorem implies that $$\int_{q^X_x(a)}\varphi^*\beta - \int_{\Gamma_x(a)}\varphi^*\beta = \int_{H^X(\cdot,a)}\varphi^*\beta - \int_{H^X(\cdot,x)} \varphi^*\beta.$$  Since $$\begin{array}{ccc}
    X & \to & \R  \\
    a & \mapsto & \int_{H^X(\cdot,a)}\varphi^*\beta - \int_{H^X(\cdot,x)} \varphi^*\beta 
\end{array}$$ is bounded, so is $$\begin{array}{ccc}
     \wb{u}{x} & \to & \R  \\
     a & \mapsto & \int_{q^X_x(a)}\varphi^*\beta - \int_{\Gamma_x(a)}\varphi^*\beta. 
\end{array}$$ For the same reasons,  $$\begin{array}{ccc}
    \wb{s}{y} & \to & \R  \\
    b & \mapsto & \int_{q_y^Y(b)}\beta - \int_{\Gamma_y(b)}\beta
\end{array}$$ is bounded.
It follows that there exists $A>0$ such that $$L(\varphi(a))=\int_{\Gamma_y(\varphi(a))}\beta <A \textup{ and } L(a) = \int_{\Gamma_x(a)}\varphi^*\beta<A.$$

To conclude the proof it suffices to remark that, if $\gamma : [0,1] \to \functrajb{}{x,y}$ is a continuous path starting at $\gamma(0) \in \functrajb{g}{x,y}$, then $\gamma(1) \in \functrajb{g}{x,y}$. Indeed, since $q_{x,y}$ is continuous, $[q_{x,y}(\gamma(1))] = [q_{x,y}(\gamma(0))] = g$.
\end{proof}

As in Proposition \ref{prop: representing chain system}, we can build a representing chain system $$\left\{ \sigma_{x,y}^g \in C_{|x|-|y|}(\functrajb{g}{x,y}), \ x \in \Crit(\alpha), y \in \Crit(\beta), g \in \pi_1(Y)\right\}$$ such that, for all $x \in \Crit(\alpha), y \in \Crit(\beta)$ and $ g \in \pi_1(Y)$,

$$\partial \sigma_{x,y}^g = \sum_{\substack{h' \in \pi_1(X), g'' \in \pi_1(Y), \\ \varphi^*(h') \cdot g'' = g,\\z \in \Crit(\alpha)}}  s^{h'}_{x,z} \times \sigma^{g''}_{z,y} - \sum_{\substack{g', g'' \in \pi_1(Y), \\ g' \cdot g'' = g,\\w \in \Crit(\beta)}} (-1)^{|x|-|w|} \sigma_{x,w}^{g'} \times s^{g''}_{w,y}.$$

It remains now to evaluate this family using $\{q^{\varphi}_{x,y}\}$ to obtain a family of chains $$\left\{\nu^g_{x,y} \in C_{|x|-|y|}(\Omega^g Y), \ x \in \Crit(\alpha), y \in \Crit(\beta), g \in \pi_1(Y)\right\}$$ that satisfies  \eqref{eq: continuation map g}: $$\partial\nu^g_{x,y} = \sum_{\substack{h' \in \pi_1(X), g'' \in \pi_1(Y), \\ \varphi^*(h') \cdot g'' = g,\\z \in \Crit(\alpha)}} (-1)^{|x|-|z|} \varphi_*(m^{h'}_{x,z}) \cdot \nu^{g''}_{z,y} - \sum_{\substack{g', g'' \in \pi_1(Y), \\ g' \cdot g'' = g,\\w \in \Crit(\beta)}} (-1)^{|x|-|w|} \nu_{x,w}^{g'} \cdot m^{g''}_{w,y}.$$

The direct map associated with $\varphi: X \to Y$ is defined by 

$$\deffct{\varphi_*}{C_*(X,\Xi_X,\varphi^*\G)}{C_*(Y,\Xi_Y,\G)}{\sigma \otimes x}{\displaystyle \sum_y \sigma \cdot \nu_{x,y} \otimes y,}$$ where $$\nu_{x,y} = \sum_{g \in \pi_1(X)} \nu^g_{x,y} \in C_{|x|-|y|}(\Omega Y,u).$$

\subsubsection{Shriek maps}

Let $\varphi : X^n \to Y^k$ be a smooth map.
We assume the generic transversality condition:

$$\varphi \lvert_{W^s_{\alpha}(x)} \pitchfork W^s_{\beta}(y),$$

for each $x \in \Crit(\alpha)$ and $y \in \Crit(\beta).$

\begin{defi}
    For each $x \in \Crit(\alpha)$ and $y \in \Crit(\beta)$, define the $|y|-|x|+n-k$-dimensional manifolds 

    $$\Mtraj^{\varphi_!}(y,x) := W^s(x) \cap \varphi^{-1}\left(W^u(y)\right) \cong  W^s(x) \ftimes{\varphi}{}W^u(y) $$ and $$\overline{\Mtraj}^{\varphi_!}(y,x) = \wb{s}{x} \cap \varphi^{-1}\left(\wb{u}{y}\right) \cong \wb{s}{x} \ftimes{\varphi}{} \wb{u}{y},$$

    that we orient by $$\left( \Or \ \wb{s}{y}, \Or \ \overline{\Mtraj}^{\varphi_!}(y,x) \right) = \Or \ \wb{s}{x}.$$
\end{defi}

When the context is clear, we will use the notation $\functrajb{}{y,x} = \overline{\Mtraj}^{\varphi_!}(y,x)$.

This manifold has boundary and corners such that $$\partial \functrajb{}{y,x} = \bigcup_{w \in \Crit(\beta)} \trajbi{\beta}{y,w} \times \functrajb{}{w,x} \cup \bigcup_{ z \in \Crit(\alpha)} \functrajb{}{y,z} \times \trajbi{\alpha}{z,x}.$$

For convenience we will denote $[x] = |x|+ k$ and $[y] = |y| +n$, so that $$\dim\left( \functrajb{}{y,x} \right) =|y|-|x|+n-k = [y] - [x].$$

It has been proven in \cite[Lemma 10.4.2]{BDHO23} that the orientation difference between the product orientation and the boundary orientation is

    $$\Or \ \partial \functrajb{}{y,x} = (-1)^{[y]-[w]} \left( \Or \ \trajbi{\beta}{y,w}, \Or \ \functrajb{}{w,x} \right)$$ and $$\Or \ \partial \functrajb{}{y,x} = (-1)^{[y]-[x]-1} \left( \Or \ \functrajb{}{y,z}, \Or \ \trajbi{\beta}{z,x}\right).$$

Evaluation maps are defined using an idea similar to that used for direct maps (see \cite[Section 10.4]{BDHO23}).

\begin{lemme}
There exists a family of evaluation maps $$q^{\varphi_!}_{y,x}: \overline{\Mtraj}^{\varphi_!}(y,x) \to \Omega Y,$$ one for any two critical points $x \in \Crit(\alpha)$, $y \in \Crit(\beta)$, such that on the boundary
$$q^{\varphi_!}_{y,x}(\lambda,a,\varphi(a)) = q^Y_{y,w}(\lambda) \# q^{\varphi_!}_{w,x}(a,\varphi(a))$$ if $(\lambda,a,\varphi(a)) \in \trajb{y,w} \times \functrajb{}{w,x},$ and $$q^{\varphi_!}_{y,x}(a,\varphi(a),\lambda') = q^{\varphi_!}_{y,z}(a,\varphi(a)) \# \varphi(q^X_{z,x}(\lambda'))$$ if $(a,\varphi(a), \lambda') \in \functrajb{}{y,z} \times \trajb{z,x}.$

\end{lemme}

\begin{defi}
    Let $g \in \pi_1(Y)$, $x \in \Crit(\alpha)$ and $y \in \Crit(\beta).$ Define $$\functrajb{g}{y,x} = \left\{ (a,\varphi(a)) \in \functrajb{}{y,x}, \left[q^{\varphi_!}_{y,x}(a,\varphi(a)) \right] = g \right\}.$$
\end{defi}

Analogous arguments as for the proof of Proposition \ref{prop : compactness of moduli spaces for direct maps} show the following result.

\begin{prop}\label{prop : compactness of moduli spaces for shriek maps}
    For each $x \in \Crit(\alpha)$, 
    $y \in \Crit(\beta)$ and $g \in \pi_1(Y),$  $\functrajb{g}{y,x}$ is a compact manifold with boundary and corners of dimension $[x]-[y]$ such that $$\partial \functrajb{g}{y,x} = \bigcup_{\substack{g',g'' \in \pi_1(Y), \\ g' \cdot g'' = g, \\ w \in \Crit(\beta)}} \trajbi{g'}{y,w} \times \functrajb{g''}{w,x} \cup \bigcup_{\substack{h'' \in \pi_1(X), g'' \in \pi_1(Y), \\ g' \cdot \varphi^*(h'') = g,\\z \in \Crit(\alpha)}} \functrajb{g'}{y,z} \times \trajbi{h''}{z,x}.$$
\end{prop}

\begin{flushright}
    $\blacksquare$
\end{flushright}

We can build a representing chain system $$\left\{ \sigma_{y,x}^g \in C_{[y]-[x]}(\functrajb{g}{y,x}), \ x \in \Crit(\alpha), y \in \Crit(\beta), g \in \pi_1(Y)\right\},$$ such that for all $x \in \Crit(\alpha), y \in \Crit(\beta)$ and $ g \in \pi_1(Y)$,

$$\partial \sigma_{y,x}^g = \sum_{\substack{g', g'' \in \pi_1(Y), \\ g' \cdot g'' = g,\\w \in \Crit(\beta)}}  s^{g'}_{y,w} \times \sigma^{g''}_{w,x} - \sum_{\substack{h' \in \pi_1(X), g'' \in \pi_1(Y), \\ g' \cdot \varphi^*(h'') = g,\\z \in \Crit(\alpha)}} (-1)^{[y]-[z]} \sigma_{y,z}^{g'} \times s^{h''}_{z,x}.$$

It remains now to evaluate this family using $\{q^{\varphi_!}_{y,x}\}$ to obtain a family of chains $$\left\{\nu^{\varphi_!,g}_{y,x} \in C_{[x]-[y]}(\Omega^g Y), \ x \in \Crit(\alpha), y \in \Crit(\beta), g \in \pi_1(Y)\right\}$$ that satisfies  \eqref{eq: continuation map g}:

$$\partial \nu_{y,x}^{\varphi_!,g} = \sum_{\substack{g', g'' \in \pi_1(Y), \\ g' \cdot g'' = g,\\w \in \Crit(\beta)}}  m^{g'}_{y,w} \times \nu^{\varphi_!,g''}_{w,x} - \sum_{\substack{h' \in \pi_1(X), g'' \in \pi_1(Y), \\ g' \cdot \varphi^*(h'') = g,\\z \in \Crit(\alpha)}} (-1)^{[y]-[z]} \nu_{y,z}^{\varphi_!,g'} \times \varphi_*\left(m^{h''}_{z,x}\right).$$

The shriek map associated with $\varphi: X \to Y$ is defined by 

$$\deffct{\varphi_!}{C_*(Y,\Xi_Y,\G)}{C_{*+n-k}(X,\Xi_X,\varphi^*\G)}{\sigma \otimes y}{\displaystyle \sum_x \sigma \cdot \nu^{\varphi_!}_{y,x} \otimes x,}$$ where $$\nu_{y,x}^{\varphi_!} = \sum_{g \in \pi_1(X)} \nu^{\varphi_!,g}_{y,x} \in C_{[y]-[x]}(\Omega Y,u).$$

\subsubsection{Equivalence with Definition \ref{defi : functoriality DG Morse-Novikov with Latour trick}}

To prove that these definitions are equivalent to the ones given in Definition \ref{defi : functoriality DG Morse-Novikov with Latour trick}, we need to prove that the direct and shriek maps commute with the continuation maps up to chain homotopy:

$$\xymatrix{
C_*(X,\Xi^X_0, \varphi^*\G) \ar[r]^{\varphi_*} \ar[d]_{\Psi_{01}^X} & C_*(Y, \Xi^Y_0, \G) \ar[d]^{\Psi_{01}^Y}\\
C_*(X,\Xi^X_1, \varphi^*\G) \ar[r]_{\varphi_*} & C_*(Y, \Xi^Y_1, \G)
}$$

and 

$$\xymatrix{
C_{*+n-k}(X,\Xi^X_0, \varphi^*\G)  \ar[d]_{\Psi_{01}^X} & C_*(Y, \Xi^Y_0, \G) \ar[l]_-{\varphi_!} \ar[d]^{\Psi_{01}^Y}\\
C_{*+n-k}(X,\Xi^X_1, \varphi^*\G)  & C_*(Y, \Xi^Y_1, \G) \ar[l]^-{\varphi_!}.
}$$

To do that, we refer to \cite[Proposition 10.2.1]{BDHO23} and \cite[Proposition 10.5.1]{BDHO23} to prove that $\Psi_{01} : C_*(X,\Xi_{0},\G) \to C_*(X,\Xi_1,\G)$ is chain homotopic to $\Id_* : C_*(X,\Xi_{0},\G) \to C_*(X,\Xi_1,\G)$ and that $\Psi_{10} : C_*(X,\Xi_1,\G) \to C_*(X,\Xi_0,\G)$ is chain homotopic to $\Id_! : C_*(X,\Xi_1,\G) \to C_*(X,\Xi_0,\G)$. The proof of these propositions can be adapted in our context using the spaces $\trajbi{g}{x,y}, \overline{\mathcal{M}}_{g}^{\Id}(x,y)$  and $\overline{\mathcal{M}}_{g}^{\Id_!}(x,y)$ and the DG Morse-Novikov toolset.
We also refer to \cite[Section 10.3]{BDHO23} to prove that, up to chain homotopy, $\varphi_*\psi_* = (\varphi \circ \psi)_*$ and $\psi_! \circ \varphi_! = (\varphi \circ \psi)_!$ for all smooth functions $\varphi : X \to Y$ and $\psi : Z \to X$.

Since any map commutes with the identity, which is chain homotopic to the continuation map, the commutativity of the two diagrams above up to chain homotopy follows.\\

To conclude, it suffices to remark that, for a set of DG Morse-Novikov data obtained by the Latour Trick, the definitions of the direct and shriek maps coincide, up to chain homotopy equivalence with the definitions of \cite{BDHO23}. This is true since, in this case, the spaces $\overline{\mathcal{M}}^{\varphi}_g(x,y)$ for $g \in \pi_1(X)$ are either empty or unions of connected components of $\overline{\mathcal{M}}^{\varphi}(x,y)$. Therefore $$\nu_{x,y} = \sum_g \nu^g_{x,y} \in C_{|x|-|y|}(\Omega Y)$$ is the cocycle defined in \cite{BDHO23}.

\section{Chas-Sullivan type product in enriched Morse-Novikov homology. Proof of Theorem \ref{theorem G}}\label{section : Chas-Sullivan type product in enriched Morse-Novikov homology}

The goal of this section is to prove the following theorem:

\begin{thm}[Theorem \ref{theorem G}]\label{thm : Chas-Sullivan product in Morse-Novikov}

Let $u \in H^1(X,\R)$ and a fibration $F \hookrightarrow \fibration$  endowed with a morphism of fibrations $m : E \ftimes{\pi}{\pi} E \to E$ such that $\pi^*u = 0 \in H^1(E,\R)$ and $F$ is locally path-connected. Endow $\F^u = C_*(F,u)$ with a DG right $C_*(\Omega X,u)$-module structure induced by a transitive lifting function  $\Phi : E \ftimes{\pi}{\ev_0} \mathcal{P}X \to E$.

The morphism of fibrations $m : E \ftimes{\pi}{\pi} E \to E$ induces a degree $-n$ product

$$\CS^u : H_*(X, \F^u)^{\otimes 2} \to H_*(X, \F^u).$$

such that the following properties hold:
    \begin{itemize}
    \item \textbf{\textup{Associativity :}}  If $m_* : H_*(E \ftimes{\pi}{\pi}E,u) \to H_*(E,u)$ is associative, then so is $\CS^u$.
    \item \textbf{\textup{Commutativity :}}  If $m_* : H_*(E \ftimes{\pi}{\pi}E,u) \to H_*(E,u)$ is commutative, then $\CS^u$ is commutative up to sign $$\CS^u(\gamma \otimes \delta) = (-1)^{(n-|\gamma|)(n-|\delta|)} \CS^u(\delta \otimes \gamma).$$
    \item \textbf{\textup{Functoriality :}}\\ $\bullet$ For any pointed, oriented, closed, and connected manifold $Y^k$, any continuous map $g : Y \to X$ satisfies that $\pi_Y^*g^*u = \pi^* u = 0$ and induces a degree $-k$ product for the fibration $F \hookrightarrow g^*E \overset{\pi_Y}{\to} Y$

    $$\CS^{Y,g^*u} : H_i(Y,g^*\F^u)\otimes H_j(Y, g^*\F^u) \to H_{i+j-k}(Y,g^*\F^u),$$ such that $g_! : H_*(X,\F^u) \to H_{*+n-k}(Y, g^*\F^u)$ is a morphism of rings up to sign.
    
     $\bullet$ If $g$ is an orientation-preserving homotopy equivalence then, $g_!$ and
     $g_* : H_*(Y,g^*\F^u) \to H_{*}(X, \F^u)$ are  isomorphisms of rings inverse of each other.
     
     \item \textbf{\textup{Spectral sequence :}} Let $\Xi$ be a set of DG Morse data on $X$. The canonical filtration
    $$F_p(C_*(X,\Xi, \F^u)) = \bigoplus_{\substack{i +j = k\\ i \leq p}} \F^u_j \otimes \Z\Crit_i(f)$$
    induces a spectral sequence $E^r_{p,q}$ that is endowed with an algebra structure
    $$E^r_{p,q} \otimes E^r_{\ell,m} \to E^r_{p+\ell-n, q+m}$$ induces by a chain-level model $\CS^u : C_*(X,\Xi,\F^u)^{\otimes 2} \to C_*(X,\Xi,\F^u)$ and converges towards $H_*(X, \F^u)$ as algebras.  For $s,t \geq 0$ $E^2_{s,t} = H_s(X,H_t(\F^u))$ and the algebra structure is given, up to sign, by the intersection product on $X$ with coefficients in $H_t(\F^u)$.
    \end{itemize}

Moreover, if $u=0$, then $\CS^u = \CS $ is the Chas-Sullivan product in DG Morse Homology defined in \cite[Theorem 7.1]{Rie24}.

\end{thm}

The product $\CS$ was defined in \cite[Section 7]{Rie24} as follows:

$$\underbrace{H_*(X,C_*(F))^{\otimes 2}}_{\simeq H_*(E)^{\otimes 2}} \overset{K}{\to} \underbrace{H_*(X^2,C_*(F^2))}_{\simeq H_*(E^2)} \overset{\Delta_!}{\to} \underbrace{H_{*-n}(X,\Delta^*C_*(F^2))}_{\simeq H_{*-n}(E \ftimes{\pi}{\pi} E)} \overset{\tilde{m}}{\to} H_{*-n}(X,C_*(F)),$$

where: \begin{itemize}
    \item[$\bullet$] $K : C_*(X,\Xi,C_*(F))^{\otimes 2} \to C_*(X^2,\Xi_{X^2},C_*(F^2))$ is the DG cross-product defined in \cite[Section 6]{Rie24},
    \item[$\bullet$] $\Delta_! : C_*(X^2,\Xi_{X^2},C_*(F^2)) \to C_{*-n}(X,\Xi,\Delta^*C_*(F^2))$ is the shriek map of the diagonal $\Delta : X \to X^2$ defined in \cite[Section 9 and 10]{BDHO23}.
    \item[$\bullet$] $\tilde{m} : C_{*-n}(X,\Xi,\Delta^*C_*(F^2)) \to C_{*-n}(X,\Xi,C_*(F))$ is the map induced by the morphism of fibrations $m : E \ftimes{\pi}{\pi} E \to E$ defined in \cite[Section 5]{Rie24}.
\end{itemize} 

We have already seen in Section \ref{Section : Functoriality in the DG Morse-Novikov theory} that the shriek maps can be extended to the DG Morse-Novikov setting. It remains to prove such results for morphisms of fibrations and for the cross-product $K.$

\subsection{Morphism of fibrations. Proof of Theorem \ref{theorem D}}\label{subsection : Morphism of fibrations}

Let $u \in H^1(X,\R)$ and $F_0 \hookrightarrow E_0 \overset{\pi_0}{\to} X$, $F_1 \hookrightarrow E_1 \overset{\pi_1}{\to} X$ be (Hurewicz) fibrations such that $\pi_0^*u = \pi_1^*u=0$ and $F_0,F_1$ are locally path-connected. Let $\alpha \in \Omega^1(X)$ a representative of $u$. A morphism of fibrations is a continuous map $\varphi:E_0 \to E_1$ such that $\pi_0 = \pi_1 \circ \varphi$. 

Let $\varphi : E_0 \to E_1$ be a morphism of fibrations. If $\alpha \in u$ and $f_1 : E_1 \to \R$ is a primitive of $\pi^*_1\alpha$ (in the sense of Definition \ref{def : primitive of pi alpha}), then $f_0 = f_1 \circ \varphi$ is a primitive of $\pi^*_0\alpha$. Since we have proved that $C_*(E_0,f_0)$ does not depend, up to chain homotopy equivalence, on the choice of primitive $f_0$, we will always assume the relation $f_0 = f_1 \circ \varphi$.

It is clear that $\varphi$ induces a chain map $$\varphi_* : C_*(E_0, f_0) \to C_*(E_1,f_1).$$

We will now prove that the arguments used in \cite[Theorem 5.8]{Rie24} carry along in this context to prove the following analogous theorem:

\begin{thm}\label{thm : morphisme induit commute avec iso Morse-Novikov}
    
    Let $\Xi$ be a set of DG Morse-Novikov data on $X$.

    \begin{itemize}
        \item[i)] There exists a sequence of maps $\left\{\varphi_{n+1} : [0,1]^n \times F_0 \times \Omega X^{n-1} \times \mathcal{P}_{\star \to X} X \to E_1\right\}$, called a \textbf{coherent homotopy} for $\varphi$ that induces a morphism of complexes $$\tilde{\varphi} : C_*(X,\Xi, C_*(F_0,f_0)) \to C_*(X,\Xi, C_*(F_1,f_1)).$$

        \item[ii)] A coherent homotopy also induces a chain homotopy $v : C_*(X,\Xi,C_*(F_0,f_0)) \to C_{*+1}(E_1,f_1)$ between $\Psi_1 \circ \tilde{\varphi}$ and $\varphi_* \circ \Psi_0$ for any set of DG Morse data $\Xi$ on $X$,  where the morphisms $\Psi_0$ and $\Psi_1$ are the quasi-isomorphisms given by the Fibration Theorem. In other words, the following diagram commutes up to chain homotopy
    \end{itemize}
    \[
    \xymatrix{
    C_*(X,\Xi, C_*(F_0,f_0)) \ar[r]^-{\Tilde{\varphi}} \ar[d]^{\Psi_0} & C_*(X,\Xi, C_*(F_1,f_1)) \ar[d]_{\Psi_1} \\
    C_*(E_0,f_0) \ar[r]_{\varphi_*} & C_*(E_1,f_1).
    }
    \]
\end{thm}

We restate here a topological result about fibrations showing that a morphism of fibrations $\varphi : E_0 \to E_1$ commutes with any transitive lifting functions $\Phi_0 : E_0 \ftimes{\pi}{\ev_0} \mathcal{P}X \to E_0$ and $\Phi_1 : E_1 \ftimes{\pi}{\ev_0} \mathcal{P}X \to E_1$ associated respectively to $E_0$ and $E_1$ up to \textbf{coherent homotopy}, \emph{i.e.} successive homotopies compensating for the fact that, in general, $\varphi \circ \Phi_0 \neq \Phi_1(\varphi, \cdot).$

We here restate \cite[Lemma 5.9]{Rie24}.
\begin{lemme}\label{lemme : morphism of fibrations induces coherent homotopy Morse-Novikov}
    Let $F_0 \hookrightarrow E_0 \overset{\pi_0}{\to} X $ and $F_1 \hookrightarrow E_1 \overset{\pi_1}{\to} X$ be two fibrations over $X$ endowed with transitive lifting functions $\Phi_0 : E_0 \times \mathcal{P}_{\star \to X} X \to E_0$ and $\Phi_1 : E_1 \times \mathcal{P}_{\star \to X} X \to E_1$. Let $I = [0,1].$
    
    For any morphism of fibrations $\varphi : E_0 \to E_1$ there exists a sequence of maps $$\varphi_{n+1} : I^n \times F_0 \times \Omega X^{n-1} \times \mathcal{P}_{\star \to X} X \to E_1 \textup{ for } n\geq 1$$ called \textbf{coherent homotopy} such that  $\pi_1 \circ \varphi_{n+1}(t_1,\dots, t_n, \alpha, \gamma_1, \dots, \gamma_{n}) = \ev_1(\gamma_n)$ and
    \begin{equation}\label{eq : Ai morphism of fibration}
    \begin{split}
    & \varphi_{n+1}(t_1,\dots, t_n, \alpha, \gamma_1, \dots, \gamma_n) = \\
    & \left\{\begin{array}{rl}
      \varphi_{n}(t_2,\dots, t_n, \Phi_1(\alpha,\gamma_1),\gamma_2, \dots, \gamma_{n})   &  \textup{if} \ t_1=1, \\
       \varphi_{n}(\hat{t}_j , \alpha, \gamma_1, \dots,\gamma_{j-1} \cdot \gamma_j,\dots, \gamma_{n})  & \textup{if} \ t_j=1, \ j \geq 2,\\
       \Phi_2\left(\varphi_{j}(t_1,\dots, t_{j-1}, \alpha, \gamma_1, \dots, \gamma_{j-1}),\gamma_{j} \cdot \ \dots \ \cdot \gamma_{n})\right) & \textup{if} \ t_j=0,
    \end{array}\right.
   \end{split}
    \end{equation} 
 where we denoted $\varphi_1 = \varphi$.
\end{lemme}
\begin{flushright}
    $\blacksquare$
\end{flushright}

This lemma is useful in DG Morse theory to prove that a morphism of fibrations induces a morphism $\left\{\phi_{n+1} : C_*(F_0) \otimes C_*(\Omega X)^{\otimes n} \to C_*(F_1) \right\}$ of $\Ai$-modules over $C_*(\Omega X)$ such that $\phi_1 = \varphi_*$, where the $C_*(\Omega X)$-module structures of $C_*(F_0)$ and $C_*(F_1)$ are induced by $\Phi_0$ and $\Phi_1$. This morphism of $\Ai$-modules induces a chain map 
$$\tilde{\varphi} : C_*(X,\Xi,C_*(F_0)) \to C_*(X,\Xi, C_*(F_1)), \ \tilde{\varphi}(\sigma \otimes x) = \varphi_*(\sigma) \otimes x + \sum_{\substack{n \geq 1\\ z_1, \dots, z_n}} \pm \phi_{n+1}(\sigma \otimes m_{x,z_1} \otimes \dots \otimes m_{z_{n-1},z_n}) \otimes z_n.$$ 
See \cite[Section 5]{Rie24} for more details.\\

\begin{proof}[Proof of Theorem \ref{thm : morphisme induit commute avec iso Morse-Novikov}.i)]

We will now prove that this coherent homotopy preserves length. 

\begin{lemme}\label{lemme : coherent homotopy preserves length}
    Let $\{\varphi_{n+1} : I^n \times F_0 \times \Omega X^{n-1} \times \mathcal{P}_{\star \to X} X \to E_1, \ n \geq 1\}$ be the coherent homotopy constructed for the morphism of fibrations $\varphi$. Then, for all $v \in F_0$, $\gamma_1, \dots, \gamma_{n-1} \in \Omega X$, $\gamma_n \in \mathcal{P}_{\star \to X}X$, $$f_1 \left( \varphi_{n+1}(t_1, \dots, t_n, v, \gamma_1, \dots, \gamma_n) \right) = f_0(v) + \sum_{i=1}^{n-1} u([\gamma_i]) + \int_{\gamma_n} \alpha.$$
\end{lemme}

\begin{proof}[Proof of Lemma \ref{lemme : coherent homotopy preserves length}]
Let $n \geq 1$, $t_1, \dots, t_n \in I$, $v \in F_0$ and $\gamma_1, \dots, \gamma_n \in \Omega X$. The map $\varphi_{n+1}$ is defined as 

$$\varphi_{n+1}(t_1, \dots, t_n, v,\gamma_1, \dots, \gamma_n) = \Phi_2\left( \varphi\left( \Phi_1\left(v, (\gamma_1 \# \ \dots \ \# \gamma_n)\lvert_{I_n^0}\right) \right),  (\gamma_1 \# \ \dots \ \# \gamma_n)\lvert_{I_n^1} \right), $$

where $I_n^0$ and $I_n^1$ are some intervals that depend on $t_1, \dots, t_n$ and $a_1, \dots, a_n\in [0,+\infty)$, where $\gamma_i : [0,a_i] \to X$ for $i \in \{1, \dots, n\}$, are such that $I_n^0 \cup I_n^1 = [0, a_1 + \dots + a_n]$ and $I_n^0 \cap I_n^1$ is a single point.  Therefore, $$\int_{I_n^0} (\gamma_1 \# \ \dots \ \# \gamma_n)^* \alpha + \int_{I_n^1} (\gamma_1 \# \ \dots \ \# \gamma_n)^* \alpha = \int_{\gamma_1 \# \ \dots \ \# \gamma_n} \alpha = \sum_{i=1}^n u([\gamma_i]). $$

For readability, denote $\tau = \gamma_1 \# \ \dots \ \# \gamma_n.$
We compute

$$f_0(\Phi_0(v,\tau\lvert_{I_n^0})) - f_0(v) = \int_{I_n^0} \tau^* \alpha$$ and $$f_1\left(\Phi_1(\varphi(\Phi_0(v,\tau\lvert_{I_n^0})), \tau\lvert_{I^1_n})\right) - f_0(\Phi_0(v,\tau\lvert_{I_n^0})) = f_1\left(\Phi_1(\varphi(\Phi_0(v,\tau\lvert_{I_n^0})), \tau\lvert_{I^1_n})\right) - f_1\left(\varphi(\Phi_0(v,\tau\lvert_{I_n^0}))\right) = \int_{I_n^1} \tau^*\alpha.$$

Therefore, \begin{align*}
    &f_1 \left(\varphi_{n+1}(t_1, \dots, t_n, v, \gamma_1, \dots, \gamma_n)\right) - f_0(v) = \sum_{i=1}^n u([\gamma_i]) + \int_{\gamma_n}\alpha\\
    &\Leftrightarrow f_1 \left( \varphi_{n+1}(t_1, \dots, t_n, v, \gamma_1, \dots, \gamma_n) \right) = f_0(v) + \sum_{i=1}^n u([\gamma_i]) + \int_{\gamma_n}\alpha.
\end{align*}
\end{proof}

In particular, the coherent homotopy $(\varphi_{n+1})_{n \geq 1}$ induces a morphism $$\left\{\phi_{n+1} : C_*(F_0,f_0) \otimes C_*(\Omega X,u)^{\otimes n} \to C_*(F_1,f_1) \right\}$$ of $\Ai$-modules over $C_*(\Omega X,u)$, where 

$$\phi_{n+1}(\sigma \otimes \omega_1 \otimes \dots \otimes \omega_n) = \varphi_{n+1,*}(\Id_{I^n} \otimes \sigma \otimes \omega_1 \otimes \dots \otimes \omega_n)$$ if $n \geq 1$ and $$\phi_1 = \varphi_*.$$

\end{proof}

\begin{proof}[Proof of Theorem \ref{thm : morphisme induit commute avec iso Morse-Novikov}.ii)]

First assume that the set of DG Morse-Novikov data $\Xi_1$ is given by the Latour Trick (see Proposition \ref{prop : Latour Trick}). The proof of Theorem \ref{thm : morphisme induit commute avec iso Morse-Novikov} \emph{ii)} is the same as \cite[Theorem 5.8 ii)]{Rie24}. It relies on the fact that \eqref{eq : Ai morphism of fibration} induces an "$\Ai$- relation"

\begin{align*}
\sum_{n \geq 1} (-1)^{n+1} \partial \phi_{n+1} &=\sum_{n\geq 1} (-1)^{n+1} \phi_n \Phi_{1,*} - \sum_n  \Phi_{2,*} (\phi_n \otimes 1)\\
    & + \sum_{n\geq 1} (-1)^{n+1}(-1)^n \phi_{n+1} \partial_{\F_1} - \sum_{n \geq 1} \sum_{r\geq 1} \phi_{n+1}(1^{\otimes r} \otimes \mu_1 \otimes 1^{\otimes n-r}) \\
    & - \sum_{n\geq 2} \sum_{r \geq 1} (-1)^{r+n} \phi_n(1^{\otimes r} \otimes \mu_2 \otimes 1^{\otimes n-r-1})
\end{align*}

as a functional equality for the maps

$$\phi_{n+1} : C_*(F_0,f_0) \otimes C_*(\Omega X,u)^{\otimes n-1} \otimes C_*(\mathcal{P}_{\star \to X} X)  \to C_*(E_1,f_1)$$
defined by
$$\phi_{n+1}(\sigma \otimes \omega_1 \otimes \dots \otimes \omega_n) = \varphi_{n+1,*}(\Id_{I^n} \otimes \sigma \otimes \omega_1 \otimes \dots \otimes \omega_n).$$ 

Consider the pull-back fibrations $\theta^*E_0$ and $\theta^* E_1$ endowed with the associated transitive lifting functions $\Phi'_0$ and $\Phi'_1$, and let $f'_i : \theta^*E_i \to \R, \ f'_i(y,e) = f'(e)$ for $i \in \{0,1\}.$ The morphism of fibrations $\varphi : E_0 \to E_1$ induces a morphism of fibrations $\varphi : \theta^* E_0 \to \theta^*E_1$ and we denote $$\theta^* \tilde{\varphi} : C_*(X,\Xi_1, C_*(\theta^*F_0,f_0)) \to C_*(X,\Xi_1, C_*(\theta^*F_1,f_1))$$ the associated chain map.

With such a set of DG Morse-Novikov data, we constructed in Section \ref{section : A Fibration Theorem for DG Morse-Novikov theory} a family of chains $$\left\{ m_x \in C_{|x|}(\mathcal{P}_{\star \to X} X), \ x \in \Crit(\alpha)\right\}$$ 
and we defined $\Psi_i : C_*(X,\Xi, C_*(F_i,f_i)) \to C_*(X,\Xi,C_*(\theta^*F_i, f'_i)) \overset{\psi_i}{\to} C_*(\theta^*E_i,f'_i) \to C_*(E_i, f_i),$ where $$\psi_i(\sigma \otimes x) = \Phi'_i(\sigma \otimes m_x).$$

The map $$v : C_*(X,\Xi_1,C_*(\theta^*F_0,f'_0)) \to C_{*+1}(\theta^*E_1,f'_1)$$ defined by $$v(\sigma \otimes x) = \sum_{\substack{n \\ z_1, \dots, z_{n-1}}} \pm \phi_{n+1}(\sigma \otimes m_{x,z_1} \otimes \dots \otimes m_{z_{n-2},z_{n-1}} \otimes m_{z_{n-1}})$$ is a chain homotopy between $\psi_1 \circ \theta^*\tilde{\varphi} $ and $\varphi_* \circ \psi_0$. See the proof of \cite[Theorem 5.8 ii)]{Rie24} for more details on the signs and for the computations.

Now, let $\Xi_0$ be any set of DG Morse-Novikov data. It is a direct consequence of \cite[Proposition 5.19]{Rie24} that $\tilde{\varphi}$ commutes with the continuation maps up to chain homotopy. It follows that the diagram

$$\xymatrix{
C_*(X,\Xi_0, C_*(F_0, f_0)) \ar[r]^{\tilde{\varphi}} \ar@/_4.5pc/[dd]_{\Psi_{E_0}} \ar[d]^{\Psi_{01}} & C_*(X, \Xi_0, C_*(F_1, f_1)) \ar[d]^{\Psi_{01}} \ar@/^4.5pc/[dd]^{\Psi_{E_1}} \\
C_*(X,\Xi_1, C_*(F_0, f_0)) \ar[r]^{\tilde{\varphi}} \ar[d]^{\Psi_{E_0}} & C_*(X, \Xi_1, C_*(F_1, f_1)) \ar[d]^{\Psi_{E_1}}\\
C_*(E_0,f_0) \ar[r]^{\varphi_*} & C_*(E_1,f_1)
}$$ commutes, and this concludes the proof.
\end{proof}

In particular, if $E_0=E_1$ but are endowed with different lifting functions, the identity $E_0 \to E_1$ induces a chain homotopy equivalence $\tilde{\Id} : C_*(X,\Xi_0,C_*(F_0,f_0)) \to C_*(X,\Xi_0,C_*(F_1,f_1))$. 

\begin{cor}\label{cor : independance of the transitive lifting function}
    If $\fibration$ is a fibration such that $\pi^*u = 0$ and $\Xi$ is a set of DG Morse-Novikov data, the complex $C_*(X,\Xi, C_*(F,u))$ does not depend, up to chain homotopy equivalence, on the choice of transitive lifting function $\Phi : E \ftimes{\pi}{\ev_0} \mathcal{P}X \to X$.
\end{cor}
\begin{flushright}
    $\blacksquare$
\end{flushright}

The following lemma states the compatibility between morphisms of fibrations and direct and shriek maps. It is a extension of \cite[Proposition 5.23]{Rie24}.

\begin{lemme}\label{prop : morphisme Ai commute avec direct et shriek MNDG} 
    Let $h: X^n \to Y^m$ be a continuous map. Let $v \in H^1(Y,\R)$ and let $G_0 \hookrightarrow E_0 \overset{\pi_0}{\to} Y$, $G_1 \hookrightarrow E_1 \overset{\pi_1}{\to} Y$ be two fibrations over $Y$ such that $\pi_0^* v= 0$, $\pi_1^*v = 0$ and $G_0,G_1$ are locally path-connected. Let $\varphi : E_0 \to E_1$ be a morphism of fibrations.

    Then the following diagrams commute

{
\begin{minipage}{7cm}
    $$
    \xymatrix{
    H_*(X, h^*C_*(G_0,v)) \ar[r]^-{h^*\Tilde{\varphi}} \ar[d]_{h_*} & H_*(X,h^*C_*(G_1,v)) \ar[d]^{h_*}\\
    H_*(Y,C_*(G_0,v)) \ar[r]_{\Tilde{\varphi}} & H_*(Y, C_*(G_1,v))
    }
    $$
\end{minipage}
\hspace{1em}
\begin{minipage}{7cm}
   $$
    \xymatrix{
    H_{*-m+n}(X, h^*C_*(G_0,v)) \ar[r]^-{h^*\Tilde{\varphi}}  & H_{*-m+n}(X,h^*C_*(G_1,v)) \\
    H_*(Y,C_*(G_0,v)) \ar[u]^{h_!} \ar[r]_{\Tilde{\varphi}} & H_*(Y, C_*(G_1,v)) \ar[u]_{h_!}.
    }
    $$
\end{minipage}
}

\end{lemme}

\subsection{The cross-product \texorpdfstring{$K$}{K}. Proof of Theorem \ref{theorem F}}\label{subsection : The cross product K}

Let $F \hookrightarrow E \overset{\pi_X}{\to} X$ be a fibration, let $u \in H^1(X,\R)$ such that $\pi_X^*u= 0 \in H^1(E,\R)$ and let $f : E \to \R$ be a primitive of $\pi_X^*\alpha$ (in the sense of Definition \ref{def : primitive of pi alpha}), where $\alpha \in u$.
Let $(Y,\star_Y)$ be another pointed, oriented, closed, and connected manifold and let $G \hookrightarrow E_Y \overset{\pi_Y}{\to} Y$ be a fibration. Let $v \in H^1(Y,\R)$ such that $\pi_Y^*v = 0 \in H^1(E_Y,\R)$ and let $g : E_Y \to \R$ be a primitive of $\pi_Y^*\beta$ where $\beta \in v$. 

The operations $$
    \begin{array}{ccccc}
        F \times G \times \Omega(X \times Y) & \to & F \times \Omega X \times G \times \Omega Y & \to & F \times G  \\
        (\alpha, \beta , \gamma) &\mapsto & (\alpha,\gamma_X, \beta, \gamma_Y) & \mapsto & (\alpha \cdot \gamma_X, \beta \cdot \gamma_Y).
    \end{array} $$ give rise to a module structure $$C_*(F \times G) \otimes C_*(\Omega(X \times Y)) \overset{\EZ}{\to} C_*(F \times G \times \Omega(X \times Y)) \to C_*(F \times \Omega X \times G \times \Omega Y) \to C_*(F \times G)$$ that can be extended by linearity to a module structure $$C_*(F \times G ,f + g) \otimes C_*(\Omega(X \times Y), u+v) \to C_*(F\times G ,f + g).$$

    For any pair of topological spaces $A$ and $B$, the \textbf{Eilenberg-Zilber} map for cubical chains  $\EZ : C_*(A) \otimes C_*(B) \to C_*(A \times B)$ is defined by  $$ \forall \alpha \in C^0([0,1]^k,A), \forall \beta \in C^0([0,1]^{\ell}, B), \  \EZ((s\mapsto \alpha(s)) \otimes (t \mapsto \beta(t))) = (s,t) \mapsto (\alpha(s), \beta(t)).$$

    Let $\Xi_X = (\alpha, \xi_X, s^g_{x,x'}, o_X, \mathcal{Y}_X,\theta_X)$ and $\Xi_Y= (\beta, \xi_Y, s^h_{y,y'}, o_Y, \mathcal{Y}_Y,\theta_Y)$ be sets of DG Morse-Novikov data for $X$ and $Y$ respectively and denote $\{m_{x,x'}^X\}$ and $\{m^Y_{y,y'}\}$ the associated twisting cocycles.

 From $\Xi_X$ and $\Xi_Y$, we define a set of DG Morse-Novikov data $\Xi_{X \times Y}$ on $X \times Y$ by:\\
 
 \begin{itemize}
     \item $\eta_{(x,y)} = \alpha_x + \beta_y$. Note that $\eta$ is a Morse 1-form on $X \times Y$ that satisfies $|(x,y)| = |x| + |y| $.\\
    \item $\xi(x,y) = (\xi_X(x), \xi_Y(y))$ is a pseudo-gradient associated with $\eta$.\\
   \item There is a canonical identification $\overline{W^u_\eta}(x,y) \simeq \overline{W^u_{\alpha}}(x) \times \overline{W^u_{\beta}}(y)$. We therefore use the orientation $\Or \ \wbi{u}{\eta}{x,y} = \left( \Or \ \wbi{u}{\alpha}{x} ,  \Or \ \wbi{u}{\beta}{y} \right) $.\\
    \item The tree $\mathcal{Y} = (\mathcal{Y}_X, \mathcal{Y}_Y)$.\\
   \item The homotopy inverse $\theta = (\theta_X, \theta_Y) : (X \times Y)/\mathcal{Y} \to X \times Y  $ of the canonical projection $p : X \times Y \to (X \times Y)/\mathcal{Y}$.
 \end{itemize}

It remains to define the representing chain system $$s^{(g,h)}_{(x,y),(x',y')} \in C_{|x|+|y|-|x'|-|y'|-1}(\trajbi{g}{x,x'} \times \trajbi{h}{y,y'}).$$

For any critical points $x,x' \in \Crit(\alpha)$, $y,y' \in \Crit(\beta)$ and $(g,h) \in \pi_1(X \times Y)$, there is an identification of the parametrized spaces of trajectories

$$\functrajb{(g,h)}{(x,y),(x',y')} = \functrajb{g}{x,x'} \times \functrajb{h}{y,y'}$$

and there exist a projection and a section 

$$
\trajbi{(g,h)}{(x,y),(x',y')} \underset{\pi}{\overset{i}{\leftrightarrows}} \trajbi{g}{x,x'} \times \trajbi{h}{y,y'}.
$$
The projection can be written $$\pi([a,b]_{X\times Y}) = (\pi_X([a,b]_{X\times Y}), \pi_Y([a,b]_{X\times Y})) = ([a]_X , [b]_Y) $$ if $(a,b) \in \functrajb{(g,h)}{(x,y),(x',y')} $ where $[\cdot]$ represents the equivalence class given by the $\R$-action.

Let $U,V$ be neighborhoods of $\Crit(\alpha)$ and $\Crit(\beta)$ respectively such that $\alpha = d\eta_X$ in $U$ and $\beta = d\eta_Y$. Choose $\epsilon>0$ small enough such that for any critical points $x \in \Crit(\alpha)$, $y \in \Crit(\beta)$, $\eta_X(x) - \epsilon \subset \eta_X(U)$ and $\eta_Y(y) - \epsilon \subset \eta_Y(V)$. 
Given $\lambda_X \in \trajbi{g}{x,x'}$ and $\lambda_Y \in \trajbi{h}{y,y'}$, denote $$i_X(\lambda_X) = \lambda_X \cap \eta_X^{-1}(\eta_X(x) - \epsilon) \in \functrajb{g}{x,x'} , \textup{ and } i_Y(\lambda_Y) = \lambda_Y \cap \eta_Y^{-1}(\eta_Y(x) - \epsilon) \in \functrajb{h}{y,y'}.$$ The section $i :  \trajbi{g}{x,x'} \times \trajbi{h}{y,y'} \to \trajbi{(g,h)}{(x,y),(x',y')}$ is then defined by $$i(\lambda_X,\lambda_X) = [i_X(\lambda_X), i_Y(\lambda_Y)].$$

\begin{lemme}
    Let $$\{s^g_{x,x'} \in C_{|x|-|x'|-1}(\trajbi{g}{x,x'}), \ x,x' \in \Crit(\alpha), \ g \in \pi_1(X)\}$$ and 
    $$\{s^h_{y,y'} \in C_{|y|-|y'|-1}(\trajbi{h}{y,y'}), \ y,y' \in \Crit(\beta), \ h \in \pi_1(Y)\}$$ be representing chain systems. There exists a representing chain system $$\{s^{(g,h)}_{(x,y),(x',y')} \in C_{|x|+|y|-|x'|-|y'|-1}(\trajbi{(g,h)}{(x,y),(x',y)})\}$$ such that: 

   $$ \left\{\begin{array}{lc}
        s^{(g,h)}_{(x,y),(x',y')} = (-1)^{|x|(|y|-|y'|)} (\{x\},s^h_{y,y'}) & \textup{if } x=x' \textup{ and } g=1,\\
        s^{(g,h)}_{(x,y),(x',y')} = (s^g_{x,x'}, \{y\}) & \textup{if } y = y' \textup{ and } h=1,\\
        \pi^*s^{(g,h)}_{(x,y),(x',y')} = 0 & \textup{otherwise}.
    \end{array} \right.$$
\end{lemme}
\begin{proof}
    $\bullet$ If $x=x'$ or $y=y'$ we can just choose $$s^{(g,h)}_{(x,y),(x,y')} = \left\{ \begin{array}{cc}
        (-1)^{|x|(|y|-|y'|)} (\{x\}, s^h_{y,y'}) & \textup{ if } g=1, \\
        0 & \textup{ otherwise,} 
    \end{array}\right. $$
    and
    $$s^{(g,h)}_{(x,y),(x',y)} = \left\{ \begin{array}{cc}
        (s^g_{x,x'}, \{y\})  & \textup{ if } h=1, \\
         0 & \textup{ otherwise,} 
    \end{array} \right.$$ and complete by induction as in Proposition \ref{prop: representing chain system} in order to obtain a family of chains $$\left\{s^{(g,h)}_{(x,y),(x',y')} \in C_*\left(\trajbi{(g,h)}{(x,y),(x',y')}\right), \ x = x' \textup{ or } y=y'\right\}$$ that satisfies the first two conditions. We just have to check that $$\partial s^{(g,h)}_{(x,y),(x,y')} = \sum_{\substack{w \in \Crit(\beta) \\ h'\cdot h'' = h}} (-1)^{|y|-|w|} s^{(g,h')}_{(x,y),(x,w)} \times s^{(g,h'')}_{(x,w), (x,y')} $$ and $$\partial s^{(g,h)}_{(x,y),(x',y)} =  \sum_{\substack{z \in \Crit(\alpha) \\ g' \cdot g''=g}} (-1)^{|x|-|z|} s^{(g',h)}_{(x,y),(z,y)} \times s^{(g'',h)}_{(z,y), (x',y)}.$$

    We only establish the first equality since the second one is analogous. If $g \neq 1$ then both sides are equal to $0$. If $g=1,$

    \begin{align*}
        \partial s^{(1,h)}_{(x,y),(x,y')} &= (-1)^{|x|(|y|-|y'|)} (\{x\}, \partial s^h_{y,y'})\\
        &= (-1)^{|x|(|y|-|y'|)} \sum_{\substack{w \\ h' \cdot h''= h}} (-1)^{|y|-|w|} (\{x\}, s^{h'}_{y,w} \times s^{h''}_{w,y'})\\
        &= \sum_{\substack{w \\ h' \cdot h''= h}} (-1)^{|y| - |w|} (-1)^{|x|(|y|-|w|)} (-1)^{|x|(|w|-|y'|)} (\{x\}, s^{h'}_{y,w}) \times (\{x\}, s^{h''}_{w,y'})\\
        &= \sum_{\substack{w \\ h' \cdot h''= h}} (-1)^{|y| - |w|} s^{h'}_{(x,y),(x,w)} \times s^{h''}_{(x,w),(x,y')}.
    \end{align*}

\noindent $\bullet$ If $x\neq x'$ and $y \neq y'$, we build $\{s^{(g,h)}_{(x,y),(x',y')}\}$ by induction on $|x|+|y|-|x'|-|y'| = \ell$. If $x \neq x'$, $y \neq y'$ and $\ell = 2$, then $|x|-|x'| = 1$ and $|y|-|y'| = 1$ or $\trajbi{(g,h)}{(x,y),(x',y')}$ is empty. Any representative $s^{(g,h)}_{(x,y),(x',y')} \in C_1(\trajbi{(g,h)}{(x,y),(x',y')})$ of the fundamental class satisfies the third condition since $\pi_{*}\left(s^{(g,h)}_{(x,y),(x',y')}\right) \in C_1(\trajbi{g}{x,x'} \times \trajbi{h}{y,y'}) $ is a $1$-chain in a $0$-dimensional space and is thus constant and degenerate.\\

Suppose that a representing chain system $s^{(g,h)}_{(a,b),(a',b')}$ respecting the three conditions has been constructed for every $|a|+|b|-|a'|-|b'| \leq \ell$ and $(g,h) \in \pi_1(X \times Y)$. Let $|x|+|y|-|x'|-|y'| = \ell +1$ such that $x \neq x'$ and $y \neq y'$, and let $(g,h) \in \pi_1(X \times Y)$. Let $\{t^{(g,h)}_{(x,y),(x',y')}\}$ be a representative of the fundamental class of $\trajbi{(g,h)}{(x,y),(x',y')}$ which satisfies \eqref{eq: système representatif}. Then $\pi_*\left(t^{(g,h)}_{(x,y),(x',y')}\right)\in C_{\ell}(\trajbi{g}{x,x'} \times \trajbi{h}{y,y'})$ is a cycle. Indeed, using the induction hypothesis

\begin{align*}
    \partial \pi_*(t^{(g,h)}_{(x,y),(x',y')}) &= \sum_{\substack{(z,w) \\ g' \cdot g''=g \\ h' \cdot h''=h}}  (-1)^{|x|+|y|-|z|-|w|} \pi_*s^{(g',h')}_{(x,y),(z,w)} \times \pi_*s^{(g'',h'')}_{(z,w),(x',y')}\\
    &= (-1)^{|y|-|y'|} s^{(1,h)}_{(x,y),(x,y')} \times s^{(g,1)}_{(x,y'),(x',y')} + (-1)^{|x|-|x'|} s^{(g,1)}_{(x,y),(x',y)} \times s^{(1,h)}_{(x',y),(x',y')}\\
    &= (-1)^{|x|+|x'|(1 + |y|-|y'|) +1} \left((s^g_{x,x'},s^h_{y,y'}) -  (s^g_{x,x'},s^h_{y,y'})\right) = 0.
\end{align*}

Since $\trajbi{g}{x,x'} \times \trajbi{h}{y,y'}$ is a manifold of dimension $\ell-1$ every $\ell$-cycle is a boundary. Hence, there exists $b \in C_{\ell+1}(\trajbi{g}{x,x'} \times \trajbi{h}{y,y'})$ such that $\partial b = \pi_*(t^{(g,h)}_{(x,y),(x',y')})$. We then define $s^{(g,h)}_{(x,y),(x',y')} = t^{(g,h)}_{(x,y),(x',y')} - \partial i_*(b) $, which satisfies the third condition. The resulting representing chain system $\{s^{(g,h)}_{(x,y),(x',y')}\}$ satisfies all the conditions and the proof is concluded.
    
\end{proof}

The family of evaluation maps $q_{(x,y), (x',y')} : \trajbi{(g,h)}{(x,y),(x',y)} \to \Omega^g X \times \Omega^h Y$ defined by $$ q_{(x,y), (x',y')}(\lambda) = \left(q_{x,x'}(\pi_X(\lambda)), q_{y,y'}(\pi_Y(\lambda) ) \right)$$ gives rise to the \textbf{Künneth twisting cocycle}

$$\left\{ m^{K,(g,h)}_{(x,y),(x',y')} \in C_{|x|+|y|-|x'|-|y'|-1}(\Omega^g X \times \Omega^h Y), \ x\in \Crit(\alpha), y \in \Crit(\beta), (g,h) \in \pi_1(X \times Y)  \right\}$$

defined by
$$m^{K,(g,h)}_{(x,y),(x',y')} = \left\{ \begin{array}{lc}
   (-1)^{|x|(|y|-|y'|)}(\star, m^h_{y,y'})  & \textup{if } x=x' \textup{ and } g=1,  \\
    (m^g_{x,x'}, \star) & \textup{if } y=y' \textup{ and } h=1,\\
    0 & \textup{otherwise.}
\end{array} \right.$$

Equivalently, we can define 

$$\left\{m^{K}_{(x,y),(x',y')} \in C_{|x|+|y|-|x'|-|y'|-1}(\Omega X \times \Omega Y, u+v), x \in \Crit(\alpha), y \in \Crit(\beta) \right\}$$

by 

$$m^{K}_{(x,y),(x',y')} = \sum m^{K,(g,h)}_{(x,y),(x',y')}  = \left\{ \begin{array}{lc}
   (-1)^{|x|(|y|-|y'|)}(\star, m_{y,y'})  & \textup{if } x=x',  \\
    (m_{x,x'}, \star) & \textup{if } y=y',\\
    0 & \textup{otherwise.}
\end{array} \right.$$

Although this twisting cocycle is not the Barraud-Cornea twisting cocycle $m^0_{(x,y),(x',y')}$ associated with the set of DG Morse-Novikov data $\Xi_{X \times Y}$, it actually computes the same homology. More precisely, we have the following result:

\begin{prop}\label{prop : dg Kunneth m^0 and m^1 quasi-iso MNDG}(see \cite[Proposition 6.5]{Rie24})
    For any right DG-module $\mathcal{H}$ over $C_*(\Omega X \times \Omega Y, u+v)$, there exists a chain homotopy equivalence $$C_*(X \times Y, m^0_{(x,y),(x',y')},\mathcal{H}) \underset{\Psi_{0K}}{\overset{\Psi_{K0}}{\leftrightarrows}} C_*(X \times Y, m^K_{(x,y),(x',y')}, \mathcal{H}).$$ 
\end{prop}
\begin{flushright}
    $\blacksquare$
\end{flushright}

It follows that \cite[Theorem 6.14]{Rie24} can be adapted in the DG Morse-Novikov setting and the cross products $K^{alg}$ and $K$ benefit from the same properties as those given in \cite[Section 6.3]{Rie24}.

\begin{thm}\label{thm : Morse-Novikov Künneth formula}
Let $\Xi_X$, $\Xi_Y$ be sets of DG Morse data on $X$ and $Y$ with respective Morse 1-forms $\alpha \in \Omega^1(X)$ and $\beta \in \Omega^1(Y)$. Let $u=[\alpha] \in H^1(X,\R)$ and $v = [\beta] \in H^1(Y,\R)$.

i) The Künneth twisting cocycle $\{m^{K,(g,h)}_{z,w} \in C_{|z|-|w|-1}(\Omega^g X \times \Omega^h Y), \ z,w \in \Crit(\alpha+\beta)\}$ associated with $\Xi_X$ and $\Xi_Y$ computes the same homology as the Barraud-Cornea cocycle associated with the set of DG data $\Xi_{X \times Y}$ constructed previously from $\Xi_X$ and $\Xi_Y$. \\

ii) Let $\F^u$ and $\G^v$ be DG modules over $C_*(\Omega X,u)$ and  $C_*(\Omega Y,v)$. Then, $$\deffct{K^{alg}}{C_*(X,\Xi_X, \F^u) \otimes_{\Z} C_*(Y, \Xi_Y, \G^v)}{C_*(X \times Y, m^K_{z,w}, \F^u \otimes_{\Z} \G^v)}{(\alpha \otimes x) \otimes (\beta \otimes y)}{(-1)^{|\beta||x|}(\alpha \otimes \beta) \otimes (x,y)}$$ is an isomorphism of complexes.\\

iii) Let $F \hookrightarrow E_X \overset{\pi_X}{\to} X$ and $G\hookrightarrow E_Y \overset{\pi_Y}{\to} Y$ be fibrations such that $\pi_X^*u = 0$, $\pi^*_Yv=0$ and $F,G$ are locally path-connected. Then $$\deffct{K}{C_*(X, \Xi_X, C_*(F,u)) \otimes_{\Z} C_*(Y,\Xi_Y, C_*(G,v))}{C_*(X \times Y, m^K_{z,w}, C_*(F \times G,u+v))}{(\alpha \otimes x) \otimes (\beta \otimes y)}{(-1)^{|\beta||x|}(\alpha , \beta) \otimes (x,y)}$$ is a quasi-isomorphism of complexes.
\end{thm}

\begin{flushright}
    $\blacksquare$
\end{flushright}

We here state two lemmas that are direct extension of \cite[Lemma 6.17]{Rie24} and \cite[Lemma 6.22]{Rie24} that state the compatibility of the cross products with morphisms of fibrations and direct and shriek maps.

The next lemma is stated using $K^{alg}$ but also applies for $K$ if $\F^u = C_*(F,u)$ and $\G^v = C_*(G,v)$ where $F$ and $G$ are locally path-connected fibers of a fibration $E_X \overset{\pi_X}{\to} X$ and $E_Y \overset{\pi_Y}{\to} Y$ respectively such that $\pi_X^*u = 0$ and $\pi_Y^*v= 0.$

\begin{lemme}\label{lemme : commutativité K avec direct et shriek M-N}
Let $\varphi : Y^{n_Y} \to X^{n_X}$ and $\psi: Z^{n_Z} \to W^{n_W} $ be two continuous maps. Let $u \in H^1(X,\R)$, $v \in H^1(W,\R)$ and let $\F^u$, $\G^v$ be a $C_*(\Omega X,u)$ right DG module and a $C_*(\Omega W, v)$ right DG module respectively. Let $\Xi_X, \Xi_Y, \Xi_Z, \Xi_W$ be sets of DG Morse-Novikov data for $X,Y,Z,W$ respectively. Then the following diagram for the direct maps commutes at the chain level 

\[
\xymatrix{
C_*(Y, \Xi_Y, \varphi^*\F^u) \otimes C_*(Z, \Xi_Z, \psi^*\G^v) \ar[r]^-{K} \ar[d]_{\varphi_* \otimes \psi_*} & C_*(Y \times Z, \Xi_{Y \times Z}, \varphi^*\F^u \otimes \psi^*\G^v) \ar[d]^{(\varphi \times \psi)_*} \\
C_*(X, \Xi_X, \F^u) \otimes C_*(W, \Xi_W, \G^v) \ar[r]_{K} & C_*(X \times W, \Xi_{X \times W}, \F^u \otimes \G^v)
}
\]

and the following diagram for the shriek maps commutes at the chain level up to sign :

\[
\xymatrix{
C_{i+n_Y-n_X}(Y, \Xi_Y, \varphi^*\F^u) \otimes C_{j+n_Z - n_W}(Z, \Xi_Z, \psi^*\G^v) \ar[r]^-{K}  & C_{i+j +n_Y + n_Z -n_X -n_W}(Y \times Z, \Xi_{Y \times Z}, \varphi^*\F^u \otimes \psi^*\G^v)  \\
C_i(X, \Xi_X, \F^u) \otimes C_j(W, \Xi_W, \G^v) \ar[u]^{\varphi_! \otimes \psi_!} \ar[r]_{K} & C_{i+j}(X \times W, \Xi_{X \times W}, \F^u \otimes \G^v) \ar[u]_{(\varphi\times \psi)_!}
}
\]

More precisely, $$K( \varphi_!(\alpha \otimes x) \otimes \psi_!(\beta \otimes w)) = (-1)^{(n_Y-n_X)(n_W  -|\beta|-|w|) + (n_Z-n_W)(|\alpha| + |x|)} (\varphi_! \times \psi_!)K(\alpha \otimes x \otimes \beta \otimes w).$$

\end{lemme}

\begin{flushright}
    $\blacksquare$
\end{flushright}

\begin{lemme}\label{lemme : induced maps on a Cartesian product commutes with K MNDG}
    Let $F_i \hookrightarrow E^X_i \overset{\pi^X_i}{\to} X$ and $G_i \hookrightarrow E^Y_i \overset{\pi^Y_i}{\to} Y$ be fibrations over two pointed, oriented, closed, and connected manifolds $(X,\star)$ and $(Y,\star_Y)$ with respective transitive lifting functions $\Phi^X_i$ and $\Phi^Y_i$ for $i \in \{0,1\}$. Let $u\in H^1(X,\R)$ and $v \in H^1(Y,\R)$. Assume that $\pi^{X,*}_i u = 0$ and $\pi^{Y,*}_i u = 0$ for $i \in \{0,1\}$ and assume that $F_0,F_1,G_0,G_1$ are locally path-connected.
    Let $\varphi : E^X_0 \to E^X_1$ and $\psi: E^Y_0 \to E^Y_1$ be morphisms of fibrations over $X$ and $Y$ respectively. 
    Endow the fibrations $E^X_i \times E^Y_i \to X \times Y$ with the transitive lifting functions $$\Phi_i^{X \times Y} = (\Phi_i^X,\Phi_i^Y) \ \textup{for } i\in \{0,1\}.$$
     Then $\varphi \times \psi : E^X_0 \times E^Y_0 \to  E^X_1 \times E^Y_1$ is a morphism of fibrations and the following diagram commutes
$$\xymatrix{
H_*(X,C_*(F_0,u)) \otimes H_*(Y, C_*(G_0,v)) \ar[d]^{\tilde{\varphi} \otimes \tilde{\psi}} \ar[r]^K & H_*(X \times Y, C_*(F_0 \times G_0, u+v)) \ar[d]^{\widetilde{\varphi \times \psi}} \\
H_*(X,C_*(F_1,u)) \otimes H_*(Y, C_*(G_1,v)) \ar[r]^{K} & H_*(X \times Y, C_*(F_1 \times G_1,u+v)).
}$$

\end{lemme}

\begin{flushright}
    $\blacksquare$
\end{flushright}

\subsection{Product and properties}

We now accomplished all the preliminary work in order to state one of the main results of this article.

The degree $-n$ product $\CS^u : H_*(X,\F^u)^{\otimes 2} \to H_*(X,\F^u)$ is defined by the composition 

$$\begin{array}{ccl}
    \CS^u :   H_i(X,\F^u) \otimes H_j(X,\F^u) 
    &\overset{(-1)^{n(n-j)}K}{\longrightarrow} & H_{i+j}(X^2, C_*(F^2,u))\\
    &\overset{\Delta_!}{\longrightarrow} &  H_{i+j-n}(X, \Delta^*C_*(F^2,u))\\
    &\overset{\tilde{m}}{\longrightarrow} & H_{i+j-n}(X,\F^u).
\end{array}$$

\begin{rem}
    If $\pi: E \to X$ admits a section $s : X \to E$,  then $\pi^*u=0 \Leftrightarrow s^*\pi^*u = u = 0$. In particular, the DG Morse-Novikov extension of the Neutral Element property in \cite[Theorem 7.1]{Rie24} could only be stated when $u=0$, which corresponds to the Morse case. 
\end{rem}

The proofs of the properties in Theorem \ref{thm : Chas-Sullivan product in Morse-Novikov} are the same as in \cite[Section 7]{Rie24} using the construction of shriek and direct maps in Section \ref{Section : Functoriality in the DG Morse-Novikov theory}, the construction of the morphism induced by a morphism of fibrations in Section \ref{subsection : Morphism of fibrations} and the construction of the cross product $K$ in Section \ref{subsection : The cross product K}. We stated properties of compatibility between these three types of maps (Lemma \ref{prop : morphisme Ai commute avec direct et shriek MNDG}, Lemma \ref{lemme : commutativité K avec direct et shriek M-N} and Lemma \ref{lemme : induced maps on a Cartesian product commutes with K MNDG}) which are integral to the proofs.

\section{Appendix: Homology and projective limit, algebraic properties}\label{Appendix : Homology and projective limit, algebraic properties}

In this article, we are interested in projective limits of chain complexes and homology groups filtered by $\Z_{\leq 0}$ where the parameter converges towards $-\infty$ and will use notation conventions accordingly. We will lay out the definition and basic properties of the projective limit that we will use throughout this article such as, the Universal Property of projective limits or a vanishing condition on the derived functor $\limun$. In a more topological note, we also consider $\R$-filtered topological spaces $(Y,Y_c)$ such that $Y_c \subset Y_{c'}$ if $c\leq c'$ and the projective limit $$\limproj C_*(Y,Y_c)$$ since these complexes naturally appear in our work. We give a chain homotopy equivalence criterion for such limits.  We will use the Appendix of the book \cite{Mas78} as a reference.

\subsection{Definitions and basic properties}

\begin{defi}
   
    $\bullet$ \textbf{A projective system of chain complexes} is a family $(A_i)_{i \in \Z_{\leq 0}}$ of complexes endowed with a family of morphisms $\pi_{i,j}: A_i \to A_j$ for $i\leq j \in \Z_{\leq 0}$ such that $\pi_{i,i} = \Id$ for all $i \in \Z_{\leq 0}$ and $\pi_{j,k}\circ \pi_{i,j} = \pi_{i,k}$ for any $i\leq j \leq k.$

    $\bullet$ Let $\left((A_i)_{i \in \Z_{\leq 0}}, \pi_{i,j}\right)$ be a projective system and define $$d: \prod_{i \in \Z_{\leq 0}} A_i \to \prod_{i \in \Z_{\leq 0}} A_i,\ d((a_i)_{i \in \Z_{\leq 0}}) = (a_i- \pi_{i-1,i}(a_{i-1}))_{i \in \Z_{\leq 0}}.$$  The \textbf{projective limit} $A$ of the projective system  is the complex defined by 
    $$A:= \left\{ (\sigma_i)_{i \in \Z_{\leq 0}} \in  \prod_{i \in \Z_{\leq 0}} A_i, \ \pi_{i,j}(\sigma_i) = \sigma_j\right\} = \textup{Ker}(d)$$ endowed with the differential $$\partial (\sigma_i)_{i \in \Z_{\leq 0}} = (\partial \sigma_i)_{i \in \Z_{\leq 0}}.$$ The projective limit $A$ naturally comes equipped with the projections $\pi_j: A \to A_j$, $\pi_j((\sigma_i)_{i \in \Z_{\leq 0}}) = \sigma_j$ such that the following diagram commutes:

    $$\xymatrix{
    & A\ar[dr]^{\pi_j} \ar[dl]_{\pi_i} & \\
    A_i \ar[rr]^{\pi_{i,j}} & & A_j.
    }$$
We will often denote $\limproj A_i:= A$.    
\end{defi}

Fix for the rest of this section a projective system $(A_i)_{i \in \Z_{\leq 0}}$ and let $A = \limproj A_i$ be its projective limit.

The projective limit has the following universal property.

\begin{universalprop}\label{universal property projective limit}
    For any complex $B$ endowed with morphisms $p_i: B \to A_i$ for any $i \in \Z_{\leq 0}$ such that $\pi_{i,j} \circ p_i = p_j$, there exists a unique map $\psi: B \to A$, $\psi(b) = (p_i(b))_{i \in \Z_{\leq 0}} \in A$ such that the following diagram commutes

     $$\xymatrix{
     & B \ar@/^1.0pc/[ddr]^{p_j}  \ar@/_1.0pc/[ddl]_{p_i} \ar@{-->}[d]^{\psi} & \\
    & A\ar[dr]^{\pi_j} \ar[dl]_{\pi_i} & \\
    A_i \ar[rr]^{\pi_{i,j}} & & A_j.
    }$$
\end{universalprop}
\begin{flushright}
    $\blacksquare$
\end{flushright}

\begin{cor}\label{cor: fonctoriality limproj}
    Let $((B_i), \pi^B_{i,j})$ be a projective system with projective limit $B$. A \textbf{morphism of projective systems} $\{\psi_i : A_i \to B_i,\ i \in \Z_{\leq 0}\}$ (i.e.  $\pi^B_{i,j} \circ \psi_i = \psi_j \circ \pi^A_{i,j}$ for all $i,j \in \Z_{\leq 0}$) induces a morphism $\psi: A \to B$, $\psi((a_i)_{i \in \Z_{\leq 0}}) = (\psi_i(a_i))_{i \in \Z_{\leq 0}}$  between their projective limits. 
\end{cor}
\begin{flushright}
    $\blacksquare$
\end{flushright}

\subsection{The derived functor \texorpdfstring{$\limun$}{limun}}

Unlike for direct limits, the homology of a projective limit need not be isomorphic to the projective limit of the homology groups.

\begin{thm}\cite[Theorem A.14]{Mas78}
The failure of right-exactness of the functor $\limproj = \textup{Ker}(d)$ is measured by $\limun := \textup{Coker}(d)$. More precisely, given projective systems $(A_i)$, $(B_i)$, and $(C_i)$ such that there exists a short exact sequence
$$0 \to A_i \to B_i \to C_i \to 0,$$ 
where the arrows commute with the respective projections of these systems, then we have the exact sequence
$$0 \to \limproj A_n \to \limproj B_n \to \limproj C_n \to \limun A_n \to \limun B_n \to \limun C_n \to 0.$$
\end{thm}

However, there exists a vanishing condition for $\lim^1$, known as the \textbf{Mittag-Leffler condition}:

\begin{prop}\cite[Lemma A.17]{Mas78}
    If $((A_i)_{i \in \Z}, \pi_{i,j})$ satisfies the Mittag-Leffler condition 
    $$\forall i \in \Z, \ \exists j \leq i, \ \forall k \leq j, \ \Image(\pi_{k,i}) = \Image(\pi_{j,i}),$$
then $\limun A_i = 0.$ 
\end{prop}
\begin{flushright}
    $\blacksquare$
\end{flushright}

All the projective systems of complexes $((A_i),\pi_{i,j})$ considered in this article satisfy the condition that $\pi_{i,j}: A_i \to A_j$ is surjective for any $i \leq j \in \Z_{\leq 0}$. In particular, they all satisfy the Mittag-Leffler condition. We will give a independent proof of the vanishing of $\limun A_i$ in this particular case.

\begin{lemme}\label{lemme : projection surjective implique limun = 0}
    Assume that $\pi_{i,j}: A_i \to A_j$ is surjective for any $i \leq j \in \Z_{\leq 0}$. Then $\limun A_i = 0.$ 
\end{lemme}
\begin{proof}
    In this case, $$d: \prod_{i \in \Z_{\leq 0}} A_i \to \prod_{i \in \Z_{\leq 0}} A_i,\ d((a_i)_{i \in \Z_{\leq 0}}) = (a_i- \pi_{i-1,i}(a_{i-1}))_{i \in \Z_{\leq 0}}$$ is surjective and therefore $\textup{Coker}(d) = \limun A_i = 0$. Indeed, if $(b_i) \in \prod_{i \in \Z_{\leq 0}} A_i$, define inductively $$a_0 = b_0 \textup{ and } \ a_{i} \in \pi_{i,i+1}^{-1}( a_{i+1} - b_{i+1}) \textup{ for } i \leq -1. $$
\end{proof}

\begin{lemme}\label{lemme: limproj d'un quotient de sys proj}
    Let $((B_i)_{i \in\Z_{\leq 0}}, \pi^B_{i,j})$ be a subsystem of $((A_i)_{i \in\Z_{\leq 0}}, \pi^A_{i,j})$ i.e. $$\forall i\leq j , \ B_i \subset A_i \textup{ and } \pi^B_{i,j}(B_i) = \pi^A_{i,j}(B_i) \subset B_j.$$ Assume that $\pi^B_{i,j} = \pi^{A}_{i,j} \lvert_{B_i}$ is surjective for all $i\leq j$.  Then $\pi^A$ induces a projection on $\left(\faktor{A_i}{B_i}\right)_{i \in \Z_{\leq 0}}$ such that 
    $$\limproj \faktor{A_i}{B_i} \cong \faktor{\limproj A_i}{\limproj B_i}.$$

\end{lemme}

\begin{proof}

Let $i \leq j$. Since $\pi^A_{i,j}(B_i) \subset B_{j}$, the map 
$$A_i \overset{\pi^A_{i,j}}{\to} A_{j} \to \faktor{A_{j}}{B_{j}}$$ 
induces a projection $$\pi^{A/B}_{i,j} : \faktor{A_i}{B_i} \to \faktor{A_j}{B_j}$$ such that $$\xymatrix{
A_i \ar[d] \ar[r]^{\pi^A_{i,j}} & A_j \ar[d] \\
\faktor{A_i}{B_i} \ar[r]_{\pi^{A/B}_{i,j}} & \faktor{A_j}{B_j} 
}$$ commutes. We have the exact sequence:
$$0 \to B_i \to A_i \to \faktor{A_i}{B_i}\to 0,$$
where the maps commute with the projections. Applying $\limproj$, we obtain 
$$0 \to \limproj B_i \to \limproj A_i \to \limproj \faktor{A_i}{B_i} \to \limun B_i \to \limun A_i \to \limun \faktor{A_i}{B_i} \to 0.$$

Since the projection $\pi^B$ is assumed to be surjective, Lemma \ref{lemme : projection surjective implique limun = 0} implies that $\limun B_i = 0$, thereby concluding the proof of the lemma.
\end{proof}

\subsection{\texorpdfstring{$\R$}{R}-filtered complexes}

The projective systems encountered in this article are complexes modded out by a family of subcomplexes naturally filtered by $\R$. The next lemma will not only explain why the "Novikov completions" considered are indeed projective limits but also justify that we can unequivocally consider such systems as filtered by $\Z_{\leq 0}$ and therefore obtain a well-defined notion of projective limits for such projective systems filtered by $\R$ with the properties listed above.

\begin{lemme}\label{Lemme: limite projective d'un quotient de complexe}
    Let $C$ be a chain complex and $(C^n)_{n \in \Z_{\leq 0}}$ a sequence of subcomplexes such that  $$C^{n} \subset C^{m}, \ \text{ if } n<m.$$

    Consider the projective system $\left( C/C^n, \pi_{n,m}\right)$ where $\deffct{\pi_{n,m}}{\faktor{C}{C^{n}}}{\faktor{C}{C^{m}}}{\overline{\tau}^{n}}{\overline{\tau}^{m}}$ is the canonical projection for $n\leq m$. Its projective limit is

    $$\limproj \faktor{C}{C^{n}} = \faktor{\hat{C}}{\bigcap_n C^{n}},$$

    where we denoted $$\hat{C} = \left \{ \displaystyle \sum_{i\in \Z_{\leq 0}} a_i \tau_i,  \  \ a_i \in \Z, \tau_i \in C , \text{ and }  \forall n \in \Z_{\leq 0}, \ \#\{i, \ a_i \neq 0 \textup{ and }\tau_i \notin C^{n}\} < \infty \right \},$$
    and 
    $$\bigcap_n C^n = \left\{ \sum_{i \in \Z_{\leq 0}} a_i \tau_i, \ \forall i\in \Z_{\leq 0}, \tau_i \in \bigcap_n C^n \right\}.$$
    
\end{lemme}

\begin{proof}
    
    Define for $n \in \Z_{\leq 0}$, the chain maps 

    $$\deffct{\phi^{n}}{\hat{C}}{\faktor{C}{C^{n}}}{\displaystyle \sum_{i \geq 0} a_i\tau_i}{\displaystyle \overline{  \sum_{i \geq 0} a_i \tau_i}^{n} = \overline{ \sum_{\tau_i \notin C^{n}} a_i \tau_i}^{n}.}$$

    The following diagram is commutative for all $n \in \Z_{\leq 0}$:

    \[
    \xymatrix{
\hat{C} \ar[r]^{\phi^{n}} \ar[dr]_{\phi^{m}} & \faktor{C}{C^{n}} \ar[d]^{\pi_{n,m}}\\
 & \faktor{C}{C^{m}}
}
    \]

The Universal property of projective limits \ref{universal property projective limit} gives a chain map 
$$\deffct{\psi}{\hat{C}}{\displaystyle \lim_{\longleftarrow} \faktor{C}{C^{n}}}{\tau}{\left (\phi^{n}(\tau)\right )_{n \in \Z_{\leq 0}}.}$$

Remark that $$ \tau \in \textup{Ker}(\psi) \ \Longleftrightarrow \ \forall n \in \Z, \ \overline{\tau}^{n} = 0  \ \Longleftrightarrow \ \tau \in \displaystyle \bigcap_n C^{n}.$$

Therefore $\psi$ factors through an injective morphism $\tilde{\psi}: \faktor{\hat{C}}{\bigcap C^n} \to \limproj \faktor{C}{C^n}$.\\

It remains to prove that $\psi$ is surjective.
Let $(\overline{\tau_n}^{n})_{n \in \Z_{\leq 0}} \in \limproj \faktor{C}{C^{n}} = \left \{ (\overline{\tau_n}^{n})_{n \in \Z_{\leq 0}}, \  \pi_{n, m}(\overline{\tau_n}^{n}) = \overline{\tau_{m}}^{m}  \right \} $.

Fix, for all $n \in \Z_{\leq 0}$ a representative $\tau_n \in C$ of  $\overline{\tau_n}^{n}$.  For $n \leq m$, since   $\overline{\tau_n}^{m} = \pi_{n, m}(\overline{\tau_n}^{n}) = \overline{\tau_{m}}^{m},$ then $\tau_n - \tau_{m} \in C^{m}$. 
Let $n \in \Z_{\leq 0}$,
\begin{align*}
    \phi^{n}\left( \displaystyle \sum_{k \in \Z_{\leq 0}} \tau_{k} - \tau_{k+1} \right) &= \sum_k \overline{\tau_k - \tau_{k+1}}^{n}\\
    &= \sum_{k\geq n} \overline{\tau_k - \tau_{k+1}}^{n} + \sum_{k < n} \overline{\tau_k - \tau_{k+1}}^{n}\\
    &= \overline{\tau_n}^{n} + 0
\end{align*}

The second sum is zero since, if $k<n$, then $\tau_k - \tau_{k+1} \in C^{k+1} \subset C^n$. Therefore $$\psi \left( \displaystyle \sum_{n \in \Z} \tau_{n} - \tau_{n+1} \right) = (\overline{\tau_n}^{n})_{n \in \N}.$$

\end{proof}

\begin{cor}\label{cor : limproj with parameter in R}
    If $(C^c)_{c \in \R}$ is a chain complex filtered by $\R$ such that $C^c \subset C^{c'}$ if $c \leq c'$, then for any strictly increasing function $\varphi: \Z_{\leq 0} \to \R$ such that $\lim_{n \to - \infty}\varphi(n) = -\infty$, $$\limproj \faktor{C}{C^n} = \limproj \faktor{C}{C^{\varphi(n)}}.$$ We will therefore denote $\limproj C^c = \limproj C^n.$
\end{cor}
\begin{flushright}
    $\blacksquare$
\end{flushright}

\begin{rem}
    This corollary holds when we replace $\R$ by any cofinal set of $\R$.  
\end{rem}

Going back to topological considerations, we state two lemmas proving that a map $\varphi : Y \to Z$ between two $\R$-filtered topological spaces that suitably respects the filtrations will induce a map between the projective limits and give a condition for this map to induce a homotopy equivalence between the homology of the projective limits. 

\begin{lemme}\label{lemme: projection on proj sys induce induces iso}
    Let $A$ be a chain complex endowed with a family of subcomplexes $(A_c)_{c \in \R}$ such that $A_c \subset A_{c'}$ if $c<c'$. Consider the projective system $((A/A_c), \pi_{c,c'})$ where $\pi_{c,c'} : A/A_c \to A/A_{c'}$ is the canonical projection. Let $K >0$  and denote for any $c \in \R$, $B_c = A_{c+K}$, $\pi^B_{c,c'} = \pi_{c+K, c'+K}$. The projections $(\pi_{c,c+K})_{c \in \R}$ induce an identification of complexes $$\limproj A/A_c \cong \limproj A/B_c.$$
\end{lemme}

\begin{proof}
    Using Corollary \ref{cor: fonctoriality limproj}, we see that the projections $(\pi_{c,c+K})_{c \in \R}$ induce the morphism $\pi: \limproj A/A_c \to \limproj A/B_c$, \begin{align*}
        \pi((a_c)_{c \in \R}) &= (\pi_{c,c+K}(a_c))_c \\
        &= (a_{c+K})_c.
    \end{align*}

    We identify the sequence $(a_{c+K})_{c \in \R}$ and $(a_c)_{c \in \R}.$
\end{proof}

\begin{lemme}\label{lemme: limite proj homotopy equivalence} 
    Let $Y,Z$ be topological spaces each filtered respectively by $(Y_c)_{c\in \R}$ and $(Z_c)_{c\in \R}$ i.e. if $c\leq c'$, then $Y_c \subset Y_{c'}$ and $Z_c \subset Z_{c'}$. Consider the projective systems $$\left(C_*(Y, Y_{c}), \pi^Y_{c,c'}\right) \textup{ and }\left(C_*(Z, Z_{c}), \pi^Z_{c,c'}\right),$$ where $\pi^Y_{c,c'}: C_*(Y,Y_c) \to C_*(Y,Y_{c'})$ and  $\pi^Z_{c,c'}: C_*(Z,Z_c) \to C_*(Z,Z_{c'})$   are defined by the canonical projections for $c<c'$.
Let $\phi: Y \to Z$ and $\psi: Z \to Y$ be homotopy equivalences that are inverses of each other, and denote by $$F_Z : [0,1] \times Z \to Z \textup{ such that } \ F_Z(0,\cdot) = \phi \circ \psi, \ F_Z(1, \cdot) = \Id_Z$$ and $$F_Y : [0,1] \times Y \to Y, \textup{ such that } F_Y(0,\cdot) = \psi \circ \phi, \ F_Y(1, \cdot) = \Id_Y,$$ the associated homotopies. Assume that there exists a constant $K>0$ such that

$$\forall c \in \R, \ \phi(Y_c) \subset Z_{c+K},  \psi(Z_c) \subset Y_{c+K}, \ F_Y( \cdot ,Y_c) \subset Y_{c+2K} \textup{ and } F_Z(\cdot,Z_c) \subset Z_{c+2K}.$$

    Then $\phi_*$ and $\psi_*$ induce chain homotopy equivalences $$\xymatrix{
    \limproj C_*(Y,Y_c) \ar@<2pt>[r]^{\phi_*} &  \limproj C_*(Z,Z_c) \ar@<2pt>[l]^{\psi_*}}.$$ 

\end{lemme}

\begin{proof}

The map $\phi : Y \to Z$ induces a map $$\phi : \limproj C_*(Y,Y_c) \to \limproj C_*(Z,Z_{c+K}) = \limproj C_*(Z,Z_c),$$    and $\psi : Z \to Y$ induces $$\psi : \limproj C_*(Z,Z_c) \to \limproj C_*(Y, Y_{c+K}) = \limproj C_*(Y, Y_c).$$

Moreover, the homotopies  $F_Y: [0,1] \times (Y, Y_c) \to (Y,Y_{c+2K})$ and $F_Z: [0,1] \times (Z, Z_c) \to (Z,Z_{c+2K})$ are homotopies of pairs and therefore $$\psi_* \circ \phi_* : C_*(Y, Y_c) \to C_*(Y, Y_{c+2K}) \textup{ is chain homotopic to }  \Id_{Y,*} = \pi^{Y}_{c,c+2K,*}: C_*(Y, Y_c) \to C_*(Y, Y_{c+2K}),$$ and $$\phi_* \circ \psi_* : C_*(Z, Z_c) \to C_*(Z, Z_{c+2K}) \textup{ is chain homotopic to }  \Id_Z = \pi^{Z}_{c,c+2K,*}: C_*(Z, Z_c) \to C_*(Z, Z_{c+2K}).$$
Lemma \ref{lemme: projection on proj sys induce induces iso} concludes the proof.

\end{proof}

\setcounter{secnumdepth}{-1}

\bibliography{mybib}
\bibliographystyle{alpha}

\end{document}

%% file: Preambule.tex
\usepackage{overpic}
\usepackage{amsmath}
\usepackage{xfrac}
\usepackage{blindtext}
\usepackage{comment}
\usepackage{faktor}
\usepackage[all]{xy}
\usepackage[french, main=english]{babel}
\usepackage[utf8x]{inputenc}
\usepackage{float}
\usepackage{graphicx}
\usepackage[colorinlistoftodos]{todonotes}
\usepackage{url}
\usepackage{hyperref}
\usepackage{array}
\usepackage{tabularx}
\usepackage{setspace}
\usepackage[T1]{fontenc}
\usepackage{subfig}
\usepackage[margin=1in]{geometry} 
\usepackage{fancyhdr}
\usepackage{rotating}
\usepackage{amsfonts,amssymb,amsthm,scrextend}
\usepackage{pinlabel}
\usepackage[nottoc,numbib]{tocbibind}
\usepackage{appendix}
\usepackage{lscape}

\newcommand{\N}{\mathbb{N}}

\newcommand{\Z}{\mathbb{Z}}

\newcommand{\R}{\mathbb{R}}

\newcommand{\F}{\mathcal{F}}
\newcommand{\G}{\mathcal{G}}

\newcommand{\Ai}{\mathcal{A_{\infty}}}

\newcommand{\deffct}[5]{#1: \begin{array}{ccc}
    #2 & \to & #3  \\
     #4 & \mapsto & \displaystyle #5
\end{array}}
\newcommand{\trajb}[1]{\overline{\mathcal{L}}(#1)}
\newcommand{\trajbi}[2]{\overline{\mathcal{L}}_{#1}(#2)}
\newcommand{\traji}[2]{\mathcal{L}_{#1}(#2)}
\newcommand{\traj}[1]{\mathcal{L}(#1)}
\newcommand{\wb}[2]{\overline{W}^{#1}(#2)}
\newcommand{\wbi}[3]{\overline{W}^{#1}_{#2}(#3)}

\newcommand{\functrajb}[2]{\overline{\Mtraj}_{#1}(#2)}
\newcommand{\Mtraj}{\mathcal{M}}

\newcommand{\fibration}{E \overset{\pi}{\to} X}

\newcommand{\Crit}{\textup{Crit}}
\newcommand{\Or}{\textup{Or}}
\newcommand{\Id}{\textup{Id}}

\newcommand{\End}{\textup{End}}
\newcommand{\EZ}{\textup{EZ}}
\newcommand{\ev}{\textup{ev}}

\newcommand{\Tor}{\textup{Tor}}
\newcommand{\Ker}{\textup{Ker}}

\newcommand{\Image}{\textup{Im}}

\newcommand{\Y}{\mathcal{Y}}

\newcommand{\Xq}{X/\Y}

\newcommand{\ftimes}[2]{\hspace{0.1em} _{#1}\!\times_{#2}}

\newcommand{\limproj}{\displaystyle \lim_{\longleftarrow}}

\newcommand{\limun}{\textup{lim}^1}

\newcommand{\CS}{\textup{CS}_{DG}}

\renewcommand{\qed}{\hfill$\blacksquare$}

\renewenvironment{proof}{\begin{addmargin}[1em]{0em}\begin{newproof}}{\end{newproof}\end{addmargin}\qed}

\definecolor{darkgreen}{RGB}{0,168,0}

\newtheorem{thm}{Theorem}[section]
\newtheorem{notation}[thm]{Notation}
\newtheorem{defi}[thm]{Definition}
\newtheorem{rem}[thm]{Remark}
\newtheorem{prop}[thm]{Proposition}
\newtheorem{cor}[thm]{Corollary}
\newtheorem{lemme}[thm]{Lemma}

\newtheorem{universalprop}[thm]{Universal property}

\newtheorem{theorem}{Theorem}
\newtheorem{proposition}[theorem]{Proposition}